\documentclass[10pt, a4paper, reqno]{amsart}

\usepackage{amsmath,amsthm,amssymb,mathrsfs,graphicx,extpfeil}
\usepackage{epsfig}
\usepackage{indentfirst, latexsym, amssymb, enumerate,amsmath,graphicx}
\usepackage{float} 
\usepackage{colortbl}
\usepackage{epsfig,subfigure}
\usepackage{bm}
\usepackage{empheq}
\usepackage{xcolor}
\usepackage{bbm}
\usepackage{lipsum}
\usepackage{ragged2e}
\usepackage{tikz}

\makeatletter
\def\section{\@startsection{section}{1}%
  \z@{.7\linespacing\@plus\linespacing}{.5\linespacing}%
  {\normalfont\large\bfseries\centering}}
\renewenvironment{abstract}{%
  \ifx\maketitle\relax
    \ClassWarning{\@classname}{Abstract should precede
      \protect\maketitle\space in AMS document classes; reported}%
  \fi
  \global\setbox\abstractbox=\vtop \bgroup
    \normalfont\Small
    \list{}{\labelwidth\z@
      \leftmargin3pc \rightmargin\leftmargin
      \listparindent\normalparindent \itemindent\z@
      \parsep\z@ \@plus\p@
      
    }%
    \item[\hskip\labelsep\bfseries\abstractname.]%
}{%
  \endlist\egroup
  \ifx\@setabstract\relax \@setabstracta \fi
}
\def\@setabstract{\@setabstracta \global\let\@setabstract\relax}
\def\@setabstracta{%
  \ifvoid\abstractbox
  \else
    \skip@20\p@ \advance\skip@-\lastskip
    \advance\skip@-\baselineskip \vskip\skip@
    \box\abstractbox
    \prevdepth\z@ 
  \fi
}

\usepackage[sort,nocompress]{cite}

\renewcommand{\tocsection}[3]{%
  \indentlabel{\@ifnotempty{#2}{\bfseries\ignorespaces#1 #2\quad}}\bfseries#3}
\renewcommand{\tocsubsection}[3]{%
  \indentlabel{\@ifnotempty{#2}{\ignorespaces#1 #2\quad}}#3}

\usepackage{hyperref}
\hypersetup{colorlinks=true,linkcolor=teal,citecolor=blue,filecolor=black,urlcolor=blue,
} 
\usepackage[capitalise]{cleveref}

\newcommand\@dotsep{4.5}
\def\@tocline#1#2#3#4#5#6#7{\relax
  \ifnum #1>\c@tocdepth 
  \else
    \par \addpenalty\@secpenalty\addvspace{#2}%
    \begingroup \hyphenpenalty\@M
    \@ifempty{#4}{%
      \@tempdima\csname r@tocindent\number#1\endcsname\relax
    }{%
      \@tempdima#4\relax
    }%
    \parindent\z@ \leftskip#3\relax \advance\leftskip\@tempdima\relax
    \rightskip\@pnumwidth plus1em \parfillskip-\@pnumwidth
    #5\leavevmode\hskip-\@tempdima{#6}\nobreak
    \leaders\hbox{$\m@th\mkern \@dotsep mu\hbox{.}\mkern \@dotsep mu$}\hfill
    \nobreak
    \hbox to\@pnumwidth{\@tocpagenum{\ifnum#1=1\bfseries\fi#7}}\par
    \nobreak
    \endgroup
  \fi}
\AtBeginDocument{%
\expandafter\renewcommand\csname r@tocindent0\endcsname{0pt}
}
\def\l@subsection{\@tocline{2}{0pt}{2.5pc}{5pc}{}}

\makeatother

\usepackage[top=2.7cm,bottom=2.7cm,left=2.65cm,right=2.7cm]{geometry}

\numberwithin{equation}{section}

\newtheorem{theorem}{Theorem}[section]
\newtheorem{lemma}{Lemma}[section]

\newtheorem{proposition}{Proposition}[section]

\newtheorem{remark}{Remark}[section]

\newtheorem{definition}{Definition}[section]

\DeclareMathOperator*{\esssup}{ess\,sup}
\newcommand{\bu}{\mathbf{u}}
\newcommand{\bq}{\mathbf{q}}

\newcommand{\rmdiv}{\mathrm{div}}
\newcommand{\bbS}{\mathbb{S}}
\newcommand{\bfF}{\mathbf{F}}

\newcommand{\sfB}{\mathsf{B}}
\newcommand{\bfn}{\mathbf{n}}
\newcommand{\bfv}{\mathbf{v}}
\newcommand{\rmd}{\mathrm{d}}
\newcommand{\bmphi}{\bm{\phi}}
\newcommand{\bmpsi}{\bm{\psi}}
\newcommand{\vphi}{\varphi}
\DeclareMathOperator{\R}{\mathbb{R}}
\begin{document}
\date{\today}

\title[MHD-structure interaction problem]
{Weak solutions to a full compressible magnetohydrodynamic flow interacting with thermoelastic structure}

\author[Bhandari]{Kuntal Bhandari}
\address[Kuntal Bhandari]{\newline Institute of Mathematics, Czech Academy of Sciences, \v{Z}itn\'{a} 25, 11567 Praha 1, Czech Republic.}
\email{bhandari@math.cas.cz.}

\author[Huang]{Bingkang Huang}
\address[Bingkang Huang]{\newline School of Mathematics, Hefei University of Technology, Hefei 230009, China.}
\email{bkhuang92@hotmail.com, bkhuang@whu.edu.cn.}

\author[Ne\v{c}asov\'{a}]{\v{S}\'{a}rka Ne\v{c}asov\'{a}}
\address[\v{S}\'{a}rka Ne\v{c}asov\'{a}]{\newline Institute of Mathematics, Czech Academy of Sciences, \v{Z}itn\'{a} 25, 11567 Praha 1, Czech Republic.}
\email{matus@math.cas.cz.}

{\thanks{\textbf{Funding information.} The research of B. Huang is supported by National Natural Science Foundation of China No.11901148 and the Fundamental Research Funds for the Central Universities No.JZ2022HGTB0257.    K. Bhandari and \v{S}. Ne{\v{c}}asov{\'{a}} received funding from   the  Praemium Academiae of \v{S}.  Ne{\v{c}}asov{\'{a}}. The Institute of Mathematics, CAS is supported by RVO:67985840}  
}


\begin{abstract} This paper is concerned with  an interaction problem  between a full compressible, electrically conducting fluid and a thermoelastic shell in a two-dimensional setting. 
The shell is modeled by linear thermoelasticity equations,  and encompasses a time-dependent domain which is filled with a fluid described by full compressible (non-resistive) magnetohydrodynamic equations. The magnetohydrodynamic flow and the shell are fully coupled, resulting in a fluid-structure interaction problem that involves heat exchange. We establish the existence of weak solutions through {\em domain extension}, {\em operator splitting}, {\em decoupling}, {\em penalization of the interface condition}, and {\em appropriate limit passages}.  
\\[2mm]
\noindent
{\bf Key words.} Fluid-structure interaction; full compressible magnetohydrodynamics; weak solutions. 
\\[2mm]
\noindent
{\bf 2020 Mathematics Subject Classification.} Primary: 35D30, 74F10; Secondary: 35M13, 74F05, 76W05.
\end{abstract}

\maketitle

{
\hypersetup{linkcolor=black}
\tableofcontents
}
\section{Introduction}

\subsection{The problem under study}
In this paper, we  consider the flow of a compressible, electrically and thermally conducting fluid in a two-dimensional region with  heat conducting elastic boundary. The fluid and structure are fully coupled through the continuity of velocity and temperature.\footnote{Specifically, we shall work with  entropy inequality  for the fluid and temperature equation for the  thermoelastic plate. We also point out that most of  the thermoelastic plate models in literature (including ours) are based upon   ``small  temperature''  with  respect to the
reference temperature, and  the entropy depends linearly on temperature; we refer \cite[Remark 1.5]{MMNRT-22}.}
Basically, the fluid domain is determined by elastic displacement which is in turn obtained by solving a linear Koiter shell equation with the forcing coming from the fluid. Thus, the fluid and the structure are fully coupled to each other, and it becomes a moving boundary problem.

\subsubsection{\bf Geometry of the structure and domain}
Let $\Omega\subset \mathbb{R}^2$ be an open, bounded, connected  domain whose boundary is parameterized by an injective map $\bmphi\in C^3(\Gamma;\mathbb{R}^2)$, that is $\partial\Omega=\bmphi(\Gamma)$, where $\Gamma=\mathbb{R}/\mathbb{Z}$ represents the one-dimensional torus (i.e., a circle). We use $\bfn(y)$ and $w(t,y)$ to denote the unit outer normal to $\Omega$ and the displacement of the structure, respectively.  With slight abuse of notation we shall identify functions
defined on $\Gamma$ and $\partial \Omega$. We further assume that the shell deforms in the normal direction, and that the structure displacement is  of the  form $w(t,y)\bfn(y)$, $y\in\Gamma$. Therefore, the deformed (or elastic) boundary at time $t$ can be  expressed by the following mapping: 
\begin{equation*}
\bmphi_{w}(t,y)=\bmphi(y)+w(t,y)\bfn(y),\quad t\in (0,T),\hspace{0.1cm} y\in\Gamma. 
\end{equation*}

According to a classical result from differential geometry (see \cite[Section 10]{Lee-03}), there exist positive constants $\alpha_{\partial\Omega}$ and $\beta_{\partial\Omega}$ such that the mapping $\bmphi_w(t,\cdot)$ is injective for $w(t,y)\in (\alpha_{\partial\Omega}, \beta_{\partial\Omega})$. The deformed boundary is then given by 
\begin{equation*}
\Gamma_{w}(t)=\{\bmphi_{w}(t,y)  \mid y\in \Gamma\}.
\end{equation*}

The fluid domain $\Omega_{w}(t)\subset\mathbb{R}^2$ varies with time and it is described as the interior of $\Gamma_{w}(t)$, such that $\partial\Omega_{w}(t)=\Gamma_{w}(t)$. 
To be more precise,  let $\widetilde{\bmphi_w}$ be an arbitrary injective smooth extension of $\bmphi_w$ to $\Omega$. Then the
fluid domain at time $t$ is defined by $\Omega_w(t)= \widetilde{\bmphi_w}(t, \Omega)$ (the detailed formulation is given by \eqref{flow-fun}).
 Observe  that, the fluid domain is
well-defined if $\bmphi_w(t, \cdot)$ is injective, which is true on  condition $w(t, \cdot)\in (\alpha_{\partial \Omega}, \beta_{\partial \Omega})$. 
Thus, our existence result is valid as long as this condition holds. 


Finally, we denote the moving domain and the interface domain as follows, both in Eulerian coordinates:
\begin{equation}\label{moving-domain}
Q_T^{w}=\cup_{t\in(0,T)}\{t\}\times\Omega_{w}(t), \quad \Gamma_T^{w}=\cup_{t\in(0,T)}\{t\}\times\Gamma_{w}(t). 
\end{equation}
Furthermore, the outer unit normal to $\Omega_{w}(t)$ is given by $\bfn_{w}(t):=\frac{\nabla \widetilde{\bmphi_{w}}^{-1}\bfn}{|\nabla \widetilde{\bmphi_{w}}^{-1}\bfn|}$, and the surface element is prescribed by 
$\rmd\Gamma_{w}(t)=\sigma_w\rmd\Gamma$, with $\sigma_{w}:=\mathrm{det} \nabla\widetilde{\bmphi_{w}}|\nabla \widetilde{\bmphi_{w}}^{-1}\bfn|$.


\subsubsection{\bf The equations for shell} 
We consider the following linear  thermoelasticity equations which describe the motion of the shell (see \cite[Section 1.5]{lagnese1989boundary}):  
\begin{equation}\label{shell}
	\left\{
		\begin{aligned}
&w_{tt}+\Delta^2w+\Delta\theta-\alpha_1\Delta w_{t}-\alpha_2\Delta w_{tt}=\sigma_w \bfF \cdot \bfn,\\
&\theta_t-\Delta\theta-\Delta w_t=\sigma_w q,
\end{aligned}
\right.
\end{equation}
where $w(t,y)$ is the displacement of the shell in the normal direction $\bfn$ w.r.t. the reference configuration $\Gamma$, $\theta(t,y)$ is the temperature of the shell, $\bfF$ is the surface force coming from the  pressure of the fluid and stress tensor, and $q$ is the entropy flux. Moreover, the constants  $\alpha_1$ and $\alpha_2$ are positive and respectively denote the coefficients of viscoelasticity and rotational inertia.


\subsubsection{\bf The magnetohydrodynamic flow}  
The full compressible magnetohydrodynamic (MHD) flow  in $(0,T)\times \mathbb R^{3}$ is governed by the following set of partial differential equations  (see for instance \cite{cabannes2012theoretical,Ducomet-Feireisl-CMP}):
\begin{align}
&\varrho_t+\rmdiv(\varrho\bu)=0,  \label{conti-eq} \\
&\mathbf B_t-\textbf{curl} (\bu \times \mathbf B) + \textbf{curl} (\nu \textbf{curl}\, \mathbf B) =0, \ \ \rmdiv \mathbf{B}=0 , 
\label{mag-eq}\\
& (\varrho \bu)_t + \rmdiv (\varrho \bu \otimes \bu) + \nabla p(\varrho, \vartheta) = \rmdiv \bbS(\vartheta, \nabla \bu) + \textbf{curl} \, \mathbf B \times \mathbf B ,  \label{momen-eq}  \\
&(\varrho s)_t  + \rmdiv (\varrho s \bu) + \rmdiv \left( \frac{\bq}{\vartheta}\right) = 
\frac{1}{\vartheta} \left( \bbS (\vartheta, \nabla \bu) : \nabla \bu - \frac{\bq \cdot \nabla \vartheta}{\vartheta} + \nu |\textbf{curl}\, \mathbf B|^2\right)  \label{entropy-eq} . 
\end{align}
Here, the unknowns $\varrho$, $\bu$, $\mathbf B$, $\vartheta$, $p$, and $s$ respectively represent the density, velocity,   magnetic field, temperature,  pressure, and entropy of the flow. The viscous stress tensor $\bbS(\vartheta,\nabla\bu)$ is defined by 
\begin{equation*} 
\bbS(\vartheta,\nabla\bu)=\mu(\vartheta)\left(\nabla\bu+\nabla^\top \bu- \frac{2}{3}\rmdiv\bu\mathbb{I}_{3}\right)+\eta(\vartheta)\rmdiv\bu\mathbb{I}_3, 
\end{equation*}
where $\mu(\vartheta)$, $\eta(\vartheta)$ are the shear and bulk viscosity coefficients, and $\mathbb{I}_3$ is the identity matrix in $\mathbb{R}^{3\times 3}$. 

Furthermore, $\nu$ is the resistivity coefficient, representing the diffusion of magnetic field. The heat flux $\bq$ satisfies the Fourier's law: 
\begin{equation*} 
	\bq=-\kappa(\vartheta)\nabla\vartheta,
\end{equation*}
where $\kappa(\vartheta)$ is the heat conductivity coefficient. 
Moreover,  Gibbs' relation holds  
\begin{equation}\label{Gibbs' eqeuation}
\vartheta D s(\varrho, \vartheta)=De(\varrho,\vartheta)+p(\varrho,\vartheta)D\left(\frac{1}{\varrho}\right), 
\end{equation}
 where $D$ denotes total differential and $e$ is the total energy of the flow.

\vspace*{.1cm} 

Motivated from the work \cite{LS-21},   in the present paper, we consider the case when the
motion of fluid takes place in the plane, precisely in the domain $Q^w_T$ while the magnetic field acts on the fluid only in the vertical direction. Thus, we choose 
\begin{align}
\begin{dcases} 
\varrho(t,x)= \varrho(t,x_1,x_2), \ \ \bu(t,x) = (\bu_1, \bu_2, 0 ) (t,x_1,x_2) , \\
\vartheta (t,x) = \vartheta(t,x_1,x_2), \ \ \mathbf B(t,x) = (0,0,b)(t,x_1,x_2) .
\end{dcases}
\end{align}
Here and in the sequel, we denote by $x$ the 2D spatial variable $(x_1,x_2)$, and by  the the notation $\bu$ we mean $(\bu_1,\bu_2)$, the velocity of the fluid in the plane domain. So, the system \eqref{conti-eq}--\eqref{entropy-eq} on time-space domain $Q^w_T$, reduces to  
%
\begin{equation}\label{full-MHD}
	\left\{
		\begin{aligned}
&\varrho_t+\rmdiv(\varrho\bu)=0,\\
&b_t+\rmdiv(b\bu)=0,\\
&(\varrho\bu)_t+\rmdiv(\varrho\bu\otimes\bu)+\nabla\left(p(\varrho,\vartheta)+\frac{1}{2}b^2\right)=\rmdiv\bbS(\vartheta,\nabla\bu),\\
&(\varrho s)_t+\rmdiv(\varrho s\bu)+\rmdiv\left(\frac{\bq}{\vartheta}\right)=\frac{1}{\vartheta}\left(\bbS(\vartheta,\nabla\bu):\nabla\bu-\frac{\bq\cdot\nabla\vartheta}{\vartheta}\right). 
\end{aligned}
\right.
\end{equation}
In above, the viscous stress tensor takes the form
\begin{equation}\label{tensor}
\bbS(\vartheta,\nabla\bu)=\mu(\vartheta)\left(\nabla\bu+\nabla^\top \bu- \rmdiv\bu\mathbb{I}_{2}\right)+\eta(\vartheta)\rmdiv\bu\mathbb{I}_2, 
\end{equation}
where  $\mathbb{I}_2$ is the identity matrix in $\mathbb{R}^{2\times 2}$.

\subsubsection{\bf The kinematic and dynamic coupling conditions} The kinematic coupling conditions state that 
the velocity and temperature are continuous on the interface $\Gamma_T:=(0,T)\times\Gamma$, namely, 
\begin{equation}\label{VT-con}
	\begin{aligned}
&w_t\bfn=\bu\circ \bmphi_w,\\
& \theta=\vartheta\circ \bmphi_w.
	\end{aligned}
\end{equation}
The structure interface elastodynamics is driven by the stress tensor and entropy,
\begin{equation}\label{F-Q-con}
	\begin{aligned}
&\bfF=-\left[\left(p(\varrho,\vartheta)\mathbb{I}_2+\frac{1}{2}b^2\mathbb{I}_2-\bbS(\vartheta, \nabla\bu)\right)\bfn_{w}\right]\circ\bmphi_{w},\\
&q=-\left(\frac{\bq}{\vartheta}\cdot\bfn_{w}\right)\circ\bmphi_{w},
\end{aligned}
\end{equation}
which are said to be the dynamic coupling conditions.

\subsubsection{\bf Initial conditions}
The initial data are given in the  following way:
\begin{equation}\label{ini-data}
	\begin{aligned}
&\varrho(0,\cdot)=\varrho_0,\hspace{0.3cm} b(0,\cdot)=b_0, \hspace{0.3cm} (\varrho\bu)(0,\cdot)=(\varrho\bu)_0, \hspace{0.3cm} \vartheta(0,\cdot)=\vartheta_0, \\
&w(0,\cdot)=w_0,\hspace{0.3cm}  w_t(0,\cdot)=v_0,\hspace{0.3cm} \theta(0,\cdot)=\theta_0.
\end{aligned}
\end{equation}
Moreover, we consider that  the initial data satisfy 
\begin{align}\label{ini-data-1}
&\varrho_0\in L^{\gamma}(\Omega_{w}(0)),\hspace{0.3cm}\varrho_0\geq 0, \hspace{0.3cm}\varrho_0\not\equiv 0, \hspace{0.3cm} \varrho_0|_{\mathbb{R}^2\setminus \Omega_{w}(0)}=0,
\notag\\
&\varrho_0>0 \hspace{0.3cm} \mathrm{in}\hspace{0.3cm} \{x\in \Omega_w(0) \mid (\rho\bu)_0 >0\},  \hspace{0.3cm} 
\frac{(\varrho\bu)_0^2}{\varrho_0}\in L^1(\Omega_{w}(0)), \notag\\
&b_0\in L^2({\Omega_w(0)}), \hspace{0.3cm}b_0\geq 0, \hspace{0.3cm}b_0\not\equiv 0, \hspace{0.3cm} b_0|_{\mathbb{R}^2\setminus \Omega_{w}(0)}=0,\\
&\vartheta_0 \in L^4(\Omega_w(0)), \ \vartheta_0>0 \hspace{0.3cm}\mathrm{a.e.} \hspace{0.2cm} \mathrm{in} \hspace{0.2cm} \Omega_{w}(0), \hspace{0.3cm}(\varrho s)_0:=\varrho_0s(\varrho_0,\vartheta_0)\in L^1(\Omega_{w}(0)),
\notag\\
&v_0\in L^2(\Gamma), \hspace{0.3cm} \sqrt{\alpha_2}v_0\in H^1(\Gamma), \hspace{0.3cm} w_0\in H^2(\Gamma), \hspace{0.3cm} \theta_0\in L^2(\Gamma), \hspace{0.3cm} \theta_0\geq 0 ,\notag 
\end{align}
Finally we assume that the following compatibility condition holds:
\begin{align*}
\alpha_{\partial \Omega} < w_0 < \beta_{\partial \Omega}.  
\end{align*}

\subsubsection{\bf{Constitutive relations}}
To simplify the presentation and include physically relevant cases,   we assume, in agreement with \cite{FN-09},  that the viscosity coefficients in $\eqref{tensor}$ are continuously differentiable functions of the temperature, namely $\mu, \eta \in C^1([0,\infty))$ satisfying 
\begin{equation}\label{vis-coe}
	\begin{aligned}
	&0<\underline{\mu}(1+\vartheta)\leq \mu(\vartheta)\leq \overline{\mu}(1+\vartheta), \hspace{0.3cm}\sup_{\vartheta\in[0,\infty)}|\mu'(\vartheta)|\leq C,\\
	&0\leq \underline{\eta}(1+\vartheta) \leq  \eta(\vartheta)\leq \overline{\eta}(1+\vartheta), \hspace{0.3cm}  \forall \vartheta>0, 
\end{aligned}
\end{equation}
where $\underline{\mu}>0$, $\overline{\mu}>0$, $\underline{\eta}\geq 0$ and $\overline{\eta}\geq 0$ are  constants. 

The heat conductivity coefficient satisfies 
\begin{equation}\label{heat-coe}
		0<\underline{\kappa}(1+\vartheta^3)\leq \kappa(\vartheta)\leq \overline{\kappa}(1+\vartheta^3), \hspace{0.3cm}  \forall \vartheta>0,
\end{equation}
where $\underline{\kappa}$ and $\overline{\kappa}$ are positive constants. 

\vspace*{.2cm} 

Further, motivated from \cite{FN-09,Ducomet-Feireisl-CMP}, we consider the following constitutive relations for the pressure, internal energy and entropy.  The pressure and internal energy satisfy 
\begin{equation}\label{state-equ}
\begin{aligned}
& p(\varrho,\vartheta)=p_M(\varrho,\vartheta)+p_R(\vartheta), \hspace{0.3cm} p_{M}(\varrho,\vartheta) ={\varrho^\gamma+\varrho\vartheta},\hspace{0.3cm}p_R(\vartheta)=\frac{a}{3}\vartheta^4,   \\
& e(\varrho,\vartheta)=e_M(\varrho,\vartheta)+e_R(\varrho, \vartheta), \hspace{0.3cm} e_{M}(\varrho,\vartheta)=\frac{\varrho^{\gamma-1}}{\gamma-1}+\vartheta ,\hspace{0.3cm}e_R(\varrho, \vartheta)=\frac{a}{\varrho}\vartheta^4, \\
& \mathrm{where \hspace{0.1cm}the  \hspace{0.1cm}Stefan-Boltzmann\hspace{0.1cm}constant\hspace{0.1cm} }a>0, \  \mathrm{and \hspace{0.1cm} {adiabatic\hspace{0.1cm}  index \hspace{0.1cm} \gamma>1}}.
\end{aligned}
\end{equation} 
On the other hand, in terms of Gibb's equation $\eqref{Gibbs' eqeuation}$, the entropy satisfies 
\begin{equation}\label{ent-sta}
s(\varrho,\vartheta)=s_M(\varrho,\vartheta)+s_R(\varrho,\vartheta), \hspace{0.3cm} s_M(\varrho,\vartheta)={\ln\left(\frac{\vartheta}{\varrho}\right)},\hspace{0.3cm}s_R(\varrho,\vartheta)=\frac{4a}{3}\frac{\vartheta^3}{\varrho}.
\end{equation}

\medskip 

In the present work, we aim at proving the existence of weak solutions to the fluid-structure interaction system $\eqref{shell}$, $\eqref{full-MHD}$ under the coupling conditions $\eqref{VT-con}$-$\eqref{F-Q-con}$, the initial conditions $\eqref{ini-data}$-$\eqref{ini-data-1}$, and constitutive relations \eqref{vis-coe}-\eqref{ent-sta}.

\subsection{Bibliographic comments}  
The fluid-structure interaction (FSI) problem arises in many real-life situation and understanding the interaction between the fluid and structure is crucial for various applications \cite{Bodnar-Sarka-Galdi}.  The mathematical study of FSI problems has gained its popularity in the last two decades and a lot of progress has been achieved. 
Let us mention some interesting results concerning the incompressible fluid-structure interaction problems. 
In \cite{CDEG-05,MS-22,MC-13,Grandmont-SIMA,LR-14}  (and references therein), the authors dealt with  
 the existence of a weak
solution of  FSI problems where the elastic structure is described by a lower dimensional model of a plate/shell type. On the other hand, in \cite{Benesova-JEMS}, the authors analyzed  the existence of a weak solution to FSI problem which involve regularized, nonlinear 3D viscoelastic structure.  We further mention that the work \cite{Muha-Canic-JDE} is dealt with FSI problem with linear multilayered structure. The above works consider large data case and a weak solution exists  as long as geometry does not degenerate, i.e.,  self-contact does not occur. 
For 
the local-in-time or small data existence of strong solutions to FSI problems, we refer the works 
\cite{Grandmont-Lequeurre-IHP,Ignatova-et-al,Maity-Raymond-Roy,Raymond-Vaninathan} and  references within. We further mention the work \cite{Grandmont-Hillairet}, where global-in-time
solution to a 2D-1D FSI model with a viscoelastic beam has been proved.   

\vspace*{.1cm}

For the case of compressible FSI problems, the mathematical literature  is comparatively less rich than the incompressible one. In this context, we first mention that the local-in-time existence of regular solutions has been proved in \cite{Boulakia-Guerrero,Kukavica-Tuffaha}.  
The local-in-time existence result for strong solution in 2D-1D framework has been established for
the compressible fluid-damped beam interaction in  \cite{Mitra-jmfm}.  For compressible  fluid-undamped wave interaction in a 3D-2D framework,  we refer \cite{Maity-Roy-Takeo}. Now,
concerning the existence of weak solutions for compressible FSI models, we mention the works  \cite{Breit-Sebastian-ANNPisa,Srjdan-Wang,Breit-Sebastian-ARMA}.  
Recently, the authors \cite{MMNRT-22} proved the existence of global-in-time weak solutions to the compressible Navier-Stokes-Fourier system interacting with thermoelastic shell. Furthermore, the existence of weak solutions to  the interaction problem between two compressible mutually non-interacting fluids and a shell of Koiter-type is considered in \cite{KMN-24}.  On the other hand, in \cite{Maity-takeo-nonlinear}
the existence of a strong solution for small time or small data is proven in the case
when there is an additional damping on the structure.

\vspace*{.1cm}

The elastic structure usually encompasses a time-dependent fluid domain. In a general time-dependent domain, the study of compressible  fluids is also an interesting mathematical problem and in this regard, we mention for instance the following works \cite{FKNNS-JDE,Kreml-jmpa-time-dep,BHN-2024,Kreml-Sarka-Piasecki}. The global-in-time weak solutions of the full compressible MHD flow in fixed spatial domain has been established in \cite{Ducomet-Feireisl-CMP} (see also the references therein). Recently, the existence of weak solutions for compressible MHD  flow driven by 
non-conservative boundary conditions is explored in \cite{Feireisl-JMFM-magnetic}. 

\vspace*{.1cm}

To the best of our knowledge, the mathematical analysis of  MHD flow interacting with elastic structure has not been analyzed in the literature.  In the present work, we are going to fill this gap at least partially by considering the two-dimensional situation where the magnetic field acts on the fluid  in the vertical direction.  

\subsection{Main result} 
We are in a position to state the main result of this work. 
\begin{theorem}\label{Main-Theorem} 
Suppose that the initial data $(\varrho_0, b_0, (\varrho\bu)_0, \vartheta_0, (\varrho s)_0, w_0, v_0, \theta_0)$ satisfy $\eqref{ini-data}$-$\eqref{ini-data-1}$. Assume that the equations of states satisfy $\eqref{state-equ}$-$\eqref{ent-sta}$, and the transport coefficients comply with $\eqref{vis-coe}$-$\eqref{heat-coe}$. Then, for given $\displaystyle\gamma>\frac{5}{3}$,  there exist a time $T>0$ and a weak solution to the fluid-structure interaction system $\eqref{shell}$, $\eqref{full-MHD}$ with $\eqref{VT-con}$ and $\eqref{F-Q-con}$ in the sense of Definition $\ref{weak-solutions}$ on the time interval $(0,T)$. Furthermore, either it holds that $T\rightarrow +\infty$ or the domain $\Omega_w(t)$ degenerates as $t\rightarrow T$. 
\end{theorem}

Let us state the main features in the analysis of our  fluid-structure interaction problem.

\begin{itemize}
    \item 
We consider a compressible, electrically and thermally conducting fluid in a two-dimensional bounded domain with an elastic boundary. Under the geometric assumptions on the structure, the displacement function $w(t,y)$ exhibits regularity, specifically $w \in W^{1,\infty}(0,T; L^2(\Gamma)) \cap L^\infty(0,T; H^2(\Gamma))$, which implies that the boundary defined in $\eqref{moving-domain}$ is Lipschitz (see $\eqref{embedding-Delta-t}$). This observation ensures that the regularity of the transformation is preserved, unlike when the structure is placed on a two-dimensional reference configuration $\Gamma=\mathbb{R}^2/\mathbb{Z}^2$ (see, for instance, \cite{MMNRT-22,KMN-24}). The fundamental reason is that $w \in H^{2}(\Gamma)$ does not guarantee the boundedness of $\|\nabla w\|_{L^\infty(\Gamma)}$ when $\Gamma=\mathbb{R}^2/\mathbb{Z}^2$. However, when $\Gamma =\mathbb{R}/\mathbb{Z}$, this issue does not arise. Furthermore, both the trace theorem and Korn's inequality hold without any loss of regularity.

    \item In the motion of the fluid, the magnetic field acts solely in the vertical direction. Consequently, the equation governing the magnetic field degenerates into a transport equation, which facilitates the use of the domain extension method to establish a weak solution. Moreover, domination conditions between the magnetic field and density are no longer required, as motivated by \cite{Wen-21}, which enhances the significance of our results. Notably, the range of the adiabatic index $\displaystyle \gamma > \frac53$ ensures that $\varrho \in L^2(0,T; L^2{(\Omega_w(t))})$ (a result derived in Lemma \ref{high-int-press}, see also Remark \ref{Remark-gamma-condition}), and the strong convergence of the densities w.r.t. the artificial pressure parameter $\delta$.   Subsequently, the renormalized equation holds, as detailed in Lemma \ref{Renormalized-equ}.
\end{itemize}

\subsection{Sketch of the proof for Theorem \ref{Main-Theorem}}   The proof of Theorem \ref{Main-Theorem} is split into three parts that associates to several approximation levels. We give a short sketch of the proof below. 

\begin{itemize} 
\item Given the assumption regarding the restricted deformations of the shell, there exists a sufficiently large $M_0 > 0$ such that $\Omega_w(t) \subseteq \mathsf{B} := {|x| < M_0}$ for all $t \in [0,T]$. Consequently, we consider the existence of a weak solution in the fixed domain $(0,T) \times \mathsf{B}$, thus avoiding the need to address solvability in variable domains. This demonstrates the robustness of the method and simplifies the proof. The extension involves approximating viscosity coefficients, heat conductivity coefficient, pressure, internal energy and entropy, by using several ``parameters'' (see Section \ref{Section-extension-parameters}). Additionally, we introduce the regularizing term $\delta \varrho^{\beta}$ ($\delta > 0$, $\beta \geq 4$) with the pressure. This regularization improves the integrability of the pressure term, a result that is well-known in the study of compressible fluids.

\item We must prove the existence of a weak solution in the fixed domain. To construct approximate solutions to the extended problem, we introduce a time step parameter $\Delta t$ and a time-marching scheme that combines a decoupling approach (for the fluid and the structure) with penalization of the kinematic coupling conditions. The main advantage of this approach is that the existence result for the fluid part (in the time-dependent domain) can be directly used from \cite{KMNW-18}.

\item By taking the initial density $\varrho_0$ (resp. the initial magnetic field $b_0$) as vanishing outside 
$\Omega_w(0)$,  it can be ensured that the density $\varrho$ (resp. the magnetic field $b$) vanishes outside the physical domain $\Omega_w(t)$ for each $t \in (0,T]$. Using this, our goal is to pass to the limit of the parameters associated with the viscosity coefficients, heat conductivity coefficient, pressure, internal energy, and entropy. In essence, we adapt the ideas from \cite{KMNW-18} (see also \cite{Kreml-jmpa-time-dep}) to obtain a solution in the physical space-time domain $Q_T^w$.

\item In the final step, we pass to the limit of $\delta$ (the pressure regularization limit) and determine the maximal interval of existence. This is a standard procedure in the existence proofs of weak solutions to fluid-structure interaction problems.

\end{itemize}

\subsection{Paper organization}
The rest of the paper is organized as follows. In Section \ref{Section-Geometry-ext}, the geometry extension and the concept of a weak solution are introduced. In Section \ref{Section-fixed-domain}, we establish the existence of a weak solution to the corresponding problem in a fixed domain. Several limit passages are considered in Sections \ref{Sec-limit-Delta-t}, \ref{Section-Limits} and \ref{Section-limit-delta}. Finally, concluding remarks and some useful lemmas are presented in Section \ref{Sec-Conclusion} and Appendix \ref{Sec-Appendix}, respectively.

\section{Geometry extension and weak solution}\label{Section-Geometry-ext}

\subsection{Geometry extension}
Define a tubular neighborhood of $\partial\Omega$ as 
$$N_{\alpha,\beta}:=
\left\{\bmphi(y)+\bfn(\bmphi(y))z \mid y\in \Gamma, z\in(\alpha_{\partial\Omega}, \beta_{\partial\Omega})\right\},$$
and the projection mapping $\pi:N_{\alpha,\beta}\rightarrow \partial\Omega$ such that for any $y\in N_{\alpha,\beta}$, there exist a unique $\pi(y)\in \partial \Omega$ and $z$ satisfying $y-\pi(y)=\bfn(\pi(y))z$. Furthermore, the signed distance function $d: N_{\alpha,\beta}\rightarrow (\alpha_{\partial\Omega}, \beta_{\partial\Omega})$ is defined as 
$$ d: y \mapsto (y-\pi(y))\cdot\bfn(\pi(y)). $$  

We also denote the following distance function
\begin{equation*} 
\widetilde{d}(y)=\left\{
\begin{aligned}
&-\mathrm{dist}(y,\partial\Omega), \quad \mathrm{when} \quad y\in \overline\Omega,\\
&\mathrm{dist}(y,\partial\Omega),\quad \mathrm{when} \quad y\in \mathbb{R}^2\backslash\overline\Omega.
\end{aligned}
\right.
\end{equation*}
Note that, $d$ and $\widetilde d$ agree in the tubular region $N_{\alpha, \beta}$.

Let $w$ be any displacement function such that
\begin{align*}
&w \in  W^{1, \infty}(0,T; L^2(\Gamma)) \cap L^\infty(0,T;  H^2(\Gamma))  \cap H^1(0,T; H^1(\Gamma)), \\
&\alpha_{\partial \Omega} < m \leq  w \leq  M < \beta_{\partial \Omega} \textnormal{ on } \Gamma_T. 
 \end{align*}
Subsequently, we define the flow function $\widetilde{\bmphi_w}:[0,T]\times\mathbb{R}^2\rightarrow \mathbb{R}^2$
\begin{equation}\label{flow-fun}
  \widetilde{\bmphi_w}(t,y)=y+f_{\Lambda}(\widetilde{d}(y))w(t,\bmphi^{-1}(\pi(y)))\bfn(\pi(y)),
\end{equation}
where cut-off function $f_{\Lambda}(\cdot)\in C_c^\infty(\mathbb{R})$ satisfying $f_{\Lambda}(\cdot)=(f*g_{\alpha})(\cdot)$ with a standard mollifier $g_{\alpha}$ which has support in $(-\alpha, \alpha)$ for $\alpha<\frac{1}{2}\min\{m'-m'', M''-M'\}$, where we consider 
\begin{align*}
\alpha_{\partial\Omega}<m''<m'<m<M<M'<M''<
\beta_{\partial\Omega}  , 
\end{align*} 
 and $f(\cdot)$ is given by  
\begin{align*}
   f(y) = 
\begin{cases} 
    \displaystyle   1 - \frac{y - m'' + m'}{m'}   & \text{ for } y \in (m'', m'-m'') , \\
 \displaystyle   1 & \text{ for } y \in [m'-m'', M''-M'), \\
    \displaystyle    1-\frac{y-M''+M'}{M'} & \text{ for } y \in [M'' - M', M''), \\
     \displaystyle      0 & \text{ for } y \in (-\infty, m''] \cup  [M'', + \infty).
\end{cases} 
\end{align*} 
In fact, for $y\in N_{\alpha,\beta}$, one can rewrite the map $\widetilde{\bmphi_w}$ (with the help of $y-\pi(y)=\mathbf n(\pi(y)) d(y)$) as follows: 
\begin{equation}\label{extension-fun}
  \widetilde{\bmphi_w}(t,y) = \left[1-f_{\Lambda}({d}(y))\right] y+f_{\Lambda}({d}(y))\left[\pi(y)+\left({d}(y)+ w(t,\bmphi^{-1}(\pi(y))) \right)\bfn(\pi(y))\right], 
\end{equation}
and for $z \in N_{\alpha,\beta}$, the associated inverse map is  
\begin{equation*}
  \widetilde{\bmphi_w}^{-1}(t,z)=\left[1-f_{\Lambda}({d}(z))\right]z + 
  f_{\Lambda}({d}(z))\left[\pi(z)+
  \left( {d}(z) - w(t,\bmphi^{-1}(\pi(z)) ) \right)\bfn(\pi(z))\right]. 
\end{equation*}
In other words, for $\widetilde{\bmphi_w}^{-1}:[0,T]\times\mathbb{R}^2\rightarrow \mathbb{R}^2$, we can write  
\begin{equation*}
  \widetilde{\bmphi_w}^{-1}(t,z)=z
  -f_{\Lambda}(\widetilde{d}(z)) w(t,\bmphi^{-1}(\pi(z)))\bfn(\pi(z)).
\end{equation*}

From the explicit expressions of $\widetilde{\bmphi_w}$ and $\widetilde{\bmphi_w}^{-1}$, we observe that they inherit the same regularity from $w$.  

\vspace*{.1cm}
With the above constructions, 
we define (as mentioned earlier) $\Omega_w(t)=\widetilde{\bmphi_w}(t, \Omega)$ with the boundary $\Gamma_w(t)=\widetilde{\bmphi_w}(t, \partial\Omega)$ for all $t\in [0,T]$, where $\Omega \subset \mathbb R^2$, as given before with $\partial \Omega = \bmphi(\Gamma)$, and 
$\widetilde{\bmphi_w}$ is given by $\eqref{extension-fun}$. 

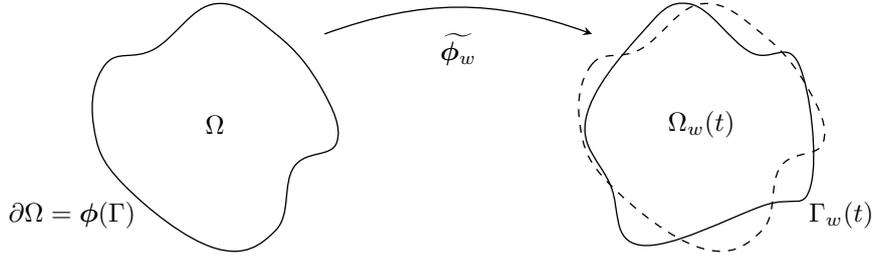
\begin{figure}[htbp!]
	\centering
	\begin{tikzpicture}[line width=0.5pt,scale=0.8]
		
		\draw plot[smooth cycle, tension=0.7] coordinates{(0,0) (0.2,1) (1,1.25) (2,2) (3,1.5) (3.75,0.6) (4,-0.3) (3.3,-0.7) (3,-1.75) (2,-2.05) (0.5,-1)};

 \draw [shift ={(8,0)}] plot[smooth cycle, tension=0.7] coordinates{(0.1,0.1) (1,1.5) (1.8,2) (2.8,1.2) (3.6,1) (3.8,-1) (3,-1.35) (1,-2) (0.5,-1)};
 \draw [dashed,shift ={(8,0)}] plot[smooth cycle, tension=0.7] coordinates{(0,0) (0.2,1) (1,1.25) (2,2) (3,1.5) (3.75,0.6) (4,-0.3) (3.3,-0.7) (3,-1.75) (2,-2.05) (0.5,-1)};

		
		
		\draw [-stealth] (3.8,1.5) to [bend left=20] (8.2,1.5);
		\draw node at (6,1.25) {$\widetilde{ \bmphi_w}$}; 
		
		\draw node at (2,0) {$\Omega$};
		\draw node at (-0.35,-1.5) {$\partial \Omega=\bmphi(\Gamma)$};
		\draw node at (10,0) {$\Omega_w(t)$};
		\draw node at (12.3,-1.5) {$\Gamma_w(t)$};
		
	\end{tikzpicture}
	\caption{The $2$D fluid domain $\Omega_w(t)$ determined by the shell $\Gamma_w(t)$  at time $t$.}
	\label{fig:domain}
\end{figure}


\subsection{Weak formulations}
We present the definition of weak solutions to the fluid-structure system $\eqref{shell}$, $\eqref{full-MHD}$ with the kinematic and dynamic conditions $\eqref{VT-con}$-$\eqref{F-Q-con}$, initial data \eqref{ini-data-1}, and constitutive relations \eqref{vis-coe}, \eqref{heat-coe}, \eqref{state-equ} and  \eqref{ent-sta}.  Compared to the classical definition of the weak solutions to the full compressible MHD, the equations of the motion of the shell $\eqref{shell}$ should be also incorporated. Specifically, the explicit definition of weak solutions is given in the following sense.
\begin{definition}\label{weak-solutions}
  We call that $(\varrho, b, \bu, \vartheta, w, \theta)$ is a weak solution to the fluid-structure system $\eqref{shell}$, $\eqref{full-MHD}$ with the initial data $(\varrho_0, b_0, (\varrho\bu)_0, \vartheta_0, (\varrho s)_0, w_0, v_0, \theta_0)$ satisfying $\eqref{ini-data}$-$\eqref{ini-data-1}$ if the following conditions hold: 
  \begin{itemize}
\item Function spaces:
unknowns $(\varrho, b, \bu, \vartheta)$ satisfy that 
\begin{equation}\label{fun-spa-1}
\left\{\begin{aligned}
&\varrho \geq 0,\hspace{0.3cm}b\geq 0, \hspace{0.3cm} \vartheta >0\hspace{0.3cm} \mathrm{a.e.\hspace{0.1cm} in}\hspace{0.3cm} Q_T^{w},\\
&\varrho \in L^\infty(0,T; L^\gamma(\Omega_w(t))) \cap L^2(0,T;L^2(\Omega_w(t))), \\
&b \in L^\infty(0,T; L^2(\Omega_w(t))), \\
&\bu \in L^2(0,T;W^{1,p}(\Omega_w(t))) \hspace{0.1cm}{\mathrm{with}} \hspace{0.1cm} 1\leq p \leq 2, \hspace{0.3cm} \varrho|\bu|^2\in L^\infty(0,T;L^1(\mathbb{R}^2)),\\
&\vartheta \in L^\infty (0,T;L^4(\Omega_w(t))), \hspace{0.3cm}\vartheta, \nabla\vartheta, \log\vartheta, \nabla \log\vartheta \in L^2(Q_T^{w}),\\
&\varrho s, \varrho s\bu, \frac{\bq}{\vartheta} \in  L^1(Q_T^{w});
\end{aligned}
\right.
\end{equation}
and unknowns $(w, \theta)$ satisfy that 
\begin{equation}\label{fun-spa-2}
	\left\{\begin{aligned}
& w \in L^\infty (0,T;H^2(\Gamma))\cap W^{1,\infty}(0,T; L^2(\Gamma)), \\
& \sqrt{\alpha_1}w\in H^1(0,T;H^1(\Gamma)),\hspace{0.3cm} \sqrt{\alpha_2} w \in W^{1,\infty}(0,T;H^1(\Gamma)), \\
& \theta >0\hspace{0.1cm} {\mathrm{a.e.}}\hspace{0.1cm}\mathrm{in} \hspace{0.1cm} \Gamma_T, \hspace{0.3cm} \theta \in L^\infty(0,T;L^2(\Gamma))\cap L^2(0,T;H^1(\Gamma)),\\
&\log \theta \in L^2(0,T;H^s(\Gamma)), \hspace{0.1cm} \mathrm{with} \hspace{0.1cm}0< s\leq \frac{1}{2}.
  \end{aligned}
\right.
\end{equation}

\item The continuity on the interface $\Gamma_T^{w}$: $w_t \bfn=\bu|_{\Gamma_{w}(t)}$ and $\theta=\vartheta|_{\Gamma_{w}(t)}$ for all $t\in(0,T)$.

\item The continuity equation
\begin{equation}\label{weak-con}
\iint_{Q_T^w}\left(\varrho\partial_t\phi+\varrho\bu\cdot\nabla\phi\right)=-\int_{\Omega_w(0)}\varrho_0\phi(0,\cdot)
\end{equation}
holds for all $\phi\in C^1_c([0,T)\times\overline{\Omega_w(t)})$.


\item The non-resistive magnetic equation
\begin{equation}\label{mag-equ}
\iint_{Q_T^w}\left(b\partial_t\phi+b\bu\cdot\nabla\phi\right)=-\int_{\Omega_w(0)}b_0\phi(0,\cdot)
 \end{equation}
 holds for all $\phi\in C^1_c([0,T)\times\overline{\Omega_w(t)})$.

We must mention here that the equations \eqref{weak-con} and \eqref{mag-equ} should also be satisfied in the sense of renormalized solutions as was introduced by DiPerna and Lions \cite{Lions-Diperna}, see Lemma \ref{Renormalized-equ} in the present paper.

\item The coupled momentum equation
  \begin{align}\label{mom-equ}
 &\iint_{Q_T^w}\varrho\bu \cdot\partial_t\bmphi+\iint_{Q_T^w}\left(\varrho\bu\otimes\bu+p(\varrho,\vartheta)\mathbb{I}_2+\frac{1}{2}b^2\mathbb{I}_2-\bbS(\vartheta,\nabla\bu)\right): \nabla\bmphi\notag\\
 &+\iint_{\Gamma_T}(w_t\psi_t-\Delta w\Delta\psi-\alpha_1\nabla w_t\cdot\nabla\psi +\alpha_2\nabla w_t\cdot\nabla\psi_t+\nabla\theta\cdot\nabla\psi)\\
 &=-\int_{\Omega_w(0)}(\varrho\bu)_0\bmphi(0,\cdot)-\int_{\Gamma}v_0\psi(0,\cdot)-\alpha_2\int_{\Gamma}\nabla v_0\cdot\nabla\psi(0,\cdot), \notag
 \end{align}
holds for all $\bmphi\in C^1_c([0,T)\times\overline{\Omega_w(t)}) $ and $\psi\in C_c^\infty([0,T)\times \Gamma)$ such that $\bmphi|_{\Gamma_w(t)}= \psi \bfn$ on $\Gamma_T$. 

\item The coupled entropy inequality
  \begin{align}\label{entropy-ine}
&\iint_{Q_T^w}\left(\varrho s\vphi_t+\varrho s \bu\cdot\nabla\vphi-\frac{\kappa(\vartheta)\nabla\vartheta}{\vartheta}\cdot\nabla\vphi\right)+\iint_{Q_T^w}\frac{\vphi}{\vartheta}\left(\bbS(\vartheta,\nabla\bu):\nabla\bu+\frac{\kappa(\vartheta)|\nabla\vartheta|^2}{\vartheta}\right)\notag\\
&+\iint_{\Gamma_T}(\theta\tilde\varphi_t-\nabla\theta\cdot\nabla\tilde\varphi+\nabla w\cdot\nabla\tilde\varphi_t)\\
&\leq -\int_{\Omega_w(0)}\varrho_0 s(\varrho_0,\vartheta_0)\vphi(0,\cdot)-\int_{\Gamma}\theta_0\tilde\varphi(0,\cdot)-\int_{\Gamma}\nabla w_0\cdot\nabla\tilde\varphi(0,\cdot), \notag
\end{align}
holds for all non-negative function $\vphi\in C^1_c([0,T)\times\overline{\Omega_w(t)}) $ and $\tilde\varphi \in C_c^\infty([0,T)\times \Gamma)$ such that $\vphi|_{\Gamma_w(t)}=\tilde\vphi$ on $\Gamma_T$. 

\item The coupled energy inequality
  \begin{align}\label{energy-ine}
    &\int_{\Omega_w(t)}\left(\frac{1}{2}\varrho|\bu|^2+\frac{1}{2}b^2+\varrho e(\varrho, \vartheta)\right)(t)\notag\\
    &+\int_{\Gamma}\left(\frac{1}{2}|w_t|^2+\frac{1}{2}|\Delta w|^2+\frac{\alpha_2}{2}|\nabla w_t|^2+\frac{1}{2}|\theta|^2\right)(t)+\int_0^t\int_{\Gamma}(\alpha_1|\nabla w_t|^2+|\nabla \theta|^2)\\
    &\leq \int_{\Omega_w(0)}\left(\frac{1}{2}\frac{|(\varrho\bu)_0|^2}{\varrho_0}+\frac{1}{2}b_0^2+\varrho_0e(\varrho_0,\vartheta_0)\right)\notag\\
&\hspace{0.5cm}+\int_{\Gamma}\left(\frac{1}{2}|v_0|^2+\frac{1}{2}|\Delta w_0|^2+\frac{\alpha_2}{2}|\nabla v_0|^2+\frac{1}{2}|\theta_0|^2\right), \notag
\end{align}
holds for all $t\in [0,T]$. 
\end{itemize}
\end{definition}

\begin{remark}
 In contrast to \cite{FN-09}, we consider energy inequality rather than energy equality, which is sufficient for our definition of weak solution. In fact,  if
the above defined weak solution is smooth enough, it will be a strong one. This fact has been justified in  \cite[Section 1.2]{paer-Poul}, see also the book \cite{Book-Sarka}. 
\end{remark}


\section{Corresponding problem in a fixed domain}\label{Section-fixed-domain} 

Based on the energy estimate $\eqref{energy-ine}$, it is revealed that the displacement $w$ is bounded. Thus, we can choose a sufficient large $M_0>0$ such that $\Omega_w(t)\subseteq \mathsf{B}:=\{|x|<M_0\}$ for all $t\in [0,T]$. Then, by using  appropriate approximations of the initial data and transport coefficients, we extend the coupled system to the fixed domain $\mathsf{B}$.

\subsection{Extension of initial data and transport coefficients}\label{Section-extension-parameters}
For the given displacement $w$ and a positive constant $\omega>0$, the shear and bulk viscosity coefficients $\mu(\vartheta)$ and $\eta(\vartheta)$ are approximated by
\begin{equation}\label{vis-app}
\mu_w^{\omega}:=\mu(\vartheta)g_w^{\omega},\hspace{0.3cm}\eta_w^{\omega}:=\eta(\vartheta)g_w^{\omega} , 
\end{equation}
where $g_w^{\omega}$ is a smooth function depending on the displacement $w$ and satisfying 
\begin{equation}\label{g-app}
  \begin{aligned}
&g_w^{\omega}(t, x)\in C_c^\infty([0,T]\times\mathbb{R}^2), \hspace{0.3cm} 0<\omega <g_w^{\omega}(t, x)\leq 1\hspace{0.1cm}\mathrm{in}\hspace{0.1cm} [0,T]\times\mathsf{B},\\
&g_w^{\omega}(t, x)|_{\Omega_w(t)}=1, \hspace{0.1cm} \forall t\in[0,T], \hspace{0.3cm}\|g_w^{\omega}\|_{L^p(([0,T]\times\mathsf{B})\setminus Q_T^w)}\leq C\omega, {\hspace{0.1cm}\mathrm{for}\hspace{0.1cm} \mathrm{some} \hspace{0.1cm} p\geq \frac{5}{3}}.
\end{aligned}
\end{equation}

Similarly, for positive constant $\zeta>0$, the conductivity coefficient is approximated as 
\begin{equation}\label{con-app}
\kappa_w^{\zeta}:=\kappa(\vartheta)h_w^{\zeta},
\end{equation}
where 
\begin{equation}\label{h-app}
    h_w^{\zeta}=1 \hspace{0.1cm}\mathrm{in} \hspace{0.1cm} Q_T^w, \hspace{0.3cm} h_w^{\zeta}=\zeta \hspace{0.1cm} \mathrm{in} \hspace{0.1cm} ([0,T]\times\mathsf{B})\setminus Q_T^w.
\end{equation}

Furthermore, the positive constant $a$ in the pressure and internal energy (see $\eqref{state-equ}$) is approximated as
\begin{equation}\label{a-app}
a_{w}^{\lambda}:= af_w^{\lambda},
\end{equation}
where 
\begin{equation}\label{a-app-f}
    f_w^{\lambda}=1 \hspace{0.1cm}\mathrm{in} \hspace{0.1cm} Q_T^w, \hspace{0.3cm} f_w^{\lambda}=\lambda \hspace{0.1cm} \mathrm{in} \hspace{0.1cm} ([0,T]\times\mathsf{B})\setminus Q_T^w.
\end{equation}
With this, 
the pressure  term is approximated as follows
\begin{align}
p_w^{\lambda, \delta} := p_M(\varrho, \vartheta) + \frac{a^\lambda_w}{3}\vartheta^4+ \delta (\varrho+b)^\beta, \ \ \delta > 0, \ \beta \geq \max\{4,\gamma\} ,  
\end{align}
and the internal energy and specific entropy are approximated as 
\begin{align}\label{entropy-w-lambda}
e_w^{\lambda}: = e_{M}(\varrho, \vartheta) + \frac{a^\lambda_w}{\varrho} \vartheta^4, \ \ s^\lambda_w:= s_{M}(\varrho, \vartheta) + \frac{4a^\lambda_w}{3}\frac{\vartheta^3}{\varrho}. 
\end{align}

The initial data $(\varrho_0, b_0, (\varrho\bu)_0, \vartheta_0)$ in \eqref{ini-data-1} will be approximated using the parameter $\delta$ and a sufficiently large parameter $\beta \geq  \max\{4,\gamma\}$, satisfying 
  \begin{align}\label{ini-app}
&\varrho_{0,\delta}, \hspace{0.3cm}\varrho_{0,\delta}\not\equiv 0, \hspace{0.3cm} \varrho_{0,\delta}|_{\mathbb{R}^2\setminus{\Omega_w(0)}}=0, \hspace{0.3cm}\int_{\mathsf{B}}(\varrho_{0,\delta}^{\gamma}+\delta\varrho_{0,\delta}^\beta)\leq C,\notag\\
&\varrho_{0,\delta}\rightarrow \varrho_0 \hspace{0.1cm}\mathrm{in}\hspace{0.1cm} L^{\gamma}(\mathsf{B}), \hspace{0.3cm}|\{\varrho_{0,\delta}<\varrho_0\}|\rightarrow 0\hspace{0.3cm} \mathrm{as}\hspace{0.2cm} \delta \rightarrow 0,\notag\\
&{b_{0,\delta}, \hspace{0.3cm}b_{0,\delta}\not\equiv 0, \hspace{0.3cm} b_{0,\delta}|_{\mathbb{R}^2\setminus{\Omega_w(0)}}=0, \hspace{0.3cm}\int_{\mathsf{B}}(b_{0,\delta}^2+\delta b_{0,\delta}^\beta)\leq C,}\\
&b_{0,\delta}\rightarrow b_0 \hspace{0.1cm}\mathrm{in}\hspace{0.1cm} L^{2}(\mathsf{B}), \hspace{0.3cm}|\{b_{0,\delta}<b_0\}|\rightarrow 0 \hspace{0.3cm} \mathrm{as}\hspace{0.2cm} \delta \rightarrow 0,\notag\\
&\vartheta_{0,\delta}\geq \underline{\vartheta}\geq 0, \hspace{0.3cm}\vartheta_{0,\delta}\in L^\infty(\mathsf{B})\cap C^{2+\nu_0}(\mathsf{B}), \, \text{ for some } \nu_0>0,\notag
\end{align}
\begin{equation}\label{mom-app}
  (\varrho\bu)_{0,\delta}=\left\{
    \begin{aligned}
&(\varrho\bu)_0,\hspace{0.1cm} \mathrm{when}\hspace{0.1cm} \varrho_{0,\delta}\geq \varrho_0,\\
&0, \hspace{0.1cm}\mathrm{otherwise}, 
\end{aligned}
\right.
\hspace{0.3cm} \int_{\mathsf{B}}\frac{1}{\varrho_{0,\delta}}|(\varrho\bu)_{0,\delta}|^2\leq C, 
\end{equation}
\begin{equation}\label{ene-entropy-app}
  \begin{aligned}
   & \varrho_{0,\delta}e(\varrho_{0,\delta},\vartheta_{0,\delta})\rightarrow   \varrho_{0}e(\varrho_{0},\vartheta_{0}) \hspace{0.1cm}\mathrm{in}\hspace{0.1cm} L^1(\Omega_w(0)), \\
&\varrho_{0,\delta}s(\varrho_{0,\delta},\vartheta_{0,\delta})\rightharpoonup  \varrho_{0}s(\varrho_{0},\vartheta_{0}) \hspace{0.1cm}\mathrm{in}\hspace{0.1cm} L^1(\Omega_w(0)).
\end{aligned}
\end{equation}

Now we give the precise definition of the weak solution to the fluid-structure system $\eqref{shell}$, $\eqref{full-MHD}$ with $\eqref{VT-con}$-$\eqref{F-Q-con}$ and $\eqref{ini-data}$-$\eqref{ini-data-1}$, when the time-varying domain $\Omega_w(t)$ is extended to the fixed domain $\mathsf{B}$.

\subsection{Weak formulations in the reference domain} 

Let us now write the weak formulations in the extended fixed domain $(0,T)\times \mathsf{B}$. 
\begin{definition}\label{weak-solutions-on-fixed-domain}
  We say that $(\varrho, b, \bu, \vartheta, w, \theta)$ is a weak solution to the fluid-structure system $\eqref{shell}$, $\eqref{full-MHD}$ in the fixed domain  $(0,T)\times \mathsf{B}$ with the initial data $(\varrho_0, b_0, (\varrho\bu)_0, \vartheta_0, (\varrho s)_0, w_0, v_0, \theta_0)$ satisfying $\eqref{ini-data}$-$\eqref{ini-data-1}$ if the following items hold: 
  
  \begin{itemize}
\item Function spaces:
unknowns $(\varrho, b, \bu, \vartheta)$ satisfy that 
\begin{equation}\label{fun-spa-B}
	\left\{\begin{aligned}
&\varrho \geq 0,\hspace{0.3cm} \vartheta >0\hspace{0.3cm} \mathrm{a.e.\hspace{0.1cm} in}\hspace{0.2cm} \mathsf{B} ,\\
&\varrho \in L^\infty(0,T; L^\gamma(\mathsf{B})), \\
&b \in L^\infty(0,T; L^2(\mathsf{B})), \\
&\bu \in L^2(0,T;W^{1,p}(\mathsf{B})) \hspace{0.1cm}{\mathrm{with}} \hspace{0.1cm} 1\leq p \leq 2, \hspace{0.3cm} \varrho|\bu|^2\in L^\infty(0,T;L^1(\mathsf{B})),\\
&\vartheta \in L^\infty (0,T;L^4(\mathsf{B})), \hspace{0.3cm}\vartheta, \nabla\vartheta, \log\vartheta, \nabla \log\vartheta \in L^2((0,T)\times \mathsf{B}),\\
&\varrho s, \varrho s\bu, \frac{\bq}{\vartheta} \in  L^1((0,T)\times \mathsf{B}) ;
\end{aligned}
\right.
\end{equation}
unknowns $(w, \theta)$ satisfy that 
\begin{equation}\label{fun-spa-2-B}
	\left\{\begin{aligned}
& w \in L^\infty (0,T;H^2(\Gamma))\cap W^{1,\infty}(0,T; L^2(\Gamma)), \\
& \sqrt{\alpha_1}w\in H^1(0,T;H^1(\Gamma)),\hspace{0.3cm} \sqrt{\alpha_2} w \in W^{1,\infty}(0,T;H^1(\Gamma)), \\
& \theta >0\hspace{0.1cm} {\mathrm{a.e.}}\hspace{0.1cm}\mathrm{in}\hspace{0.1cm} \Gamma_T, \hspace{0.3cm} \theta \in L^\infty(0,T;L^2(\Gamma))\cap L^2(0,T;H^1(\Gamma)),\\
&\log \theta \in L^2(0,T;H^s(\Gamma)), \hspace{0.1cm} \mathrm{with} \hspace{0.1cm}0< s\leq \frac{1}{2}.
  \end{aligned}
\right.
\end{equation}

\item The continuity on the interface $\Gamma^w_T$: $w_t \bfn=\bu|_{\Gamma_{w}(t)}$ and $\theta=\vartheta|_{\Gamma_{w}(t)}$ for all $t\in(0,T)$.

\item The continuity equation
\begin{equation}\label{weak-con-B}
\int_0^T\int_{\mathsf{B}}\left(\varrho\partial_t\phi+\varrho\bu\cdot\nabla\phi\right)=-\int_{\mathsf{B}}\varrho_0\phi(0,\cdot)
\end{equation}
holds for all $\phi\in C^1_c([0,T)\times\mathsf{B} )$.

Moreover, the continuity equation holds true in the renormalized sense:
\begin{align}\label{renor-weak-con-B}
\int_0^T \int_{\mathsf{B}} \varrho A(\varrho) \left(\partial_t \varrho +\bu \cdot \nabla \phi  \right) = \int_0^T \int_B a(\varrho) \rmdiv \bu \,  \phi +\int_{\mathsf{B}} \varrho_{0,\delta} A(\rho_{0,\delta}) \phi(0,\cdot),
\end{align}
holds for all $\phi \in C_c^\infty([0,T)\times \mathsf{B})$ and any $a\in L^\infty (0,\infty)\cap C([0,\infty))$ such that $a(0)=0$ with $A(\varrho)=A(1)+\int_1^\varrho \frac{a(z)}{z^2}dz$.

\item The non-resistive magnetic equation
\begin{equation}\label{mag-equ-B}
\int_0^T\int_{\mathsf{B}}\left(b\partial_t\phi+b\bu\cdot\nabla\phi\right)=-\int_{\mathsf{B}}b_0\phi(0,\cdot)
 \end{equation}
 holds for all $\phi\in C^1_c([0,T)\times \mathsf{B})$.

Also, the renormalized equation associated to  \eqref{mag-equ-B} holds in $(0,T)\times B$.  

\item The coupled momentum equations
  \begin{align}\label{mom-equ-B}
 &\int_0^T\int_{\mathsf{B}}\varrho\bu \cdot\partial_t\bmphi+\int_0^T\int_{\mathsf{B}}\left(\varrho\bu\otimes\bu+p_w^\lambda(\varrho,\vartheta)\mathbb{I}_2+{\frac{1}{2}b^2\mathbb{I}_2+\delta (\varrho+b)^\beta\mathbb{I}_2}-\bbS_w^{\omega}(\vartheta,\nabla\bu)\right): \nabla\bmphi\notag\\
 &+\iint_{\Gamma_T}(w_t\psi_t-\Delta w\Delta\psi-\alpha_1\nabla w_t\cdot\nabla\psi +\alpha_2\nabla w_t\cdot\nabla\psi_t+\nabla\theta\cdot\nabla\psi)\\
 &=-\int_{\mathsf{B}}(\varrho\bu)_0\bmphi(0,\cdot)-\int_{\Gamma}v_0\psi(0,\cdot)-\alpha_2\int_{\Gamma}\nabla v_0\cdot\nabla\psi(0,\cdot), \notag
 \end{align}
holds for all $\bmphi\in C^1_c([0,T)\times\mathsf{B}) $ and $\psi\in C_c^\infty([0,T)\times \Gamma)$ such that $\bmphi|_{\Gamma_w(t)}=\psi \bfn$ on $\Gamma_T$, and 
\begin{align*}
&p_w^{\lambda}(\varrho, \vartheta):=p_M(\varrho,\vartheta)+\frac{a^\lambda_{w}}{3}\vartheta^4, \\
&\bbS_w^{\omega}(\vartheta,\nabla\bu)=\mu_{w}^{\omega}(\vartheta)\left(\nabla\bu+\nabla^\top\bu-\rmdiv\bu\mathbb{I}_2\right)+\eta_{w}^{\omega}(\vartheta)\rmdiv\bu\mathbb{I}_2.
\end{align*}

\item The coupled entropy balance
  \begin{align}\label{entropy-ine-B}
&\int_0^T\int_{\mathsf{B}}\left(\varrho s_w^{\lambda}\vphi_t+\varrho s_w^{\lambda} \bu\cdot\nabla\vphi-\frac{\kappa^{\zeta}_{w}(\vartheta)\nabla\vartheta}{\vartheta}\cdot\nabla\vphi\right)+\langle \mathcal{D}_{w}^{\omega,\zeta};\vphi\rangle_{[\mathcal{M};C]([0,T]\times\overline{\mathsf{B}})} \notag\\
&+\xi \int_0^T\int_{\mathsf{B}}\vartheta^4\vphi+\iint_{\Gamma_T}(\theta\tilde\varphi_t-\nabla\theta\cdot\nabla\tilde\varphi+\nabla w\cdot\nabla\tilde\varphi_t)\\
&= -\int_{\mathsf{B}} \varrho_0 s_w^{\lambda}(\varrho_0,\vartheta_0)\vphi(0,\cdot)-\int_{\Gamma}\theta_0\tilde\varphi(0,\cdot)-\int_{\Gamma}\nabla w_0\cdot\nabla\tilde\varphi(0,\cdot), \notag
\end{align}
holds for all non-negative function $\vphi\in C^1_c([0,T)\times\mathsf{B}) $ and $\tilde\varphi \in C_c^\infty([0,T)\times \Gamma)$ such that $\vphi|_{\Gamma_w(t)}=\tilde\vphi$ on $\Gamma_T$ and $\mathcal{D}_{w}^{\omega,\zeta}$ satisfying 
\begin{align}\label{measure-D}
    \mathcal{D}_{w}^{\omega,\zeta} \geq \frac{1}{\vartheta}\left(\bbS_w^{\omega}(\vartheta,\nabla\bu):\nabla\bu+\frac{\kappa^{\zeta}_{w}(\vartheta)|\nabla\vartheta|^2}{\vartheta}\right).
\end{align}

In particular, 
\begin{align*} 
s_w^{\lambda}=s_w^{\lambda}(\varrho,\vartheta)=s_M(\varrho, \vartheta)+\frac{4a^\lambda_{w}}{3}\frac{\vartheta^3}{\varrho}.
\end{align*}

\item The coupled energy inequality
  \begin{align}\label{energy-ine-B}
    &\int_{\mathsf{B}}\bigg[\frac{1}{2}\varrho|\bu|^2+\frac{1}{2}b^2+\varrho e_{w}^{\lambda}(\varrho, \vartheta)+\frac{\delta}{\beta-1}(\varrho+b)^\beta\bigg](t)\notag\\
    &+\int_{\Gamma}\left(\frac{1-\delta}{2}|w_t|^2+\frac{1}{2}|\Delta w|^2+\frac{\alpha_2}{2}|\nabla w_t|^2+\frac{1-\delta}{2}|\theta|^2\right)(t)\notag\\
&+\xi\int_0^t\int_{\mathsf{B}}\vartheta^5+\int_0^t\int_{\Gamma}(\alpha_1|\nabla w_t|^2+|\nabla \theta|^2)\\
    &\leq \int_{\mathsf{B}}\left[\frac{1}{2}\frac{|(\varrho\bu)_{0,\delta}|^2}{\varrho_{0,\delta}}+\frac{1}{2}b_{0,\delta}^2+\varrho_{0,\delta}e_{w}^{\lambda}(\varrho_{0,\delta},\vartheta_{0,\delta})+\frac{\delta}{\beta-1}(\varrho_{0,\delta}+b_{0,\delta})^\beta\right]\notag\\
    &\hspace{0.5cm}+\int_{\Gamma}\left(\frac{1}{2}|v_0|^2+\frac{1}{2}|\Delta w_0|^2+\frac{\alpha_2}{2}|\nabla v_0|^2+\frac{1}{2}|\theta_0|^2\right), \notag
\end{align}
holds for all $t\in [0,T]$. 
  \end{itemize}
\end{definition}

Motivated by the results on moving domains in \cite{Kreml-jmpa-time-dep,KMNW-18}, we have extended the corresponding problem to the fixed domain $\mathsf{B}$. However, it is important to note that, in our case, the displacement $w(t,\bmphi(y))$ is unknown, which distinguishes our work from \cite{Kreml-jmpa-time-dep,KMNW-18}. Indeed, the system in the reference domain remains coupled and depends on the geometry through the {\em continuity condition on the interface} $\Gamma_T$. Consequently, it is not straightforward to decouple the system and solve the fluid part independently. To address this, we employ a decoupling approach based on the``operator splitting method" from \cite{Srjdan-Wang} (see also \cite{BCGTQ-13,MC-13,GGCC-09}, where the splitting method was introduced in the context of FSI) that penalizes the fluid velocity and temperature to enforce the kinematic coupling conditions. Specifically, the time interval $(0,T)$ is divided into $N$ subintervals, each of length $\Delta t = T/N$. The approximate solution is constructed via a time-marching procedure, where, in each time subinterval, we independently solve the fluid and structure subproblems. Furthermore, penalization terms are included to ensure the appropriate exchange of information across the interface $\Gamma_T^w$. As $\Delta t \to 0$, the kinematic coupling conditions are satisfied, and the interface becomes impermeable; therefore, we obtain a weak solution to the extended problem, as defined in Definition \ref{weak-solutions-on-fixed-domain}.

\subsection{The splitting scheme}

We split the time interval $(0,T)$ into $N \in \mathbb{N}$ subintervals of length $\Delta t = \frac{T}{N}$, and define $t_n = n \Delta t$ such that
\begin{equation}\label{interval}
(0,T)=\bigcup_{n=0}^{N-1}(t^n,t^{n+1}) 
\end{equation}
We then decompose the fluid-structure system $\eqref{shell}$, $\eqref{full-MHD}$ on the fixed domain $\mathsf{B}$ into two subproblems: one for the fluid and the other for the structure, solving the semi-discretized problems for both at each time step. This approach has proven to be a promising method for addressing coupled fluid-structure interaction systems (see \cite{CGM-20, MC-16-2, MC-16, MS-22}).
 
The system is decomposed in the coupled momentum equation \eqref{mom-equ-B} and the coupled entropy inequality \eqref{entropy-ine-B}. It is important to note that the kinematic condition $\eqref{VT-con}$ cannot be directly guaranteed. To address this, we introduce two unknowns, $\mathbf{v}$ and $\tilde{\vartheta}$, representing the traces of the fluid velocity and temperature on the interface $\Gamma_w(t)$. Specifically, $\mathbf{v} := \mathbf{u}|_{\Gamma_w(t)}$ and $\tilde{\vartheta} := \vartheta|_{\Gamma_w(t)}$.

It is important to note that the kinematic coupling conditions are not satisfied at the level of approximate solutions; that is, in general, $\mathbf{v} \neq \partial_t w \mathbf{n}$ and $\tilde{\vartheta} \neq \theta$. However, the penalization terms (which will be included in the decoupled equations) will ensure that $\eqref{VT-con}$ holds as $\Delta t \to 0$.

\subsection{The sub-problems} 
Before presenting the approximation scheme of the sub-problems for fluid and structure, we introduce the time-shift function denoted by $\tau_{\Delta t}f(t)$ for any function $f(t)$ as follows:
\begin{equation}\label{time-shift}
\tau_{\Delta t}f(t): =\left\{
\begin{aligned}
&f(t-\Delta t), \quad t\in [(n-1)\Delta t, n\Delta t], \ n\geq 1,\\
&f(0), \quad t\in[0,\Delta t].
\end{aligned}\right.
\end{equation}


\subsubsection{\bf{The sub-problem of structure}} The  structure sub-problem in $[n\Delta t, (n+1)\Delta t]\times \Gamma$ has three parts: 
\begin{itemize}
\item[(i)] the structure part of the momentum equation \eqref{mom-equ-B};
\item[(ii)] the structure part of the entropy balance \eqref{entropy-ine-B}; 
\item[(iii)] the structure part of the coupled energy inequality \eqref{energy-ine-B}.  
\end{itemize}

We now use induction process on $n\geq 0$. \\

 {\underline{\em Case $n=0$}}: Suppose that 
\begin{equation}\label{caae0}
  \begin{split}
   & w^0(0,\cdot):=w_0,\hspace{0.3cm}w^0_t(0,\cdot):=v_0, \hspace{0.3cm}\theta^0(0,\cdot):=\theta_0,\\
&\mathrm{and} \quad \bfv^0(t,\cdot):=v_0\bfn, \hspace{0.3cm} \tilde{\vartheta}^0(t,\cdot):=\theta_0,\hspace{0.3cm} \mathrm{for}\hspace{0.1cm} t\in[-\Delta t, 0] .
\end{split}
\end{equation}

\underline{\em Case $n \geq 1$}: Assume that the solution $(w^n, \theta^n)$ to structure sub-problem and the solution  \\
$(\varrho^n, b^n, \bu^n, \vartheta^n)$ to fluid sub-problem are already derived. Then we want to find $(w^{n+1}, \theta^{n+1})$ satisfying:
\begin{itemize}
  \item  Function space:
  \begin{equation}\label{dis-function-space}
    \left\{\begin{aligned}
&w^{n+1}\in W^{1,\infty}(n\Delta t, (n+1)\Delta t; L^2(\Gamma))\cap L^\infty(n\Delta t, (n+1)\Delta t; H^2(\Gamma)),\\
&\sqrt{\alpha_1}w^{n+1}\in H^1(n\Delta t,(n+1)\Delta t;H^1(\Gamma)),\\
&\sqrt{\alpha_2} w^{n+1}\in W^{1,\infty}(n\Delta t,(n+1)\Delta t;H^1(\Gamma)),\\
&\theta^{n+1}\in L^\infty(n\Delta t,(n+1)\Delta t; L^2(\Gamma))\cap L^2(n\Delta t,(n+1)\Delta t; H^1(\Gamma)).
  \end{aligned}
  \right.
\end{equation}

\item It holds that 
\begin{equation*}
  w^{n+1}(n\Delta t, \cdot)=w^n(n\Delta t, \cdot), \quad  w_t^{n+1}(n\Delta t, \cdot)=w_t^n(n\Delta t, \cdot),  \quad \theta^{n+1}(n\Delta t, \cdot)=\theta^n(n\Delta t, \cdot),
\end{equation*}
in the sense of weakly continuous  in time. 

\item  The weak formulation for the plate equation is
\begin{align}\label{weak-plate-function}
  &(1-\delta)\int_{n\Delta t}^{(n+1)\Delta t}\int_\Gamma w^{n+1}_t\psi_t-\delta \int_{n\Delta t}^{(n+1)\Delta t}\int_\Gamma \frac{w_t^{n+1}-\tau_{\Delta t}\bfv^{n}\cdot\bfn}{\Delta t}\psi\notag\\
  &-\int_{n\Delta t}^{(n+1)\Delta t}\int_{\Gamma}(\Delta w^{n+1}\Delta \psi +\alpha_1\nabla w_t^{n+1}\cdot\nabla\psi)\\
  &+\int_{n\Delta t}^{(n+1)\Delta t}\int_{\Gamma}(\alpha_2\nabla w^{n+1}_t\cdot\nabla\psi_t+\nabla\theta^{n+1}\cdot\nabla\psi)\notag\\
  &=(1-\delta)\int_{n\Delta t}^{(n+1)\Delta t}\frac{\mathrm{d}}{\mathrm{d}t}\int_{\Gamma}w_t^{n+1}\psi+\alpha_2\int_{n\Delta t}^{(n+1)\Delta t}\frac{\mathrm{d}}{\mathrm{d}t}\int_{\Gamma}\nabla w^{n+1}_t\cdot\nabla\psi,\notag
\end{align}
holds for all $\psi \in C^\infty([n\Delta t, (n+1)\Delta t]\times\Gamma)$. 

\item  The weak formulation for the heat equation is 
  \begin{align}\label{weak-heat-function}
&(1-\delta)\int_{n\Delta t}^{(n+1)\Delta t}\int_{\Gamma}\theta^{n+1}\tilde \vphi_t-\delta\int_{n\Delta t}^{(n+1)\Delta t}\int_{\Gamma}\frac{\theta^{n+1}-\tau_{\Delta t}\tilde{\vartheta}^n}{\Delta t}\tilde \vphi \notag\\
&-\int_{n\Delta t}^{(n+1)\Delta t}\int_{\Gamma}(\nabla \theta^{n+1}\cdot\nabla \tilde \vphi+\nabla w^{n+1}_t\cdot\nabla\tilde \vphi)\\
&=(1-\delta)\int_{n\Delta t}^{(n+1)\Delta t}\frac{\mathrm{d}}{\mathrm{d}t}\int_{\Gamma}\theta^{n+1}\tilde \vphi, \notag
\end{align}
holds for all $\tilde \vphi\in C^\infty([n\Delta t, (n+1)\Delta t]\times\Gamma)$. 

\item The energy inequality for structure is
  \begin{align}\label{weak-energy-structure}
    &\frac{\delta}{2\Delta t}\int_{n\Delta t}^t\left(\|w^{n+1}_t-\tau_{\Delta t}\bfv^{n}\cdot\bfn\|^2_{L^2(\Gamma)}+\|w^{n+1}_t\|_{L^2(\Gamma)}\right)+\frac{1-\delta}{2}\|w_t^{n+1}\|^2_{L^2(\Gamma)}(t)\notag\\
    &+\frac{\delta}{2\Delta t}\int_{n\Delta t}^t\left(\|\theta^{n+1}-\tau_{\Delta t}\tilde{\vartheta}^{n}\|^2_{L^2(\Gamma)}+\|\theta^{n+1}\|_{L^2(\Gamma)}\right)+\frac{1-\delta}{2}\|\theta^{n+1}\|^2_{L^2(\Gamma)}(t)\notag\\
    &+\frac{1}{2}\|\Delta w^{n+1}\|^2_{L^2(\Gamma)}(t)+\frac{\alpha_2}{2}\|\nabla w_t^{n+1}\|^2_{L^2(\Gamma)}(t)\\
    &+\int_{n\Delta t}^t\left(\alpha_1\|\nabla w_t^{n+1}\|^2_{L^2(\Gamma)}+\|\nabla\theta^{n+1}\|^2_{L^2(\Gamma)}\right)\notag\\
    &\leq \frac{1-\delta}{2}\|w_t^{n+1}\|^2_{L^2(\Gamma)}(n\Delta t)+\frac{1}{2}\|\Delta w^{n+1}\|^2_{L^2(\Gamma)}(n\Delta t)+\frac{\alpha_2}{2}\|\nabla w_t^{n+1}\|^2_{L^2(\Gamma)}(n\Delta t)\notag\\
    &\hspace{0.5cm}+\frac{1-\delta}{2}\|\theta^{n+1}\|^2_{L^2(\Gamma)}(n\Delta t)+\frac{\delta}{2\Delta t}\int_{n\Delta t}^t\left(\|\tau_{\Delta t}\bfv ^{n}\|_{L^2(\Gamma)}^2+\|\tau_{\Delta t}\tilde{\vartheta} ^{n}\|_{L^2(\Gamma)}^2\right),\notag
\end{align}
holds for all $t\in (n\Delta t, (n+1)\Delta t]$.
\end{itemize}

\subsubsection{\bf{The sub-problem of fluid}}
The  fluid  sub-problem in $[n\Delta t, (n+1)\Delta t]\times \mathsf{B}$ has five  parts: 
\begin{itemize}
\item[(i)] the renormalized continuity equation \eqref{renor-weak-con-B};
\item[(ii)] the magnetic equation \eqref{mag-equ-B} (in the sense of renormalized sense);  
\item[(iii)] the fluid part of the momentum equation \eqref{mom-equ-B};
\item[(iv)] the fluid part of the entropy balance \eqref{entropy-ine-B}; 
\item[(v)] the fluid  part of the coupled energy inequality \eqref{energy-ine-B}.  
\end{itemize}

\vspace*{.2cm}
We again use the induction process  on $n\geq 0$. \\

\underline{\em Case $n=0$}: We suppose 
  \begin{align*}
\varrho^0(0,\cdot):=\varrho_{0,\delta}(\cdot), \hspace{0.2cm}b^0(0,\cdot):=b_{0,\delta}(\cdot),\hspace{0.2cm} (\varrho\bu)^0(0,\cdot):=(\varrho\bu)_{0,\delta}(\cdot),\hspace{0.2cm}(\varrho s(\varrho,\vartheta))^0(0,\cdot):=\varrho_{0,\delta}s(\varrho_{0,\delta}, \vartheta_{0,\delta}) .
\end{align*}

\underline{\em Case $n\geq 1$}: Suppose that the solution $(w^{n+1}, \theta^{n+1})$ to structure sub-problem and solution 
$ ( \varrho^n, b^n, \bu^n, \vartheta^n ) $ to fluid sub-problem  are already derived. Then our aim is to find $(\varrho^{n+1}, b^{n+1}, \bu^{n+1}, \vartheta^{n+1})$ satisfying: 
\begin{itemize}
  \item Function space:
\begin{equation*}
  \left\{\begin{aligned}
&\varrho^{n+1}\geq  0, \hspace{0.3cm}\varrho^{n+1}\in L^{\infty}(n\Delta t, (n+1)\Delta t; L^\beta(\mathsf{B})),\\
&b^{n+1}\geq  0,\hspace{0.3cm} b^{n+1}\in L^{\infty}(n\Delta t, (n+1)\Delta t; L^\beta(\mathsf{B})),\\
&\bu^{n+1}, \nabla\bu^{n+1} \in L^2([n\Delta t,(n+1)\Delta t]\times\mathsf{B}), \hspace{0.3cm} \varrho^{n+1}\bu^{n+1}\in L^\infty(n\Delta t,(n+1)\Delta t;L^1(\mathsf{B})),\\
&\vartheta^{n+1}>0, \hspace{0.1cm}\mathrm{in}\hspace{0.1cm} [n\Delta t,(n+1)\Delta t]\times\mathsf{B}, \hspace{0.3cm} \vartheta^{n+1}\in L^\infty(n\Delta t, (n+1)\Delta t; L^4(\mathsf{B})),\\
&\vartheta^{n+1}, \nabla\vartheta^{n+1}\in L^2([n\Delta t,(n+1)\Delta t]\times\mathsf{B}), \hspace{0.3cm}\log\vartheta^{n+1}, \nabla\log\vartheta^{n+1} \in L^2([n\Delta t,(n+1)\Delta t]\times\mathsf{B}),\\
&\varrho^{n+1}s(\varrho^{n+1},\vartheta^{n+1}), \varrho^{n+1}s(\varrho^{n+1},\vartheta^{n+1})\bu^{n+1}, \frac{\bq(\vartheta^{n+1})}{\vartheta^{n+1}}\in L^1([n\Delta t,(n+1)\Delta t]\times\mathsf{B}).
\end{aligned}
\right.
\end{equation*}

\item It holds that  
\begin{align*}
  \varrho^{n+1}(n\Delta t)=\varrho^n(n\Delta t), \ b^{n+1}(n\Delta t)=b^n(n\Delta t), \ \varrho^{n+1}\bu^{n+1}(n\Delta t)=\varrho^n\bu^n(n\Delta t),
\end{align*}
in the sense of  weakly continuous  in time.

\item The continuity equation 
\begin{align}\label{weak-continuity}
\int_{n\Delta t}^{(n+1)\Delta t}\int_{\mathsf{B}}\varrho^{n+1}(\phi_t+\bu^{n+1}\cdot\nabla\phi)=-\int_{n\Delta t}^{(n+1)\Delta t}\frac{\mathrm{d}}{\mathrm{d}t}\int_{\mathsf{B}}\varrho^{n+1}\phi, 
\end{align}
holds for all $\phi\in C^\infty ([n\Delta t,(n+1)\Delta t]\times\mathsf{B})$. 
                                       
Here, the renormalized continuity equation takes the form:
\begin{align}\label{weak-renor-cont-ind} 
&\int_{n\Delta t}^{(n+1)\Delta t} \int_{\mathsf{B}} \varrho^{n+1} A(\varrho^{n+1})( \partial_t \phi +\bu^{n+1}\cdot \nabla \phi) \nonumber \\
&= \int_{n\Delta t}^{(n+1)\Delta t} \int_{\mathsf{B}} a(\varrho^{n+1}) \rmdiv \bu^{n+1} \phi - \int_{n\Delta t}^{(n+1)\Delta t}\frac{d}{dt}\int_{\mathsf{B}} \varrho^{n+1} A(\rho^{n+1}) \phi , 
\end{align}
holds for all $\phi \in C^\infty([n\Delta t,(n+1)\Delta t]\times \mathbb{R}^3)$ and any $a\in L^{\infty}(0,\infty) \cap C([0,\infty))$ such that $a(0)=0$ with $A(x)=A(1) + \int_1^x \frac{a(z)}{z^2}dz$.

\item The non-resistive magnetic equation 
\begin{align}\label{weak-magnetic-induction}
  \int_{n\Delta t}^{(n+1)\Delta t}\int_{\mathsf{B}}b^{n+1}(\phi_t+\bu^{n+1}\cdot\nabla\phi)=-\int_{n\Delta t}^{(n+1)\Delta t}\frac{\mathrm{d}}{\mathrm{d}t}\int_{\mathsf{B}}b^{n+1}\phi, 
  \end{align}
  holds for all $\phi\in C^\infty ([n\Delta t,(n+1)\Delta t]\times\mathsf{B})$. 

The renormalized equation also holds true for \eqref{weak-magnetic-induction} (similar to  the continuity equation \eqref{weak-renor-cont-ind}).  

  \item The momentum equations 
    \begin{align}\label{sub-fluid-momentum-equ}
    &\int_{n\Delta t}^{(n+1)\Delta t}\int_{\mathsf{B}}\varrho^{n+1}\bu^{n+1}\cdot\bmphi_t+ \int_{n\Delta t}^{(n+1)\Delta t}\int_{\mathsf{B}}\bigg(\varrho^{n+1}\bu^{n+1}\otimes\bu^{n+1}
   \notag \\
    &+p_{w^{n+1}}^{\lambda}(\varrho^{n+1},\vartheta^{n+1})\mathbb{I}_2+\frac{1}{2}(b^{n+1})^2\mathbb{I}_2+\delta (\varrho^{n+1}+b^{n+1})^\beta\mathbb{I}_2-\bbS_{w^{n+1}}^{\omega}(\vartheta^{n+1}, \nabla\bu^{n+1})\bigg):\nabla\bmphi\notag \\
    &-\delta \int_{n\Delta t}^{(n+1)\Delta t}\int_{\Gamma}\frac{\bfv^{n+1}-w_t^{n+1}\bfn}{\Delta t}\cdot\bmpsi\\
    &=\int_{n\Delta t}^{(n+1)\Delta t}\frac{\mathrm{d}}{\mathrm{d}t}\int_{\mathsf{B}}\varrho^{n+1}\bu^{n+1}\cdot\bmphi,\notag 
    \end{align}
holds for all $\bmphi \in C^\infty([n\Delta t,(n+1)\Delta t]\times\mathsf{B})$ and $\bmpsi\in C^\infty([n\Delta t,(n+1)\Delta t]\times\Gamma)$ such that $\bmphi|_{\Gamma_{w^{n+1}}}=\bmpsi$, where 
  \begin{align*}
&\bfv^{n+1}:=\bu^{n+1}|_{\Gamma_{w^{n+1}}}, \quad p_{w^{n+1}}^{\lambda}(\varrho^{n+1}, \vartheta^{n+1}):=p_M(\varrho^{n+1},\vartheta^{n+1})+\frac{a^\lambda_{w^{n+1}}}{3}(\vartheta^{n+1})^4, \\
&\bbS_{w^{n+1}}^{\omega}(\vartheta^{n+1},\nabla\bu^{n+1}):=\mu_{w^{n+1}}^{\omega}(\vartheta^{n+1})\left(\nabla\bu^{n+1}+\nabla^\top\bu^{n+1}-\rmdiv\bu^{n+1}\mathbb{I}_2\right)+\eta_{w^{n+1}}^{\omega}(\vartheta^{n+1})\rmdiv\bu^{n+1}\mathbb{I}_2.
\end{align*}

\item The entropy inequality 
  \begin{align}\label{sub-entropy-equ}
    &\int_{n\Delta t}^{(n+1)\Delta t}\int_{\mathsf{B}}\left(\varrho^{n+1}s_{w^{n+1}}^{\lambda, n+1}\vphi_t+\varrho^{n+1} s_{w^{n+1}}^{\lambda,n+1} \bu^{n+1}\cdot\nabla\vphi-\frac{\kappa_{w^{n+1}}^{\zeta}(\vartheta^{n+1})\nabla\vartheta^{n+1}}{\vartheta^{n+1}}\cdot\nabla\vphi\right)\notag\\
    &+\langle \mathcal{D}_{w^{n+1}}^{\omega,\zeta, n+1};\vphi\rangle_{[\mathcal{M};C]([n\Delta t,(n+1)\Delta t]\times\overline{\mathsf{B}})} \notag\\
    &+\xi \int_{n\Delta t}^{(n+1)\Delta t}\int_{\mathsf{B}}(\vartheta^{n+1})^4\vphi-\delta \int_{n\Delta t}^{(n+1)\Delta t}\int_{\Gamma}\frac{\tilde{\vartheta}^{n+1}-\theta^{n+1}}{\Delta t}\tilde\vphi\\
    &=\int_{n\Delta t}^{(n+1)\Delta t}\frac{\mathrm{d}}{\mathrm{d}t}\int_{\mathsf{B}}\varrho^{n+1}s^{\lambda, n+1} \vphi,\notag
\end{align}
holds for all $\vphi \in C^\infty([n\Delta t,(n+1)\Delta t]\times\mathsf{B})$ and $\tilde \vphi\in C^\infty([n\Delta t,(n+1)\Delta t]\times \Gamma)$ such that $\vphi|_{\Gamma_{w^{n+1}}}=\tilde \vphi$, $\vartheta^{n+1}|_{\Gamma_{w^{n+1}}}=\tilde{\vartheta}^{n+1}$, and 
\begin{equation*}
\mathcal{D}_{w^{n+1}}^{\omega, \zeta, n+1}\geq \frac{1}{\vartheta^{n+1}}\left(\bbS_{w^{n+1}}^{\omega}(\vartheta^{n+1},\nabla\bu^{n+1}):\nabla\bu^{n+1}+\frac{\kappa_{w^{n+1}}^{\zeta}(\vartheta^{n+1})|\nabla\vartheta^{n+1}|^2}{\vartheta^{n+1}}\right).
\end{equation*}
Furthermore, the approximate entropy is given by 
\begin{align*}
s_{w^{n+1}}^{\lambda,n+1}:=s_{w^{n+1}}^{\lambda}(\varrho^{n+1},\vartheta^{n+1}):=s_M(\varrho^{n+1}, \vartheta^{n+1})+\frac{4a^\lambda_{w^{n+1}}}{3}\frac{(\vartheta^{n+1})^3}{\varrho^{n+1}}.
\end{align*}

\item The energy inequality 
  \begin{align}\label{sub-fluid-energy-equ}
&\int_{\mathsf{B}}\bigg[\frac{1}{2}\varrho^{n+1}|\bu^{n+1}|^2+\varrho^{n+1}e^{\lambda}_{w^{n+1}}(\varrho^{n+1}, \vartheta^{n+1})+\frac{1}{2}(b^{n+1})^2+\frac{\delta}{\beta-1}(\varrho^{n+1}+b^{n+1})^\beta\bigg](t)\notag\\
&+\frac{\delta}{2\Delta t}\int_{n\Delta t}^t\int_{\Gamma}\left(|\bfv^{n+1}-w_t^{n+1}\bfn|^2+|\bfv^{n+1}|^2+|\tilde{\vartheta}^{n+1}-\theta^{n+1}|^2+|\tilde{\vartheta}^{n+1}|^2\right)\notag\\
&+\xi \int_{n\Delta t}^t\int_{\mathsf{B}}(\vartheta^{n+1})^5\\
&\leq \int_{\mathsf{B}}\bigg[\frac{1}{2}\varrho^{n}|\bu^{n}|^2+\varrho^{n}e^{\lambda}_{w^{n}}(\varrho^{n}, \theta^{n})+\frac{1}{2}(b^{n})^2+\frac{\delta}{\beta-1}(\varrho^{n}+b^n)^\beta\bigg](n\Delta t)\notag\\
&\hspace{0.5cm}+\frac{\delta}{2\Delta t}\int_{n\Delta t}^t\int_{\Gamma}(|w_{t}^{n+1}|^2+|\theta^{n+1}|^2),\notag
\end{align}
holds for any $t\in(n\Delta t, (n+1)\Delta t]$, where the approximate energy is given by
\begin{align*}
e^{\lambda}_{w^{n+1}}(\varrho^{n+1},\vartheta^{n+1}):=e_M(\varrho^{n+1},\vartheta^{n+1})+\frac{a^{\lambda}_{w^{n+1}}}{\varrho^{n+1}}(\vartheta^{n+1})^4.
\end{align*}

\item The total energy inequality
  \begin{align}\label{sub-total-energy-inequality}
&\int_{\mathsf{B}}\bigg[\frac{1}{2}\varrho^{n+1}|\bu^{n+1}|^2+\frac{1}{2}(b^{n+1})^2+\mathcal{H}_{\lambda}(\varrho^{n+1}, \vartheta^{n+1})-\frac{\mathcal{H}_{\lambda}(\bar{\varrho}, 1)}{\partial \varrho}(\varrho^{n+1}-\bar{\varrho})-\mathcal{H}_{\lambda}(\bar{\varrho}, 1)\notag\\
&+\frac{\delta}{\beta-1}(\varrho^{n+1}+b^{n+1})^\beta\bigg](t)\notag\\
&+\int_{n\Delta t}^t\int_{\mathsf{B}}\frac{1}{\vartheta^{n+1}}\left(\bbS_{w^{n+1}}^{\omega}(\vartheta^{n+1},\nabla\bu^{n+1}):\nabla\bu^{n+1}+\frac{\kappa_{w^{n+1}}^{\zeta}(\vartheta^{n+1})|\nabla\vartheta^{n+1}|^2}{\vartheta^{n+1}}\right)\notag\\
&+\frac{\delta}{2\Delta t}\int_{n\Delta t}^t\int_{\Gamma}\left(|\bfv^{n+1}-w_t^{n+1}\bfn|^2+|\bfv^{n+1}|^2+|\tilde{\vartheta}^{n+1}-\theta^{n+1}|^2+|\tilde{\vartheta}^{n+1}|^2\right)\\
&+\xi \int_{n\Delta t}^t\int_{\mathsf{B}}(\vartheta^{n+1})^5\notag\\
&\leq \int_{\mathsf{B}}\bigg[\frac{1}{2}\varrho^{n}|\bu^{n}|^2+\frac{1}{2}(b^{n})^2+\mathcal{H}_{\lambda}(\varrho^{n}, \vartheta^{n})-\frac{\mathcal{H}_{\lambda}(\bar{\varrho}, 1)}{\partial \varrho}(\varrho^{n}-\bar{\varrho})-\mathcal{H}_{\lambda}(\bar{\varrho}, 1)\notag\\
&\hspace{0.5cm}+\frac{\delta}{\beta-1}(\varrho^{n}+b^n)^\beta\bigg](n\Delta t)\notag\\
&\hspace{0.5cm}+\frac{\delta}{2\Delta t}\int_{n\Delta t}^t\int_{\Gamma}(|w_t^{n+1}|^2+|\theta^{n+1}|^2)+\delta \int_{n\Delta t}^t\int_{\Gamma}\frac{\tilde{\vartheta}^{n+1}-\theta^{n+1}}{\Delta t}, \notag
\end{align}
holds for any $t\in (n\Delta t, (n+1)\Delta t]$, where the approximate Helmholtz function $\mathcal{H}_{\lambda}$ is defined as 
\begin{equation*}
    \mathcal{H}_{\lambda}(\varrho^{n+1},\vartheta^{n+1})=\varrho^{n+1}\left(e^{\lambda}_{w^{n+1}}(\varrho^{n+1},\vartheta^{n+1})-\bar{\vartheta}s_{w^{n+1}}^{\lambda}(\varrho^{n+1},\vartheta^{n+1})\right), 
\end{equation*}
for some given constant $\bar{\vartheta}>0$, 
and the constant $\bar{\varrho}$ satisfying $\int_{\mathsf{B}}(\varrho-\bar{\varrho})=0$.
\end{itemize}

\subsection{Solving the sub-problems}
In this subsection, we will demonstrate how to solve the sub-problems and explain why all estimates are satisfied.

By modifying the method presented in \cite{MC-16} to include penalization terms, we can solve the sub-problem for the structure. Specifically, by choosing $\psi=w_t^{n+1}$ in \eqref{weak-plate-function} and $\tilde\vphi=\theta^{n+1}$ in $\eqref{weak-heat-function}$, we can derive $\eqref{weak-energy-structure}$ where we have used the following facts:
\begin{align*}
&\delta \int_{n\Delta t}^{ (n+1)\Delta t}\int_\Gamma \frac{w_t^{n+1}-\tau_{\Delta t}\bfv^{n}\cdot\bfn}{\Delta t}w^{n+1}_t\\
&=\frac{\delta}{2\Delta t}\int_{n\Delta t}^{(n+1)\Delta t}\left(\|w^{n+1}_t-\tau_{\Delta t}\bfv^{n}\cdot\bfn\|^2_{L^2(\Gamma)}+\|w^{n+1}_t\|_{L^2(\Gamma)}\right)-\frac{\delta}{2\Delta t}\int_{n\Delta t}^{(n+1)\Delta t}\|\tau_{\Delta t}\bfv^{n}\cdot\bfn\|^2_{L^2(\Gamma)},
\end{align*}
and 
\begin{align*}
&\delta \int_{n\Delta t}^{(n+1)\Delta t }\int_\Gamma \frac{\theta^{n+1}-\tau_{\Delta t}\tilde\vartheta^{n}}{\Delta t}\theta^{n+1}\\
&=\frac{\delta}{2\Delta t}\int_{n\Delta t}^{(n+1)\Delta t}\left(\|\theta^{n+1}-\tau_{\Delta t}\tilde\vartheta^{n}\|^2_{L^2(\Gamma)}+\|\theta^{n+1}\|_{L^2(\Gamma)}\right)-\frac{\delta}{2\Delta t}\int_{n\Delta t}^{(n+1)\Delta t}\|\tau_{\Delta t}\tilde\vartheta^{n}\|^2_{L^2(\Gamma)}.
\end{align*}

Regarding the existence of a solution to the sub-problem for the fluid, the initial step in our approach is to consider the existence in the fixed domain $\mathsf{B}$. Therefore, the proof of existence itself does not significantly deviate from that in \cite{LS-21}, and we provide a brief explanation to highlight the differences. There are two penalization terms 
$\displaystyle -\delta \int_{n\Delta t}^{(n+1)\Delta t}\int_{\Gamma}\frac{\bfv^{n+1}-w_t^{n+1}\bfn}{\Delta t}\cdot\bmpsi$ and $\displaystyle -\delta \int_{n\Delta t}^{(n+1)\Delta t}\int_{\Gamma}\frac{\tilde{\vartheta}^{n+1}-\theta^{n+1}}{\Delta t}\tilde \vphi$ in momentum equation $\eqref{sub-fluid-momentum-equ}$ and entropy equation $\eqref{sub-entropy-equ}$, separately. For fixed $\Delta t$, these terms can be effectively managed by treating them as compact perturbations.

\subsection{The uniform bounds for the approximated solutions}

Suppose that the sub-problems for fluid and structure have already solved in time interval $[0, m\Delta t]$ with $1\leq m\leq N-1$, then we aim to derive the uniform bounds for these solutions. Note that we omit the superscript $n+1$ of the solutions in the following analysis for simplicity.

For $n\leq m-2$, we sum $\eqref{weak-energy-structure}$ and $\eqref{sub-fluid-energy-equ}$ at times $(n+1)\Delta t$ over $n=1,...,m-2$, and add $\eqref{weak-energy-structure}$ and $\eqref{sub-fluid-energy-equ}$ at $t\in [(m-1)\Delta t, m\Delta t]$
to get 
\begin{align*}
    &\int_{\mathsf{B}}\left[\frac{1}{2}\varrho|\bu|^2+\varrho e^\lambda_w(\varrho,\theta)+\frac{1}{2}b^2+\frac{\delta}{\beta-1}(\varrho+b)^\beta\right](t)\\
    &+\frac{1-\delta}{2}\|w_t\|^2_{L^2(\Gamma)}(t)+\frac{1}{2}\|\Delta w\|^2_{L^2(\Gamma)}(t)+\frac{\alpha_2}{2}\|\nabla w_t\|^2_{L^2(\Gamma)}(t)+\frac{1-\delta}{2}\|\theta\|^2_{L^2(\Gamma)}(t)\\
&+\xi\int_0^t\int_{\mathsf{B}}\vartheta^5+\int_0^t(\alpha_1\|\nabla w_t\|^2_{L^2(\Gamma)}+\|\nabla\theta\|^2_{L^2(\Gamma)})\\
    &+\frac{\delta}{2\Delta t}\int^t_{(m-1)\Delta t}\|\tau_{\Delta t}\bfv\|^2_{L^2(\Gamma)}+\frac{\delta}{2\Delta t}\int_0^t(\|\bfv-w_t\bfn\|^2_{L^2(\Gamma)}+\|w_t-\tau_{\Delta t}\bfv\cdot\bfn\|^2_{L^2(\Gamma)})\\
    &+\frac{\delta}{2\Delta t}\int^t_{(m-1)\Delta t}\|\tau_{\Delta t}\tilde{\vartheta}\|^2_{L^2(\Gamma)}+\frac{\delta}{2\Delta t}\int_0^t(\|\tilde{\vartheta}-\theta\|^2_{L^2(\Gamma)}+\|\theta-\tau_{\Delta t}\tilde{\vartheta}\|^2_{L^2(\Gamma)})\\
    &\leq \int_{\mathsf{B}}\left[\frac{1}{2}\frac{|(\varrho\bu)_{0,\delta}|^2}{\varrho_{0,\delta}}+\varrho_{0,\delta}e^\lambda_w(\varrho_{0,\delta}, \vartheta_{0,\delta})+\frac{1}{2}b_{0,\delta}^2+\frac{\delta}{\beta-1}(\varrho_{0,\delta}+b_{0,\delta})^\beta\right]\\
    &\hspace{0.5cm}+\frac{1}{2}\|v_{0,\delta}\|^2_{L^2(\Gamma)}+\frac{1}{2}\|\Delta w_{0,\delta}\|^2_{L^2(\Gamma)}+\frac{\alpha_2}{2}\|\nabla v_{0,\delta}\|^2_{L^2(\Gamma)}+\frac{1}{2}\|\theta_{0,\delta}\|^2_{L^2(\Gamma)}.
  \end{align*}
  
Performing the same procedures for $\eqref{sub-total-energy-inequality}$, we obtain   
  \begin{align}\label{Delta-t-total-energy}
    &\int_{\mathsf{B}}\bigg[\frac{1}{2}\varrho|\bu|^2+\frac{1}{2}b^2+\mathcal{H}_\lambda(\varrho,\vartheta)-\frac{\partial \mathcal{H}_{\lambda}(\bar{\varrho},1)}{\partial \varrho}(\varrho-\bar{\varrho})-\mathcal{H}_{\lambda}(\bar{\varrho}, 1)+\frac{\delta}{\beta-1}(\varrho+b)^\beta\bigg](t)\notag\\
    &+\frac{1-\delta}{2}\|w_t\|^2_{L^2(\Gamma)}(t)+\frac{1}{2}\|\Delta w\|^2_{L^2(\Gamma)}(t)+\frac{\alpha_2}{2}\|\nabla w_t\|^2_{L^2(\Gamma)}(t)+\frac{1-\delta}{2}\|\theta\|^2_{L^2(\Gamma)}(t)\notag\\
    &+\int_0^t\int_{\mathsf{B}}\frac{1}{\vartheta}\left(\bbS_w^\omega(\vartheta,\nabla\bu):\nabla\bu+\frac{\kappa_w^\zeta(\vartheta)|\nabla\vartheta|^2}{\vartheta}\right)\notag\\
&+\xi\int_0^t\int_{\mathsf{B}}\vartheta^5+\int_0^t(\alpha_1\|\nabla w_t\|^2_{L^2(\Gamma)}+\|\nabla\theta\|^2_{L^2(\Gamma)})\\
    &+\frac{\delta}{2\Delta t}\int^t_{(m-1)\Delta t}\|\tau_{\Delta t}\bfv\|^2_{L^2(\Gamma)}+\frac{\delta}{2\Delta t}\int_0^t(\|\bfv-w_t\bfn\|^2_{L^2(\Gamma)}+\|w_t-\tau_{\Delta t}\bfv\cdot\bfn\|^2_{L^2(\Gamma)})\notag\\
    &+\frac{\delta}{2\Delta t}\int^t_{(m-1)\Delta t}\|\tau_{\Delta t}\tilde{\vartheta}\|^2_{L^2(\Gamma)}+\frac{\delta}{2\Delta t}\int_0^t(\|\tilde{\vartheta}-\theta\|^2_{L^2(\Gamma)}+\|\theta-\tau_{\Delta t}\tilde{\vartheta}\|^2_{L^2(\Gamma)})\notag\\
    &\leq \int_{\mathsf{B}}\bigg[\frac{1}{2}\frac{|(\varrho\bu)_{0,\delta}|^2}{\varrho_{0,\delta}}+\frac{1}{2}b_{0,\delta}^2+\mathcal{H}_\lambda(\varrho_{0,\delta},\vartheta_{0,\delta})-\frac{\partial \mathcal{H}_\lambda(\bar{\varrho},1)}{\partial \varrho}(\varrho_{0,\delta}-\bar{\varrho})-\mathcal{H}_\lambda(\bar{\varrho}, 1)\notag\\
    &\hspace{0.5cm}{+\frac{\delta}{\beta-1}(\varrho_{0,\delta}+b_{0,\delta})^\beta}\bigg]\notag\\
    &+\frac{1}{2}\|v_{0,\delta}\|^2_{L^2(\Gamma)}+\frac{1}{2}\|\Delta w_{0,\delta}\|^2_{L^2(\Gamma)}+\frac{\alpha_2}{2}\|\nabla v_{0,\delta}\|^2_{L^2(\Gamma)}+\frac{1}{2}\|\theta_{0,\delta}\|^2_{L^2(\Gamma)}\notag\\
    &+\frac{\delta }{4\Delta t}\int_0^t\|\tau_{\Delta t}\tilde{\vartheta}-\theta\|^2_{L^2(\Gamma)}+\delta C (\Gamma)+\delta\int_0^t\int_{\Gamma}\frac{\tilde{\vartheta}-\tau_{\Delta t}\tilde{\vartheta}}{\Delta t}.\notag
  \end{align}

 The last three terms in \eqref{Delta-t-total-energy} come from  
\begin{align*}
&\delta \sum_{n}\int_{n\Delta t}^{(n+1)\Delta t}\int_{\Gamma}\frac{{\tilde\vartheta}^{n+1}-\theta^{n+1}}{\Delta t}=\delta \sum_{n}\int_{n\Delta t}^{(n+1)\Delta t}\int_{\Gamma}\frac{{\tilde\vartheta}^{n+1}-\tau_{\Delta t}\tilde\vartheta^{n+1}+\tau_{\Delta t}\tilde\vartheta^{n+1}-\theta^{n+1}}{\Delta t}\\
&=\delta \int_0^t\int_{\Gamma}\frac{\tilde\vartheta-\tau_{\Delta t}\tilde\vartheta}{\Delta t}+\delta \int_0^t\int_{\Gamma}\frac{\tau_{\Delta t}\tilde\vartheta-\theta}{\Delta t},
\end{align*}
and 
\begin{align*}
\delta \int_0^t\int_{\Gamma}\frac{\tau_{\Delta t}\tilde\vartheta-\theta}{\Delta t}\leq \frac{\delta }{4\Delta t}\int_0^t\|\tau_{\Delta t}\tilde{\vartheta}-\theta\|^2_{L^2(\Gamma)}+\delta C(\Gamma),
\end{align*}
where $C(\Gamma)>0$ is a constant depending on $\Gamma$.

Furthermore, the straightforward calculations give that 
  \begin{align*}
\delta\int_0^t\int_{\Gamma}\frac{\tilde{\vartheta}-\tau_{\Delta t}\tilde{\vartheta}}{\Delta t}&=-\delta\int_0^{\Delta t}\int_{\Gamma}\frac{\theta_{0,\delta}}{\Delta t}+\delta \int^t_{t-\Delta t}\int_{\Gamma}\frac{\tilde{\vartheta}}{\Delta t}\\
&\leq \delta \|\theta_{0,\delta}\|^2_{L^2(\Gamma)}+\frac{\delta}{4\Delta t}\int_{t-\Delta t}^t\|\tilde{\vartheta}\|^2_{L^2(\Gamma)}+\frac{\delta}{\Delta t}\int_{t-\Delta t}^t\int_{\Gamma}1\\
&\leq \delta \|\theta_{0,\delta}\|^2_{L^2(\Gamma)}+\frac{\delta}{4\Delta t}\int_{t-\Delta t}^t\|{\vartheta}\|^2_{L^2(\Gamma)}+\frac{\delta}{4\Delta t}\int_{t-\Delta t}^t\|\theta-\tilde{\vartheta}\|^2_{L^2(\Gamma)}+
\delta C(\Gamma)\\
&\leq \delta \|\theta_{0,\delta}\|^2_{L^2(\Gamma)}+\frac{\delta}{4}\|{\vartheta}\|^2_{L^2(\Gamma)}+\frac{\delta}{4\Delta t}\int_{t-\Delta t}^t\|\theta-\tilde{\vartheta}\|^2_{L^2(\Gamma)}+
\delta C(\Gamma). 
\end{align*}

Based on the preceding uniform bounds, we can recursively deduce the solution over the entire time interval $[0,T]$. This solution depends on parameters such as $\Delta t, \omega, \zeta, \lambda, \xi, \delta$, although we simplify by not explicitly enumerating them.

From our analysis, we establish that the approximate solution $(\varrho, b, \bu, \vartheta, w, \theta)$ satisfies the energy inequality $\eqref{Delta-t-total-energy}$ on the time interval $[0,T]$. Furthermore, we need to formulate weak formulations for the continuity equation, magnetic equation, coupled momentum equation, and coupled entropy equation. Specifically, these are:
\begin{itemize}
\item The continuity equation 
\begin{align}\label{weak-continuity-app}
\int_{0}^{t}\int_{\mathsf{B}}\varrho(\phi_t+\bu\cdot\nabla\phi)=\int_{\mathsf{B}}\varrho(t,\cdot)\phi(t,\cdot)-\int_{\mathsf{B}}\varrho_{0,\delta}\phi(0,\cdot), 
\end{align}
holds for $t\in[0,T]$ and all $\phi\in C^\infty_c ([0,t]\times\mathsf{B})$.

\item The non-resistive magnetic equation 
\begin{align}\label{weak-magnetic-induction-app}
\int_{0}^{t}\int_{\mathsf{B}}b(\phi_t+\bu\cdot\nabla\phi)=\int_{\mathsf{B}}b(t,\cdot)\phi(t,\cdot)-\int_{\mathsf{B}}b_{0,\delta}\phi(t,\cdot), 
  \end{align}
holds for $t\in[0,T]$ and all $\phi\in C^\infty_c ([0,t]\times\mathsf{B})$. 

  \item The coupled  momentum equation
    \begin{align}\label{sub-fluid-momentum-equ-app}
&\int_{0}^{t}\int_{\mathsf{B}}\varrho\bu\cdot\bmphi_t+ \int_{0}^{t}\int_{\mathsf{B}}\left(\varrho\bu\otimes\bu+p_w^{\lambda}(\varrho,\vartheta)\mathbb{I}_2+\frac{1}{2}b^2\mathbb{I}_2+\delta (\varrho+b)^\beta\mathbb{I}_2-\bbS_w^{\omega}(\vartheta, \nabla\bu)\right):\nabla\bmphi\notag \\
    &-\delta \int_{0}^{t}\int_{\Gamma}\frac{(\bfv-\tau_{\Delta t}\bfv)\cdot\bfn}{\Delta t} \psi+(1-\delta)\int_0^t\int_{\Gamma}w_t\psi_t-\int_0^t\int_{\Gamma}\Delta w\Delta \psi\notag\\
    &-\int_0^t\int_{\Gamma}\alpha_1\nabla w_t\cdot\nabla\psi+\int_0^t\int_{\Gamma}(\alpha_2\nabla w_t\cdot\nabla\psi_t+\nabla\theta\cdot\nabla\psi)\\
    &=-\int_{\mathsf{B}}(\varrho\bu)_{0,\delta}\cdot\bmphi(0,\cdot)-(1-\delta)\int_{\Gamma}v_0\psi(0,\cdot)-\int_{\Gamma}\nabla v_0\cdot\nabla\psi(0,\cdot),\notag 
    \end{align}
holds for $t\in[0, T]$ and all $\bmphi \in C^\infty_c([0, t)\times\mathsf{B})$ and $\psi\in C^\infty_c([0, t)\times\Gamma)$ such that $\bmphi|_{\Gamma_{w}(t)}=\psi \bfn$ on $\Gamma_T$, where 
  \begin{align*}
&\bfv:=\bu|_{\Gamma_{w}}, \quad p_w^{\lambda}(\varrho, \vartheta):=p_M(\varrho,\vartheta)+\frac{a^\lambda_{w}}{3}\vartheta^4, \\
&\bbS_w^{\omega}(\vartheta,\nabla\bu)=\mu_{w}^{\omega}(\vartheta)\left(\nabla\bu+\nabla^\top \bu-\rmdiv\bu\mathbb{I}_2\right)+\eta_{w}^{\omega}(\vartheta)\rmdiv\bu\mathbb{I}_2.
\end{align*}

\item The coupled  entropy equation 
  \begin{align}\label{sub-entropy-equ-app}
    &\int_{0}^{t}\int_{\mathsf{B}}\left(\varrho s_w^{\lambda}\vphi_t + \varrho s_w^{\lambda} \bu\cdot\nabla\vphi -\frac{\kappa_w^{\zeta}(\vartheta)\nabla\vartheta}{\vartheta}\cdot\nabla\vphi\right)+\langle \mathcal{D}_{w}^{\omega,\zeta};\vphi\rangle_{[\mathcal{M};C]([0, t]\times\overline{\mathsf{B}})} \notag\\
    &+\xi \int_{0}^{ t}\int_{\mathsf{B}}\vartheta^4\vphi-\delta \int_{0}^{ t}\int_{\Gamma}\frac{\tilde{\vartheta}-\tau_{\Delta t}\tilde\vartheta}{\Delta t}\tilde \vphi+(1-\delta)\int_0^t\int_{\Gamma}\theta\tilde\vphi_t-\int_0^t\int_{\Gamma}\nabla\theta\cdot\nabla\tilde \vphi\\
    &+\int_0^t\int_{\Gamma}\nabla w\cdot\nabla \tilde\vphi_t=\int_{\mathsf{B}} \varrho_{0,\delta} s^{\lambda}_w(\varrho_{0,\delta}, \vartheta_{0,\delta})\vphi(0,\cdot)+(1-\delta)\int_{\Gamma}\theta_0\tilde\vphi(0,\cdot)+\int_{\Gamma}\nabla w_0\cdot\nabla\tilde\vphi(0,\cdot),\notag
\end{align}
holds for all $t\in [0,T]$ and $\vphi \in C^\infty_c([0, t)\times\mathsf{B})$ and $\tilde\vphi\in C^\infty_c([0, t)\times \Gamma)$ such that $\vphi|_{\Gamma_{w}(t)}=\tilde\vphi$ on $\Gamma_T$, $\vartheta|_{\Gamma_{w}(t)}=\tilde \vartheta$, and 
\begin{equation*}
\mathcal{D}_{w}^{\omega, \zeta}\geq \frac{1}{\vartheta}\left(\bbS_w^{\omega}(\vartheta,\nabla\bu):\nabla\bu+\frac{\kappa_w^{\zeta}(\vartheta)|\nabla\vartheta|^2}{\vartheta}\right).
\end{equation*}
Note that the approximate entropy is given by 
\begin{align*}
s_w^{\lambda}:=s_w^{\lambda}(\varrho,\vartheta)=s_M(\varrho, \vartheta)+\frac{4a^\lambda_{w}}{3}\frac{\vartheta^3}{\varrho}.
\end{align*}
\end{itemize}

Now, we are in position to state the following lemma.
\begin{lemma}
For any $t\in (0,T]$, there is a positive constant $C$ independent of $\Delta t, \omega, \zeta, \lambda, \xi, \delta$, such that the uniform bounds for approximate solutions $(\rho, b, \bu, \vartheta, w, \theta)$ hold as follows: 
\begin{align}\label{Delta-t-uni-bound}
&\frac{\delta}{\Delta t} \int_0^t(\|\tilde{\vartheta}-\theta\|^2_{L^2(\Gamma)}+\|\bfv-w_t\bfn\|^2_{L^2(\Gamma)})\leq  C,\notag\\
&\esssup_{0\leq t\leq T}(\|w_t\|^2_{L^2(\Gamma)}(t)+\|\Delta  w\|^2_{L^2(\Gamma)}(t)+\alpha_2\|\nabla w_t\|^2_{L^2(\Gamma)}(t)+\|\theta\|^2_{L^2(\Gamma)}(t))\leq C,\notag\\
&\int_0^t(\alpha_1\|\nabla w_t\|^2_{L^2(\Gamma)}+\|\nabla\theta\|^2_{L^2(\Gamma)})\leq C,\notag\\
&{\esssup_{0\leq t\leq T}\int_{\mathsf{B}}\bigg(\varrho|\bu|^2+b^2+\delta(\varrho+b)^\beta\bigg)(t)\leq C,}\\
&\esssup_{0\leq t\leq T}\int_{\mathsf{B}} (a^\lambda_w\vartheta^4+{\varrho^{\gamma}}+\varrho \vartheta)(t)\leq C\esssup_{0\leq t\leq T}\int_{\mathsf{B}} \varrho e_w^{\lambda}(\varrho,\vartheta)(t)\leq C,\notag\\
&\int_0^t\int_{\mathsf{B}}\left(\xi\vartheta^5+h_w^{\zeta}(|\nabla\log\vartheta|^2+|\nabla\vartheta^{\frac{3}{2}}|^2)\right)\leq C,\notag\\
&  \int_0^t\int_{\mathsf{B}}\frac{1}{\vartheta}\left(\bbS_{w}^{\omega}(\vartheta, \nabla\bu):\nabla\bu+\frac{\kappa_w^{\zeta}(\vartheta)|\nabla\vartheta|^2}{\vartheta}\right)\leq C.\notag
\end{align}
\end{lemma}

We derive from $\eqref{Delta-t-uni-bound}_7$, $\eqref{vis-coe}$, $\eqref{vis-app}$ and $\eqref{g-app}$ that (and by using  Korn-Poincar\'{e} inequality in the fixed domain)
\begin{align}\label{Delta-t-bu}
   \|\bu\|^2_{L^2(0,T;W^{1,2}(\mathsf{B}))}\leq \frac{C}{\omega},
\end{align}
which further gives (since we are in dimension $2$)
\begin{align}\label{Delta-t-bu-Lp}
\|\bu\|_{L^2(0,T; L^p(\mathsf{B}))} \leq \frac{C}{\sqrt{\omega}}, \quad \text{for any } p \in [2, \infty) .
\end{align}

Next, we compute 
\begin{align*}
\left[\int_{\sfB} |\varrho \bu|^{\frac{2\gamma}{\gamma+1}}\right]^{\frac{\gamma+1}{2\gamma}} \leq 
\left[\left( \int_{\sfB}\left|\sqrt{\varrho} \bu \right|^2  \right)^{\frac{\gamma}{\gamma+1}} \left(\int_{\sfB} \varrho^\gamma \right)^{\frac{1}{\gamma+1}} \right]^{\frac{\gamma+1}{2\gamma}} ,
\end{align*} 
and 
thus using $\eqref{Delta-t-uni-bound}_4$ and $\eqref{Delta-t-uni-bound}_5$, it holds that 
  \begin{align}\label{Delta-t-varrho-bu}\|\varrho\bu\|_{L^\infty(0,T;L^{\frac{2\gamma}{\gamma+1}}(\mathsf{B}))}\leq C.
\end{align}
We now observe that 
\begin{align*}
\bigg| \int_{\sfB} \varrho \bu \otimes \bu \cdot  \bmphi \bigg| \leq \|\varrho \bu\|_{L^{\frac{2\gamma}{\gamma+1}}(\sfB) } \|\bu\|_{L^p(\sfB)} \|\bmphi\|_{L^{\frac{2\gamma p}{p(\gamma-1)-2\gamma}}(\sfB)},
\end{align*} 
for any $\bmphi \in L^{\frac{2\gamma p}{p(\gamma-1)-2\gamma}}(\sfB)$ and any $p > \frac{2\gamma}{\gamma-1}$ (since the dimension is $2$). This implies (using \eqref{Delta-t-bu} and \eqref{Delta-t-varrho-bu}) 
 \begin{align}\label{Delta-t-varrho-bu-2}
\|\varrho\bu\otimes \bu\|_{L^2(0,T;L^{\frac{2p\gamma}{(2+p)\gamma+p}}(\mathsf{B}))} \leq \frac{C}{\sqrt{\omega}},
\end{align}
for any $p> \frac{2\gamma}{\gamma-1}$.

Let us observe that 
\begin{align*}
\| b \bu \|_{L^{\frac{2p}{p+2}}(\sfB)}\leq \|b\|_{L^2(\sfB)} \|\bu\|_{L^p(\sfB)} , \ \ \text{for any } p \geq 2,
\end{align*}
and therefore, by means of $\eqref{Delta-t-uni-bound}_4$ and \eqref{Delta-t-bu-Lp}, we have 
\begin{align}\label{Delta-t-b-bu-bound} 
\|b \bu \|_{L^2(0,T;L^{\frac{2p}{p+2}}(\sfB))} \leq \frac{C}{\sqrt{\omega}}.
\end{align}

We also infer from $\eqref{Delta-t-uni-bound}_5$ and $\eqref{Delta-t-uni-bound}_7$ that 
  \begin{align}\label{Delta-t-entropy}
\|\varrho s_{w}^{\lambda}(\varrho,\vartheta)\|_{L^q((0,T)\times\mathsf{B})}+\|\varrho s_{w}^{\lambda}(\varrho,\vartheta)\bu\|_{L^s((0,T)\times\mathsf{B})}+\left\|\frac{\kappa_{w}^{\zeta}{(\vartheta)}\nabla\vartheta}{\vartheta}\right\|_{L^r((0,T)\times\mathsf{B})}\leq C, 
\end{align}
for some $q, s, r>1$, where $C$ is dependent on the parameters $w, \lambda,\zeta$.


By employing the concept based on ``Bogovski\u{\i} operator'' (see \cite{Bo-80, FP-00} for details), we can achieve the improved integrability as follows:
\begin{align}\label{pressure-int}
{\iint_{\mathcal O}\bigg(\varrho^{\gamma+1}+b^{2+1}+\delta(\varrho^{\beta+1}+b^{\beta+1})\bigg)\leq C},
\end{align}
for any compact subset $\mathcal O \Subset \overline{(0,T) \times \mathsf{B}}\setminus \Gamma^w_T$,  where $C$ is independent of $\Delta t$.

\vspace{.15cm}

As for the estimate of displacement $w(t,y)$, for any $\frac{3}{4}<\theta<1$, it holds that 
\begin{align}\label{embedding-Delta-t}
w& \in  W^{1,\infty}(0,T;L^2(\Gamma))\cap L^\infty (0,T;H^2(\Gamma))\notag \\
&\hookrightarrow C^{0, 1-\theta}(0,T;H^{2\theta}(\Gamma))\\
&\hookrightarrow C^{0, 1-\theta}(0,T; C^{1, 2\theta-\frac{3}{2}}(\Gamma))\notag ,
\end{align}

where we have used the following elementary estimates 
  \begin{align*}
\|w(t)-w(s)\|_{(L^2(\Gamma), H^2(\Gamma))_{\theta,2}}&\leq \|w(t)-w(s)\|_{H^2(\Gamma)}^\theta\|w(t)-w(s)\|_{L^2(\Gamma)}^{1-\theta}\\
&\leq C\|w\|_{L^\infty(0,T;H^2(\Gamma))}^\theta\|w\|_{W^{1,\infty}(0,T;L^2(\Gamma))}^{1-\theta}|t-s|^{1-\theta}, 
\end{align*}
and 
  \begin{align*}
H^{2\theta}(\Gamma) \hookrightarrow C^{1, 2\theta-\frac{3}{2}}(\Gamma);
\end{align*}
see for instance \cite[Theorem 1.50]{Danchin-chemin-hajer}. 

Finally, the upper bound of the time interval $T$ is determined by 
\begin{equation} 
\begin{aligned}\label{T-max}
&\alpha_{\partial\Omega}< \min_{y\in\Gamma}w_0(y)-CT^{\frac{1}{5}}\leq |w_0(y)|-|w(t,y)-w_0(y)| \leq |w(t,y)| ,  \\
& |w(t,y)| \leq  |w_0(y)|+|w(t,y)-w_0(y)|\leq \max_{y\in\Gamma}w_0(y)+CT^{\frac{1}{5}}< \beta_{\partial\Omega}; 
\end{aligned}
\end{equation}
such calculations reveal that the displacement of the shell is bounded from the initial setup.

\bigskip 

In the subsequent sections of this paper, we will proceed in the following order to derive the final limiting system. Our approach involves firstly considering $\Delta t\rightarrow 0$, followed by $\omega, \zeta, \lambda, \xi \rightarrow 0 $, and finally $\delta\rightarrow 0$.

\section{Passing to the limit as $\Delta t\rightarrow 0$}\label{Sec-limit-Delta-t}  
In this section, our goal is to derive the limiting system when   $\Delta t\rightarrow 0$.

\subsection{A first set of limits}\label{limit-Delta-t}
From $\eqref{Delta-t-uni-bound}_1$ and $\eqref{Delta-t-uni-bound}_2$, we infer  that 
  \begin{align}\label{Delta-w-theta}
&\tilde\vartheta_{\Delta t}\rightharpoonup \theta,\hspace{0.5cm} \mathrm{weakly} \hspace{0.1cm}\mathrm{in} \hspace{0.1cm} L^2(0,T; L^2(\Gamma))\notag,\\
&\bfv_{\Delta t}\rightharpoonup w_t\bfn,\hspace{0.5cm} \mathrm{weakly} \hspace{0.1cm} \mathrm{in} \hspace{0.1cm} L^2(0,T; L^2(\Gamma),\\
&(w_{\Delta t})_t \rightharpoonup w_t,\hspace{0.5cm} \mathrm{weakly *} \hspace{0.1cm} \mathrm{in} \hspace{0.1cm} L^\infty(0,T;L^2({\Gamma})) \notag,\\
&\theta_{\Delta t}\rightharpoonup \theta,\hspace{0.5cm} \mathrm{weakly*} \hspace{0.1cm} \mathrm{in} \hspace{0.1cm} L^\infty(0,T;L^2({\Gamma})), \hspace{0.1cm} \mathrm{as}\hspace{0.1cm}  \Delta t\to  0.\notag
\end{align}

Moreover, for the penalization terms in \eqref{sub-fluid-momentum-equ-app} and $\eqref{sub-entropy-equ-app}$, we use the definition of time-shift in $\eqref{time-shift}$ to derive that 
\begin{align}\label{Delta-t-v}
&\int_0^T\int_{\Gamma}\frac{\bfv_{\Delta t}\cdot \bfn -\tau_{\Delta t}\bfv_{\Delta t}\cdot \bfn}{\Delta t}\psi\notag\\
&=-\int_{\Delta t}^{T-\Delta t}\int_{\Gamma}\bfv_{\Delta t}\cdot \bfn \frac{\tau_{-\Delta t}\psi-\psi}{\Delta t}+\frac{1}{\Delta t}\int_{T-\Delta t}^T\int_{\Gamma}\bfv_{\Delta t}\cdot\bfn \psi-\frac{1}{\Delta t}\int_{0}^{\Delta t}\int_{\Gamma}v_{0}\psi\\
&\to -\int_0^T\int_{\Gamma}w_t\psi_t-\int_{\Gamma}v_{0}\psi(0,\cdot) , \hspace{0.5cm}\mathrm{as} \hspace{0.2cm} \Delta t \to 0, \ \forall \psi\in C_c^\infty([0,T)\times\Gamma), \notag
\end{align}
and 
\begin{align}\label{Delta-w-theta}
&-\int_0^T\int_{\Gamma}\frac{\tilde{\vartheta}_{\Delta t}-\tau_{\Delta t}\tilde{\vartheta}_{\Delta t}}{\Delta t}\tilde\vphi\notag\\
&=-\int_{\Delta t}^{T-\Delta t}\int_{\Gamma}\tilde{\vartheta}_{\Delta t}\frac{\tau_{-\Delta t}\tilde\vphi-\tilde\vphi}{\Delta t}+\frac{1}{\Delta t}\int_{T-\Delta t}^{T}\tilde{\vartheta}_{\Delta t}\tilde\vphi-\frac{1}{\Delta t}\int_0^{\Delta t}\int_{\Gamma}\theta_{0}\tilde\vphi \\
&\to -\int_0^T\int_{\Gamma}\theta\tilde\vphi_t-\int_{\Gamma}\theta_{0}\tilde\vphi(0,\cdot),\hspace{0.5cm}\mathrm{as} \hspace{0.2cm} \Delta t \to 0, \  \forall \tilde\vphi\in C_c^\infty([0,T)\times\Gamma). \notag
\end{align}

Thus, in the weak formulations  \eqref{sub-fluid-momentum-equ-app} and $\eqref{sub-entropy-equ-app}$, we finally have 
\begin{align*}
&-\delta\int_0^T\int_{\Gamma}\frac{\bfv_{\Delta t}\cdot\bfn -\tau_{\Delta}\bfv_{\Delta t}\cdot\bfn}{\Delta t}\psi+(1-\delta)\int_0^T\int_{\Gamma}(w_{\Delta t})_t\psi_t-(1-\delta)\int_0^T\frac{\mathrm{d}}{\mathrm{d}t}\int_{\Gamma}(w_{\Delta t})_t\psi\notag\\
&\to \int_0^T\int_{\Gamma}w_t\psi_t+\int_{\Gamma}w_0\psi(0,\cdot), \hspace{0.5cm}\mathrm{as} \hspace{0.2cm} \Delta t \to 0, \ \forall \psi\in C_c^\infty([0,T)\times\Gamma)
\end{align*}
and 
  \begin{align*}
&-\delta\int_0^T\int_{\Gamma}\frac{\tilde{\vartheta}_{\Delta t}-\tau_{\Delta t}\tilde{\vartheta}_{\Delta t}}{\Delta t}\tilde\vphi-(1-\delta)\int_0^T\int_{\Gamma}\theta_{\Delta t} \tilde\vphi_t+(1-\delta)\int_0^T\frac{\mathrm{d}}{\mathrm{d} t}\int_{\Gamma}\theta_{\Delta t}\tilde\vphi\notag\\
&\to \int_0^T\int_{\Gamma}\theta\tilde\vphi_t+\int_{\Gamma}\theta_0\tilde\vphi(0,\cdot), \hspace{0.5cm}\mathrm{as} \hspace{0.2cm} \Delta t \to 0, \ \forall \tilde\vphi\in C_c^\infty([0,T)\times\Gamma).
\end{align*}

According to $\eqref{embedding-Delta-t}$, for $\frac{3}{4}<\theta<1$, it holds that 
\begin{align*}
w_{\Delta t}\in W^{1,\infty}(0,T;L^2(\Gamma))\cap L^\infty (0,T;H^2(\Gamma))\hookrightarrow C^{0, 1-\theta}(0,T;C^{1,2\theta-\frac{3}{2}}(\Gamma)), 
\end{align*}
thus 
\begin{align*}
w_{\Delta t}\rightarrow w, \hspace{0.5cm}\mathrm{in}\hspace{0.2cm} C^{\frac{1}{5}}([0,T]\times\Gamma). 
\end{align*}
Moreover, it has 
\begin{align}\label{Omega}
|(\Omega_{w_{\Delta t}}\setminus \Omega_w)\cup ( \Omega_w\setminus\Omega_{w_{\Delta t}} )|\to 0, \hspace{0.5cm}\mathrm{in}\hspace{0.2cm} C^{\frac{1}{5}}([0,T]),
\end{align}
and
\begin{align}
&g^\omega_{w_{\Delta t}}\to g^\omega_{w}, \hspace{0.5cm}\mathrm{in} \hspace{0.5cm} C^{\frac{1}{5}}([0,T]\times\mathsf{B}), \hspace{0.5cm}\mathrm{as} \hspace{0.5cm}\Delta t\to 0, \notag \\
&h^{\zeta}_{w_{\Delta t}}\to h^{\zeta}_{w}, \hspace{0.5cm} f^{\lambda}_{w_{\Delta t}}\to f^{\lambda}_{w}, \hspace{0.5cm}\mathrm{in} \hspace{0.5cm} L^\infty([0,T]\times\mathsf{B}), \hspace{0.5cm}\mathrm{as} \hspace{0.5cm}\Delta t\to 0.\label{parameter-Delta-t}
\end{align}

Using $\eqref{Delta-t-uni-bound}_4$, $\eqref{Delta-t-uni-bound}_5$ and $\eqref{Delta-t-uni-bound}_6$, we have 
\begin{align}\label{varrho-vartheta-Delta-t}
&{\varrho_{\Delta t} \rightharpoonup  \varrho, \hspace{0.5cm} \mathrm{weakly *} \hspace{0.2cm}\mathrm{in} \hspace{0.2cm} L^\infty(0,T;L^{\gamma}(\mathsf{B}))},\notag\\
&{b_{\Delta t} \rightharpoonup  b, \hspace{0.5cm} \mathrm{weakly *} \hspace{0.2cm}\mathrm{in} \hspace{0.2cm} L^\infty(0,T;L^2(\mathsf{B}))},\\
&\vartheta_{\Delta t} \rightharpoonup  \vartheta, \hspace{0.5cm} \mathrm{weakly *} \hspace{0.2cm}\mathrm{in} \hspace{0.2cm} L^\infty(0,T;L^{4}(\mathsf{B})),\notag\\
&\vartheta_{\Delta t} \rightharpoonup  \vartheta, \hspace{0.5cm} \mathrm{weakly} \hspace{0.2cm}\mathrm{in} \hspace{0.2cm} L^2(0,T;H^1(\mathsf{B})). \notag
\end{align}

With the aid of the continuity equation $\eqref{weak-continuity}$, $\eqref{weak-magnetic-induction}$ and $\eqref{Delta-t-uni-bound}_4$, and by applying Arzel\`{a}-Ascoli theorem, we can then strengthen $\eqref{varrho-vartheta-Delta-t}_1$ and $\eqref{varrho-vartheta-Delta-t}_2$ to 
  \begin{align*}
    {\varrho_{\Delta t} \rightharpoonup \varrho \hspace{0.5cm}\mathrm{in}\hspace{0.2cm} C_{\mathrm{weak}}([0,T];L^{\gamma}(\mathsf{B}))},
\end{align*}
and 
\begin{align*}
       { b_{\Delta t}\rightharpoonup b \hspace{0.5cm}\mathrm{in}\hspace{0.2cm} C_{\mathrm{weak}}([0,T];L^{2}(\mathsf{B}))}.
\end{align*}
It follows from $\eqref{Delta-t-bu}$ that 
\begin{align*}
    \bu_{\Delta t} \rightharpoonup \bu, \hspace{0.5cm} \mathrm{weakly} \hspace{0.1cm}\mathrm{in} \hspace{0.1cm}L^2(0,T;H_0^1(\mathsf{B})).
\end{align*}

Now, by from  the bound \eqref{Delta-t-b-bu-bound}, one has  
\begin{align*}
    b_{\Delta t}\bu_{\Delta t}\rightharpoonup b\bu \hspace{0.5cm}\mathrm{weakly}\hspace{0.2cm}\mathrm{in}\hspace{0.2cm} L^2(0,T;L^{\frac{2p}{p+2}}(\mathsf{B})), 
\end{align*}
for any $p\geq 2$.

On the other hand, based on the bound \eqref{Delta-t-varrho-bu}, we have  
\begin{align*}
    \varrho_{\Delta t}\bu_{\Delta t}\rightharpoonup \varrho\bu \hspace{0.5cm}\mathrm{weakly *}\hspace{0.2cm}\mathrm{in}\hspace{0.2cm} L^\infty(0,T;L^{\frac{2\gamma}{\gamma+1}}(\mathsf{B})), 
\end{align*}
together with the momentum equation,  
\begin{align*}
    \varrho_{\Delta t}\bu_{\Delta t}\rightharpoonup \varrho\bu \hspace{0.5cm}\mathrm{in}\hspace{0.1cm} C_{\mathrm{weak}}(0,T;L^{\frac{2\gamma}{\gamma+1}}(\mathsf{B})) .  
\end{align*}
Using the compact embedding $L^{\frac{2\gamma}{\gamma+1}}(\mathsf{B})$ into $H^{-1}(\mathsf{B})$, we further derive that 
\begin{align*}
    \varrho_{\Delta t}\bu_{\Delta t}\rightarrow \varrho\bu \hspace{0.5cm}\mathrm{in}\hspace{0.1cm} C_{\mathrm{weak}}(0,T; H^{-1}(\mathsf{B})), 
\end{align*}

 We further  have (by means of the bound \eqref{Delta-t-varrho-bu-2}) 
\begin{align*}
    \varrho_{\Delta t}\bu_{\Delta t}\otimes\bu_{\Delta t}\rightharpoonup \varrho\bu\otimes \bu \hspace{0.5cm}\mathrm{weakly}\hspace{0.1cm}\mathrm{in}\hspace{0.1cm} L^2(0,T;L^{\frac{2p\gamma}{(2+p)\gamma+p}}(\mathsf{B})),
\end{align*}
for all $p>\frac{2\gamma}{\gamma-1}$.

\subsection{Weak convergence of the pressure}
Based on the estimates $\eqref{pressure-int}$, $\eqref{Omega}$, $\eqref{parameter-Delta-t}$, and the definition
  \begin{align*}
   \widetilde p_{M}(\varrho,\vartheta,b):=&p_M(\varrho,\vartheta)+\frac{1}{2}b^2=\varrho^\gamma+\varrho\vartheta+\frac{1}{2}b^2,
\end{align*}
we obtain  
  \begin{align*}
&p^{\lambda,\delta}_{w_{\Delta t}}(\varrho_{\Delta t},\vartheta_{\Delta t}, b_{\Delta t}):=\widetilde p_M(\varrho_{\Delta t},\vartheta_{\Delta t}, b_{\Delta t})+\frac{a^\lambda_{w_{\Delta t}}}{3}\vartheta_{\Delta t}^4+\delta(\varrho_{\Delta t}+b_{\Delta t})^\beta \notag\\
&\hspace{3.5cm}\rightharpoonup \overline{\widetilde p_{M}(\varrho,\vartheta, b)}+\frac{a_{w}^{\lambda}}{3}\overline{\vartheta^4}+\delta\overline{(\varrho+b)^\beta} \hspace{0.5cm}\mathrm{weakly}\hspace{0.1cm}\mathrm{in}\hspace{0.1cm} L^1(\mathcal{O}) , 
\end{align*} 
for any subset $\mathcal{O}\Subset [0,T]\times\mathsf{B}$ such that $\mathcal{O} \cap  ((0,T) \times \Gamma_{w_{\Delta t}}(t)) = \emptyset$. 
\emph{Here and in the sequel, we use ``bar''  notation to denote a weak limit of a composed or nonlinear function}. The method we adopt here is same as the result in \cite[Section 3.8]{MMNRT-22}. More accurately, choosing a sequence of compact sets $\{\mathcal{O}_i\}_{i\in \mathbb{N}}$ such that 
\begin{align*}
    \mathcal{O}_i\cap   ( (0,T) \times  \Gamma_{w_{\Delta t}}(t) )   = \emptyset , \hspace{0.3cm}\mathcal{O}_i\subset\mathcal{O}_{i+1},\hspace{0.3cm}  \mathcal{O}_i\to[0,T]\times\mathsf{B} \hspace{0.5cm}\mathrm{as} \hspace{0.1cm} i\rightarrow \infty, 
\end{align*}
then we try to prove a weak limit of $p^{\lambda,\delta}_{w_{\Delta t}}(\varrho_{\Delta t},\vartheta_{\Delta t}, b_{\Delta t})$ on the whole domain $[0,T]\times\mathsf{B}$. The only difficulty for us is to preclude the concentration at the moving boundary $\Gamma_{w}(t)$, and this could be done through using  \cite[Lemma 6.4]{Breit-Sebastian-ARMA}.
\begin{lemma}\label{exc-pre-1} 
    For any $\epsilon>0$, there exists $\widetilde{\Delta t}>0$ and $\mathcal{A}_{\epsilon}\Subset [0,T]\times\mathsf{B}$ such that for all $\Delta t <\widetilde{\Delta t}$, it holds that 
    \begin{align*}
    \mathcal{A}_{\epsilon}\cap ([0,T]\times \Gamma_{w_{\Delta t}}(t)) = \emptyset, \hspace{1cm}\int_{((0,T)\times \mathsf{B})\setminus \mathcal{A}_{\epsilon}} p^{\lambda,\delta}_{w_{\Delta t}}(\varrho_{\Delta t},\vartheta_{\Delta t}, b_{\Delta t})\leq \epsilon.
    \end{align*}
\end{lemma}

\subsection{Strong convergence of the temperature}\label{Section-strong-conv-temp-Delta-t} 
In order to apply the Div-Curl lemma (see for instance \cite[Theorem 10.12]{FN-09}, or \cite{Ta-79}), we set 
\begin{align}
&{\bf{Z}}_{\Delta t}:=\bigg(\varrho_{\Delta t}s^{\lambda}_{w_{\Delta t}}(t)(\varrho_{\Delta t}, \vartheta_{\Delta t}), \hspace{0.3cm}\varrho_{\Delta t}s^{\lambda}_{w_{\Delta t}}(\varrho_{\Delta t}, \vartheta_{\Delta t}) \bu_{\Delta t}-\frac{\kappa^{\zeta}_{w_{\Delta t}}(\vartheta_{\Delta t})\nabla\vartheta_{\Delta t}}{\vartheta_{\Delta t}}\bigg),\label{Z-fun}\\
&{{\bf{G}}_{\Delta t}}:=(G(\vartheta_{\Delta t}),0,0), \label{G-fun}
\end{align}
where $G(\cdot)$ in $\eqref{G-fun}$ is a bounded and Lipschitz function on $[0,\infty)$. 

To avoid dealing with the limit on the boundary $\Gamma_{w_{\Delta t}}(t)$, let us consider a domain $\mathcal{O}\Subset (0,T)\times\mathsf{B}$ such that $\mathcal{O}\cap ((0,T)\times \Gamma_{w_{\Delta t}}(t))   =\emptyset$ for all $\Delta t\leq \widetilde{\Delta t}$, where $\widetilde{\Delta t}$ is specified in Lemma \ref{exc-pre-1}.

According to $\eqref{Delta-t-entropy}$ and $\eqref{Delta-t-uni-bound}_7$, we know that ${\bf{Z}}_{\Delta t}\in L^p(\mathcal{O})$ for some $p>1$, and $\mathrm {Curl}_{t,x}{\bf{G}}_{\Delta t}\in L^2(\mathcal{O})$ is compact in $W^{-1, q}(\mathcal{O})$ for some $q>1$. We further prove the pre-compactness of $\mathrm{Div}_{t,x}{\bf{Z}}_{\Delta t}$. Precisely, for any $\psi\in C_c^\infty(\mathcal{O})$, it follows from $\eqref{entropy-ine-B}$ and $\eqref{Delta-t-uni-bound}_6$ that 
\begin{align*}
\int_{\mathcal{O}}\mathrm{Div}_{t,x}{\bf{Z}}_{\Delta t}\psi=-\langle\mathcal{D}^{\omega,\zeta}_{w_{\Delta t}};\psi\rangle_{[\mathcal{M};C](\mathcal{O})}+\int_{\mathcal{O}}\zeta \vartheta^4_{\Delta t}\psi\leq C\|\psi\|_{C(\mathcal{O})}\leq C\|\psi\|_{W^{k,l}(\mathcal{O})}, 
\end{align*}
 for some $k>1$ and $l>3$. Thus, it has 
\begin{align*}
    \mathrm{Div}_{t,x}{\bf{Z}}_{\Delta t}\in W^{-k, l^*}(\mathcal{O})\Subset W^{-1, l'}(\mathcal{O}), \hspace{0.3cm}\mathrm{with}\hspace{0.3cm}  l'\in (1,\frac{3}{2}). 
\end{align*}
Since, for any domain $\tilde{\mathcal{O}}\Subset (0,T)\times\mathsf{B}$ satisfying $\tilde{\mathcal{O}}\cap ((0,T)\times \Gamma_{w_{\Delta t}}(t))  =\emptyset$, we can find a domain $\mathcal O$ as above such that  $\tilde{\mathcal{O}}\subset \mathcal{O}$ (for any $\Delta t \leq \widetilde{\Delta t}$),   we can conclude by Div-Curl lemma that
\begin{align}\label{Delta-t-s-G}
\overline{\varrho s^{\lambda}_{w}(\varrho,\vartheta)G(\varrho)}=\overline{\varrho s^{\lambda}_{w}(\varrho,\vartheta)}\hspace{0.1cm} \overline{G(\varrho)}, \hspace{0.3cm} \mathrm{a.e.}\hspace{0.1cm}\mathrm{in}\hspace{0.1cm} (0,T)\times\mathsf{B}. 
  \end{align}

In order to get the strong convergence of $\vartheta_{\Delta t}$, we first use the parameterized Young measures theory (see \cite[Theorem 10.21]{FN-09}) to get the following fact
\begin{align}\label{Delta-t-s-M-G}
\overline{\varrho s_M(\varrho,\vartheta)G(\varrho)}\geq \overline{\varrho s_M(\varrho,\vartheta)}\overline{G(\varrho)}, \hspace{0.2cm} \overline{\vartheta^3G(\vartheta)}\geq \overline{\vartheta^3}\hspace{0.1cm}\overline{G(\vartheta)} \hspace{0.3cm} \mathrm{a.e.}\hspace{0.1cm}\mathrm{in}\hspace{0.1cm} (0,T)\times\mathsf{B}.
  \end{align}
We then use $\eqref{Delta-t-s-G}$, $\eqref{Delta-t-s-M-G}$, and  $\eqref{entropy-w-lambda}$ to get 
\begin{align*}
\overline{\vartheta^4}= \overline{\vartheta^3}\vartheta.
\end{align*}
Subsequently, we use the monotonicity of $\vartheta^4$ to get 
\begin{align}\label{strong-convergence-vartheta-t}
\vartheta_{\Delta t}\to \vartheta \hspace{0.3cm} \mathrm{a.e.}\hspace{0.1cm} \mathrm{in}\hspace{0.1cm} (0,T)\times\mathsf{B}.
\end{align}

\subsection{Continuity of velocity and temperature on the interface} 
We need to proof the limits of $w$, $\bu$, $\theta$, $\vartheta$ derived in Section $\ref{limit-Delta-t}$ satisfy the property of the continuity on the interface in the sense of Definition $\ref{fixed-domain}$. Referring to $\eqref{Delta-t-uni-bound}_1$, we observe that $\bu_{\Delta t}|_{\Gamma_{w_{\Delta t}}}-(w_{\Delta t })_t\bfn \rightarrow 0$ and $\vartheta_{\Delta t}|_{\Gamma_{w_{\Delta t}}}-\theta_{\Delta t} \rightarrow 0$ as $\Delta t \rightarrow 0$. According to $\eqref{Delta-w-theta}_3$ and $\eqref{Delta-w-theta}_4$, it remains to proof that $(\bu_{\Delta t}) \circ \widetilde{\bmphi_{w_{\Delta t}}} \rightharpoonup \bu \circ \widetilde{\bmphi_{w}}$ and $(\vartheta_{\Delta t}) \circ \widetilde{\bmphi_{w_{\Delta t}}} \rightharpoonup \vartheta \circ \widetilde{\bmphi_{w}}$ in $L^1((0,T)\times\mathbb{R}^2)$ when $\bu_{\Delta t}$ and $\vartheta_{\Delta t}$ are extended by zero in $\mathbb{R}^2\setminus \mathsf{B}$. The idea of the proof will be similar to the analysis presented in Section \ref{Sec-con-interface} where the limit of $\delta\to 0$ is analyzed.


\subsection{Strong convergence of the fluid density and magnetic field}\label{Section-strong-conv-dens-mag-Delta-t}
For the convenience of notation, we define 
$$\mathsf{d}_{\Delta t}=\varrho_{\Delta t}+b_{\Delta t}, \ \  \mathsf{d}=\varrho + b, $$
and 
$$(\mathcal{R}_{\varrho_{\Delta t}},\mathcal{R}_{b_{\Delta t}})=\left(\frac{\varrho_{\Delta t}}{\mathsf{d}_{\Delta t}}, \frac{b_{\Delta t}}{\mathsf{d}_{\Delta t}}\right) \  \textnormal{ if } \mathsf{d}_{\Delta t} \neq 0, \ \ (\mathcal{R}_{\varrho},\mathcal{R}_{b})=\left(\frac{\varrho}{\mathsf{d}}, \frac{b}{\mathsf{d}}\right) \  \textnormal{ if } \mathsf{d} \neq 0 .$$ 
Thus, it has $0\leq \mathcal{R}_{\varrho_{\Delta t}},\mathcal{R}_{b_{\Delta t}}, \mathcal{R}_{\varrho},\mathcal{R}_{b}\leq 1$.  Apparently, the following decomposition holds 
\begin{align}\label{D-pressure}
(\varrho_{\Delta t})^\gamma+\frac{1}{2}(b_{\Delta t})^2=&(\mathcal{R}_{\varrho_{\Delta t}}\mathsf{d}_{\Delta t})^\gamma+\frac{1}{2}(\mathcal{R}_{b_{\Delta t}}\mathsf{d}_{\Delta t})^2-\bigg((\mathcal{R}_{\varrho}\mathsf{d}_{\Delta t})^\gamma+\frac{1}{2}(\mathcal{R}_{b}\mathsf{d}_{\Delta t})^2\bigg)\\
&+\bigg((\mathcal{R}_{\varrho}\mathsf{d}_{\Delta t})^\gamma+\frac{1}{2}(\mathcal{R}_{b}\mathsf{d}_{\Delta t})^2\bigg),\notag
\end{align}
\begin{align}\label{sum-b-varrho}
\varrho_{\Delta t}+b_{\Delta t}=(\mathcal{R}_{\varrho}+\mathcal{R}_{b})\mathsf{d}_{\Delta t}+(\mathcal{R}_{\varrho_{\Delta t}}-\mathcal{R}_{\varrho}+\mathcal{R}_{b_{\Delta t}}-\mathcal{R}_{b})\mathsf{d}_{\Delta t}. 
\end{align}

\begin{lemma}
    For any $t\in (0,T]$, domain $\mathcal O \Subset [0,t] \times \mathsf{B}$  such that $\mathcal O \cap ((0,t) \times \Gamma_{w_{\Delta t}}) = \emptyset$, and $\varphi \in C(\overline{\mathcal O})$ with $\varphi\geq 0$, it holds that 
 \begin{align}\label{weak-press}
   & \lim_{\Delta t\rightarrow 0}\iint_{\mathcal{O} }\varphi(\varrho+b)\bigg((\varrho_{\Delta t})^\gamma+\frac{1}{2}(b_{\Delta t})^2\bigg)\notag\\
   &\leq \lim_{\Delta t\rightarrow 0}   \iint_{\mathcal{O}} \varphi(\varrho_{\Delta t}+b_{\Delta t}) \bigg((\varrho_{\Delta t})^\gamma+\frac{1}{2}(b_{\Delta t})^2\bigg).
    \end{align}
    
\end{lemma}
\begin{proof}
Using $\eqref{D-pressure}$ and $\eqref{sum-b-varrho}$, we have 
\begin{equation}\label{Eq-I-i}
  \begin{aligned}
&\iint_{\mathcal{O}}\varphi(\varrho_{\Delta t}+b_{\Delta t}) \bigg((\varrho_{\Delta t})^\gamma+\frac{1}{2}(b_{\Delta t})^2\bigg)\\
&=\iint_{\mathcal{O}}\varphi(\mathcal{R}_{\varrho}+\mathcal{R}_{b})\mathsf{d}_{\Delta t}\bigg((\mathcal{R}_{\varrho}\mathsf{d}_{\Delta t})^\gamma+\frac{1}{2}(\mathcal{R}_{b}\mathsf{d}_{\Delta t})^2\bigg)\\
&+\iint_{\mathcal{O}}\varphi(\mathcal{R}_{\varrho_{\Delta t}} -\mathcal{R}_{\varrho}+\mathcal{R}_{b_{\Delta t}}-\mathcal{R}_{b})\mathsf{d}_{\Delta t}\bigg((\mathcal{R}_{\varrho}\mathsf{d}_{\Delta t})^\gamma+\frac{1}{2}(\mathcal{R}_{b}\mathsf{d}_{\Delta t})^2\bigg)\\
&+\iint_{\mathcal{O}}\varphi(\varrho_{\Delta t}+b_{\Delta t})\bigg[(\mathcal{R}_{\varrho_{\Delta t}}\mathsf{d}_{\Delta t})^\gamma+\frac{1}{2}(\mathcal{R}_{b_{\Delta t}}\mathsf{d}_{\Delta t})^2-\bigg((\mathcal{R}_{\varrho}\mathsf{d}_{\Delta t})^\gamma+\frac{1}{2}(\mathcal{R}_{b}\mathsf{d}_{\Delta t})^2\bigg)\bigg]\\
&:=\sum_{i=1}^3\mathcal{I}_{i}.
\end{aligned}
\end{equation} 

We shall estimate the terms $\mathcal I_2$ and $\mathcal I_3$. Since we have assumed that $\beta>\{\gamma, 2\}$,  there exists  sufficiently large integer $k_1$ such that $\beta \geq \bigg\{\gamma\frac{k_1}{k_1-1}, 2\frac{k_1}{k_1-1}\bigg\}$; see for instance, \cite{Wen-21,VWY-19}, where such result has also been utilized.
 Moreover, it follows from $\eqref{pressure-int}$ that 
  \begin{align}\label{d-beta}
\iint_{\mathcal{O}}\mathsf{d}_{\Delta t}\bigg|\mathsf{d}_{\Delta t}^{\gamma\frac{k_1}{k_1-1}}+\mathsf{d}_{\Delta t}^{2\frac{k_1}{k_1-1}}\bigg|\leq  C.
\end{align} 
    
We use H\"older's inequality, $\eqref{d-beta}$, and the property $0\leq \mathcal{R}_{\varrho}, \mathcal{R}_b\leq1$ to derive 
 \begin{align*}
|\mathcal{I}_2|\leq &C \bigg(\iint_{\mathcal O}\mathsf{d}_{\Delta t}|\mathcal{R}_{\varrho_{\Delta t}}-\mathcal{R}_{\varrho}|^{k_1}\mathsf{d}_{\Delta t}\bigg)^{\frac{1}{k_1}}\bigg[\iint_{\mathcal{O}}\mathsf{d}_{\Delta t}\bigg((\mathcal{R}_{\varrho}\mathsf{d}_{\Delta t})^\gamma+(\mathcal{R}_{b}\mathsf{d}_{\Delta t})^2\bigg)^{\frac{k_1}{k_1-1}}\bigg]^{\frac{k_1-1}{k_1}}\\
&+C\bigg(\iint_{\mathcal{O}}\mathsf{d}_{\Delta t}|\mathcal{R}_{b_{\Delta t}} -\mathcal{R}_{b}|^{k_1}\mathsf{d}_{\Delta t}\bigg)^{\frac{1}{k_1}}\bigg[\iint_{\mathcal{O}}\mathsf{d}_{\Delta t}\bigg((\mathcal{R}_{\varrho}\mathsf{d}_{\Delta t})^\gamma+(\mathcal{R}_{b}\mathsf{d}_{\Delta t})^2\bigg)^{\frac{k_1}{k_1-1}}\bigg]^{\frac{k_1-1}{k_1}}\\
\leq &C\bigg(\iint_{\mathcal{O}}\mathsf{d}_{\Delta t}|\mathcal{R}_{\varrho_{\Delta t}}-\mathcal{R}_{\varrho}|^{k_1}\mathsf{d}_{\Delta t}\bigg)^{\frac{1}{k_1}}\bigg(\iint_{\mathcal{O}}\mathsf{d}_{\Delta t}\bigg|\mathsf{d}_{\Delta t}^{\gamma\frac{k_1}{k_1-1}}+\mathsf{d}_{\Delta t}^{2\frac{k_1}{k_1-1}}\bigg|\bigg)^{\frac{k_1-1}{k_1}}\\
&+C\bigg(\iint_{\mathcal{O}}\mathsf{d}_{\Delta t}|\mathcal{R}_{b_{\Delta t}}-\mathcal{R}_{b}|^{k_1}\mathsf{d}_{\Delta t}\bigg)^{\frac{1}{k_1}}\bigg(\iint_{\mathcal{O}}\mathsf{d}_{\Delta t}\bigg|\mathsf{d}_{\Delta t}^{\gamma\frac{k_1}{k_1-1}}+\mathsf{d}_{\Delta t}^{2\frac{k_1}{k_1-1}}\bigg|\bigg)^{\frac{k_1-1}{k_1}}\\
\leq  &C\bigg(\iint_{\mathcal{O}}\mathsf{d}_{\Delta t}|\mathcal{R}_{\varrho_{\Delta t}}-\mathcal{R}_{\varrho}|^{k_1}\mathsf{d}_{\Delta t}\bigg)^{\frac{1}{k_1}}+C\bigg(\iint_{\mathcal{O}}\mathsf{d}_{\Delta t}|\mathcal{R}_{b_{\Delta t}}-\mathcal{R}_{b}|^{k_1}\mathsf{d}_{\Delta t}\bigg)^{\frac{1}{k_1}}.
\end{align*} 

Applying Lemma \ref{strong-con}, as $\Delta t\rightarrow 0$, we can conclude 
\begin{align*}
\bigg(\iint_{\mathcal{O}}\mathsf{d}_{\Delta t}|\mathcal{R}_{\varrho_{\Delta t}}-\mathcal{R}_{\varrho}|^{k_1}\mathsf{d}_{\Delta t}\bigg)^{\frac{1}{k_1}}\rightarrow 0, \quad \bigg(\iint_{\mathcal{O}}\mathsf{d}_{\Delta t}|\mathcal{R}_{b_{\Delta t}}-\mathcal{R}_{b}|^{k_1}\mathsf{d}_{\Delta t}\bigg)^{\frac{1}{k_1}}\rightarrow 0, 
   \end{align*} 
which gives 
$$\mathcal I_2 \to 0 \text{ as } \Delta t \to 0.$$

Next, we look into the term $\mathcal I_3$. 
Using the mean value theorem and Young's inequality, we get  
\begin{align}\label{minus-R-d}
&\bigg|(\mathcal{R}_{\varrho_{\Delta t}}\mathsf{d}_{\Delta t})^\gamma+\frac{1}{2}(\mathcal{R}_{b_{\Delta t}}\mathsf{d}_{\Delta t})^2-\bigg((\mathcal{R}_{\varrho}\mathsf{d}_{\Delta t})^\gamma+\frac{1}{2}(\mathcal{R}_{b}\mathsf{d}_{\Delta t})^2\bigg)\bigg|\notag\\
&\leq C\bigg(1+(\varrho_{\Delta t}+b_{\Delta t})^\gamma+(\varrho_{\Delta t}+b_{\Delta t})^2\bigg)(|\mathcal{R}_{\varrho_{\Delta t}}\mathsf{d}_{\Delta t}-\mathcal{R}_{\varrho}\mathsf{d}_{\Delta t}|+|\mathcal{R}_{b_{\Delta t}}\mathsf{d}_{\Delta t}-\mathcal{R}_{b}\mathsf{d}_{\Delta t}|)\notag\\
&\leq C\bigg(1+(\mathsf{d}_{\Delta t})^{\tilde{\gamma}}\bigg)(|\mathcal{R}_{\varrho_{\Delta t}}\mathsf{d}_{\Delta t}-\mathcal{R}_{\varrho}\mathsf{d}_{\Delta t}|+|\mathcal{R}_{b_{\Delta t}}\mathsf{d}_{\Delta t}-\mathcal{R}_{b}\mathsf{d}_{\Delta t}|), 
\end{align} 
where $\tilde{\gamma}=\max\{\gamma, 2\}$.

Under the assumption $\beta> \max\{\gamma+1, 3\}$, we can further find a sufficiently large integer $k_2$ such that $\beta+1>\bigg(\tilde{\gamma}+2-\frac{1}{k_2}\bigg)\frac{k_2}{k_2-1}$.
Then, 
by $\eqref{minus-R-d}$, $\eqref{d-beta}$, and H\"older's inequality, it holds that 
\begin{align*}
\mathcal{I}_3\leq & C \iint_{\mathcal{O}}\bigg(1+(\mathsf{d}_{\Delta t})^{\tilde{\gamma}+1}\bigg)(|\mathcal{R}_{\varrho_{\Delta t}}\mathsf{d}_{\Delta t}-\mathcal{R}_{\varrho}\mathsf{d}_{\Delta t}|+|\mathcal{R}_{b_{\Delta t}}\mathsf{d}_{\Delta t}-\mathcal{R}_{b}\mathsf{d}_{\Delta t}|)\\
\leq &C\iint_{\mathcal{O}}\bigg((\mathsf{d}_{\Delta t})^{\frac{1}{2}}(\mathsf{d}_{\Delta t})^{\frac{1}{2}}+(\mathsf{d}_{\Delta t})^{\tilde{\gamma}+2-\frac{1}{k_2}}(\mathsf{d}_{\Delta t})^{\frac{1}{k_2}}\bigg)|\mathcal{R}_{\varrho_{\Delta t}}-\mathcal{R}_{\varrho}|\\
&+C\iint_{\mathcal{O}}\bigg((\mathsf{d}_{\Delta t})^{\frac{1}{2}}(\mathsf{d}_{\Delta t})^{\frac{1}{2}}+(\mathsf{d}_{\Delta t})^{\tilde{\gamma}+2-\frac{1}{k_2}}(\mathsf{d}_{\Delta t})^{\frac{1}{k_2}}\bigg)|\mathcal{R}_{b_{\Delta t}}-\mathcal{R}_{b}|\\
\leq &C\bigg(\iint_{\mathcal{O}}\mathsf{d}_{\Delta t}\bigg)^{\frac{1}{2}}\bigg(\iint_{\mathcal{O}}\mathsf{d}_{\Delta t}|\mathcal{R}_{\varrho_{\Delta t}}-\mathcal{R}_{\varrho}|^2\bigg)^{\frac{1}{2}}\\
&+C\bigg(\iint_{\mathcal{O}}(\mathsf{d}_{\Delta t})^{(\tilde{\gamma}+2-\frac{1}{k_2})\frac{k_2}{k_2-1}}\bigg)^{\frac{k_2-1}{k_2}}\bigg(\iint_{\mathcal{O}}\mathsf{d}_{\Delta t}|\mathcal{R}_{\varrho_{\Delta t}}-\mathcal{R}_{\varrho}|^{k_2}\bigg)^{\frac{1}{k_2}}\\
&+C\bigg(\iint_{\mathcal{O}}\mathsf{d}_{\Delta t}\bigg)^{\frac{1}{2}}\bigg(\iint_{\mathcal{O}}\mathsf{d}_{\Delta t}|\mathcal{R}_{b_{\Delta t}}-\mathcal{R}_{b}|^2\bigg)^{\frac{1}{2}}\\
&+C\bigg(\iint_{\mathcal{O}}(\mathsf{d}_{\Delta t})^{(\tilde{\gamma}+2-\frac{1}{k_2})\frac{k_2}{k_2-1}}\bigg)^{\frac{k_2-1}{k_2}}\bigg(\iint_{\mathcal{O}}\mathsf{d}_{\Delta t}|\mathcal{R}_{b_{\Delta t}}-\mathcal{R}_{b}|^{k_2}\bigg)^{\frac{1}{k_2}}\\
\leq &C\bigg(\iint_{\mathcal{O}}\mathsf{d}_{\Delta t}|\mathcal{R}_{\varrho_{\Delta t}}-\mathcal{R}_{\varrho}|^2\bigg)^{\frac{1}{2}}+C\bigg(\iint_{\mathcal{O}}\mathsf{d}_{\Delta t}|\mathcal{R}_{\varrho_{\Delta t}}-\mathcal{R}_{\varrho}|^{k_2}\bigg)^{\frac{1}{k_2}}\\
&+C\bigg(\iint_{\mathcal{O}}\mathsf{d}_{\Delta t}|\mathcal{R}_{b_{\Delta t}}-\mathcal{R}_{b}|^2\bigg)^{\frac{1}{2}}+C\bigg(\iint_{\mathcal{O}}\mathsf{d}_{\Delta t}|\mathcal{R}_{b_{\Delta t}}-\mathcal{R}_{b}|^{k_2}\bigg)^{\frac{1}{k_2}}\\
\rightarrow & \,  0 , 
\end{align*} 
as $\Delta t\rightarrow 0$, where we have also used  Lemma \ref{strong-con}.

Thus, collecting above estimates in \eqref{Eq-I-i}, we get 
 \begin{align*} 
&\lim_{\Delta t\rightarrow 0}\iint_{\mathcal{O}}\varphi(\varrho_{\Delta t}+b_{\Delta t}) \bigg((\varrho_{\Delta t})^\gamma+\frac{1}{2}(b_{\Delta t})^2\bigg)\\
&=\lim_{\Delta t\rightarrow 0}\iint_{\mathcal{O}}\varphi(\mathcal{R}_{\varrho}+\mathcal{R}_{b})\mathsf{d}_{\Delta t}\bigg((\mathcal{R}_{\varrho}\mathsf{d}_{\Delta t})^\gamma+\frac{1}{2}(\mathcal{R}_{b}\mathsf{d}_{\Delta t})^2\bigg). 
 \end{align*} 

Based on the fact that the functions $z\mapsto z$ and $z \mapsto (\mathcal{R}_{\varrho}z)^\gamma +(\mathcal{R}_{b}z)^2$ are non-decreasing, and using the notation $\overline{\overline{g(\cdot)}}$ to denote the weak limit of $g(\cdot)$ with respect to $\mathsf{d}_{\Delta t}$ as $\Delta t\rightarrow 0$, we get 
 \begin{align}\label{multiply-varho-b-1} 
&\lim_{\Delta t\rightarrow 0}\iint_{\mathcal{O}}\varphi(\mathcal{R}_{\varrho}+\mathcal{R}_{b})\mathsf{d}_{\Delta t}\bigg((\mathcal{R}_{\varrho}\mathsf{d}_{\Delta t})^\gamma+\frac{1}{2}(\mathcal{R}_{b}\mathsf{d}_{\Delta t})^2\bigg)\notag\\
&=\iint_{\mathcal{O}}\varphi(\mathcal{R}_{\varrho}+\mathcal{R}_{b})\overline{\overline{\mathsf{d}\bigg((\mathcal{R}_{\varrho}\mathsf{d})^\gamma+\frac{1}{2}(\mathcal{R}_{b}\mathsf{d})^2\bigg)}}\notag\\
&\geq\iint_{\mathcal{O}}\varphi(\mathcal{R}_{\varrho}+\mathcal{R}_{b})\mathsf{d}\overline{\overline{\bigg((\mathcal{R}_{\varrho}\mathsf{d})^\gamma+\frac{1}{2}(\mathcal{R}_{b}\mathsf{d})^2\bigg)}}\\
&=\iint_{\mathcal{O}}\varphi(\varrho+b)\overline{\overline{\bigg((\mathcal{R}_{\varrho}\mathsf{d})^\gamma+\frac{1}{2}(\mathcal{R}_{b}\mathsf{d})^2\bigg)}}.\notag
 \end{align} 

Moreover, it holds that 
 \begin{align}\label{multiply-varho-b-2} 
\iint_{\mathcal{O}} \varphi(\varrho+b)\overline{\overline{\bigg((\mathcal{R}_{\varrho}\mathsf{d})^\gamma+\frac{1}{2}(\mathcal{R}_{b}\mathsf{d})^2\bigg)}}=
\iint_{\mathcal{O}}\varphi(\varrho+b)\overline{\bigg(\varrho^\gamma+\frac{1}{2}b^2\bigg)}, 
 \end{align} 
which is based on the following analysis
 \begin{align*}
&\iint_{\mathcal{O}} \varphi(\varrho+b)\overline{\overline{\bigg((\mathcal{R}_{\varrho}\mathsf{d})^\gamma+\frac{1}{2}(\mathcal{R}_{b}\mathsf{d})^2\bigg)}}\\
&=\lim_{\Delta t\rightarrow 0 }\iint_{\mathcal{O}}\varphi(\varrho+b)\bigg((\mathcal{R}_{\varrho}\mathsf{d}_{\Delta t})^\gamma+\frac{1}{2}(\mathcal{R}_{b}\mathsf{d}_{\Delta t})^2\bigg)\\
&=\lim_{\Delta t\rightarrow 0 }\iint_{\mathcal{O}}\varphi(\varrho+b)\bigg((\varrho_{\Delta t})^\gamma+\frac{1}{2}(b_{\Delta t})^2\bigg)\\
&\hspace{0.5cm}+\underbrace{\lim_{\Delta t\rightarrow 0 }\iint_{\mathcal{O}}\varphi(\varrho+b)\bigg[\bigg((\mathcal{R}_{\varrho}\mathsf{d}_{\Delta t})^\gamma+\frac{1}{2}(\mathcal{R}_{b}\mathsf{d}_{\Delta t})^2\bigg)-\bigg((\varrho_{\Delta t})^\gamma+\frac{1}{2}(b_{\Delta t})^2\bigg)\bigg]}_{\mathcal{J}}.
 \end{align*} 
We mimic the estimate in $\mathcal{I}_3$, then derive that $\mathcal{J} \rightarrow 0$ as $\Delta t\rightarrow 0$. Thus, one  has $\eqref{multiply-varho-b-2}$. Combining it with $\eqref{multiply-varho-b-1}$, we conclude the proof.    
\end{proof}

A weak compactness identity for effective pressure $\widetilde p_M(\varrho, \vartheta, b)+\delta(\varrho+b)^{\beta}-(\mu(\vartheta)+\eta(\vartheta))\mathrm{div}\bu$ is established as in the following lemma. Since the proof is similar to \cite[Section 3.6.5]{FN-09}, we omit it for brevity. Only note that for test function $\psi\in C_c^\infty((0,T)\times(\mathsf{B}\setminus\Gamma_{w_{\Delta t}}(t)))$, there exists $\widetilde{\Delta t}$ such that for all $\Delta t\leq \widetilde{\Delta t}$ it holds  $\psi\in C_c^\infty((0,T)\times(\mathsf{B}\setminus\Gamma_{w_{\Delta t}}(t)))$.

\begin{lemma}
 For any 
$\psi\in C_c^\infty((0,T)\times(\mathsf{B}\setminus\Gamma_{w_{\Delta t}}(t)))$, it holds that 
  \begin{align}\label{Delta-t-EVF}
&\lim_{\Delta t\to 0}\int_{[0,T]\times\mathsf{B}}(\widetilde p_M(\varrho_{\Delta t}, \vartheta_{\Delta t}, b_{\Delta t})+\delta(\varrho_{\Delta t}+b_{\Delta t})^\beta -(\mu(\vartheta_{\Delta t})+\eta(\vartheta_{\Delta t}))\mathrm{div}\bu_{\Delta t})(\varrho_{\Delta t}+b_{\Delta t})\psi\notag \\
&=\int_{[0,T]\times\mathsf{B}}(\overline{\widetilde p_M(\varrho, \vartheta, b)+\delta(\varrho+b)^\beta}-(\mu(\vartheta)+\eta(\vartheta))\mathrm{div}\bu)(\varrho+b)\psi. 
\end{align}
\end{lemma}

\vspace{.5cm}

Now, going back to $\eqref{Delta-t-EVF}$, we rewrite it as 
\begin{align}\label{Delta-t-EVF-2}
&\overline{[\widetilde p_M(\varrho, \vartheta, b)+\delta(\varrho+b)^{\beta}](\varrho+b)}- \overline{[\widetilde p_M(\varrho_, \vartheta, b)+\delta(\varrho+b)^{\beta}]}\hspace{0.1cm} (\varrho+b)\notag\\
&=(\mu(\vartheta)+\eta(\vartheta))(\overline{\mathrm{div}\bu(\varrho+b)}-\mathrm{div}\bu(\varrho+b)),\hspace{0.3cm} \mathrm{a.e.}\hspace{0.1cm} \mathrm{in}\hspace{0.1cm} (0,T)\times\mathsf{B}.
\end{align}
Using $\eqref{weak-press}$ and the strong convergence of $\vartheta_{\Delta t}$ in $\eqref{strong-convergence-vartheta-t}$, we obtain that 
\begin{align}\label{p-M-weak}
\overline{[\widetilde p_M(\varrho_, \vartheta, b)+\delta(\varrho+b)^{\beta}](\varrho+b)}- \overline{[\widetilde p_M(\varrho_, \vartheta, b)+\delta(\varrho+b)^{\beta}]}\hspace{0.1cm} (\varrho+b)\geq 0, \hspace{0.3cm} \mathrm{a.e.}\hspace{0.1cm} \mathrm{in}\hspace{0.1cm} (0,T)\times\mathsf{B}.
\end{align}
Furthermore, we combine $\eqref{Delta-t-EVF-2}$ and $\eqref{p-M-weak}$ to get 
\begin{align}\label{div-u-varrho-b}
\overline{\mathrm{div}\bu(\varrho+b)}\geq \mathrm{div}\bu(\varrho+b), \hspace{0.3cm} \mathrm{a.e.}\hspace{0.1cm} \mathrm{in}\hspace{0.1cm} (0,T)\times\mathsf{B}.
\end{align}

Actually, the limits $(\varrho, \bu)$ and $(b, \bu)$ are solutions to the corresponding renormalized equations in a fixed domain. Inspired by \cite[Lemma 4.4]{VWY-19}, we can also control $\varrho_{\Delta t}$ and $b_{\Delta t}$ in the form of $(\cdot)\log(\cdot)$ as follows:
\begin{align}\label{log-varrho-b-t}
&\int_{\mathsf{B}}(\varrho_{\Delta t}\log\varrho_{\Delta t}-\varrho\log\varrho+b_{\Delta t}\log b_{\Delta t}-b\log b)(t,\cdot)\notag\\
&\leq \int_0^t\int_{\mathsf{B}}(\varrho+b)\mathrm{div}\bu-\int_0^t\int_{\mathsf{B}}(\varrho_{\Delta t}+b_{\Delta t})\mathrm{div}\bu_{\Delta t},
\end{align}
for almost every $t\in(0,T)$.

Letting $\Delta t\rightarrow 0$ in $\eqref{log-varrho-b-t}$, and then using $\eqref{div-u-varrho-b}$, we derive 
\begin{align}\label{blogb-negative}
\int_{\mathsf{B}}(\overline{\varrho\log\varrho}-\varrho\log\varrho+\overline{b\log b}-b\log b)(t,\cdot)\leq 0.
\end{align}
The convexity of $z\mapsto z\log z$ leads to 
\begin{align}\label{blogb-positive}
\overline{\varrho\log\varrho}\geq \varrho\log\varrho, \quad \overline{b\log b}\geq b\log b.
\end{align}

It follows from $\eqref{blogb-negative}$ and $\eqref{blogb-positive}$ that 
\begin{align*}
\int_{\mathsf{B}}(\overline{\varrho\log\varrho}-\varrho\log\varrho+\overline{b\log b}-b\log b)(t,\cdot)= 0.
\end{align*}

Moreover, it holds that 
\begin{align*}
\overline{\varrho\log\varrho}=\varrho\log\varrho,\quad \overline{b\log b}=b\log b ,\hspace{0.3cm} \mathrm{a.e.}\hspace{0.1cm} \mathrm{in}\hspace{0.1cm} (0,T)\times\mathsf{B}, 
\end{align*}
which gives 
\begin{align}\label{varhho-b-t}
\varrho_{\Delta t}\rightarrow \varrho,\quad  b_{\Delta t }\rightarrow b, \hspace{0.3cm} \mathrm{a.e.}\hspace{0.1cm} \mathrm{in}\hspace{0.1cm} (0,T)\times\mathsf{B}.
\end{align}

With $\eqref{strong-convergence-vartheta-t}$ and $\eqref{varhho-b-t}$ in hand, we have 
\begin{align*} 
& \vartheta_{\Delta t}^4 \to \vartheta^4 \hspace{0,3cm} \mathrm{a.e.}\hspace{0.1cm}\mathrm{in}\hspace{0.1cm} (0,T)\times\mathsf{B},\notag\\
&\frac{\kappa^{\zeta}_{w_{\Delta t}}(\vartheta_{\Delta t})\nabla \vartheta_{\Delta t}}{\vartheta_{\Delta t}}\rightharpoonup \frac{\kappa^{\zeta}_{w}(\vartheta)\nabla \vartheta}{\vartheta}\hspace{0,3cm} \mathrm{weakly}\hspace{0.1cm}\mathrm{in}\hspace{0.1cm} L^{1}((0,T)\times\mathsf{B}),\notag\\
&\bbS_{w_{\Delta t}}^\omega(\vartheta_{\Delta t}, \nabla \bu_{\Delta t})\rightharpoonup \bbS_{w}^\omega(\vartheta, \nabla \bu) \hspace{0,3cm} \mathrm{weakly}\hspace{0.1cm}\mathrm{in}\hspace{0.1cm} L^{1}((0,T)\times\mathsf{B}),\\
&p_{w_{\Delta t}}(\varrho_{\Delta t}, \vartheta_{\Delta t}, b_{\Delta t})\rightharpoonup p_{w}(\varrho, \vartheta, b)\hspace{0,3cm} \mathrm{weakly}\hspace{0.1cm}\mathrm{in}\hspace{0.1cm} L^{1}((0,T)\times\mathsf{B}),\notag\\
&\varrho_{\Delta t}s_{w_{\Delta t}}(\varrho_{\Delta t}, \vartheta_{\Delta t})\to \varrho s_{w}(\varrho,\vartheta)\hspace{0,3cm} \mathrm{a.e.}\hspace{0.1cm}\mathrm{in}\hspace{0.1cm} (0,T)\times\mathsf{B},\notag\\
&\varrho_{\Delta t}s_{w_{\Delta t}}(\varrho_{\Delta t}, \vartheta_{\Delta t})\bu_{\Delta t}\rightharpoonup \varrho s_{w}(\varrho,\vartheta)\bu\hspace{0,3cm} \mathrm{weakly}\hspace{0.1cm}\mathrm{in}\hspace{0.1cm} L^{1}((0,T)\times\mathsf{B}).\notag
\end{align*}

\subsection{Construction of test functions}

We know that the test functions in $\eqref{mom-equ-B}$ and $\eqref{entropy-ine-B}$ depend on the solution $w$. Test functions satisfy 
$\bmphi_{\Delta t}|_{\Gamma_{w_{\Delta t}}(t)}=\psi_{\Delta t}\bfn$ in the $\Delta t$--layer. When $\Delta t \rightarrow 0$, we need to ensure that this relationship is satisfied on the boundary $\Gamma_w(t)$. Indeed,  
one can solve this problem by directly adopting the method from  \cite[Section 6.1.5]{KMN-24}, which is omitted here for brevity.
Roughly speaking, for a fixed $\bmphi\in C^\infty([0,T]\times \mathbb{R}^2)$, we define $\psi_{\Delta t}=\bmphi|_{\Gamma_{w_{\Delta t}}(t)}\cdot \bfn$. Moreover, we know that $\psi_{\Delta t }$ inherits the regularity from $w_{\Delta t}$, and the limit $\psi=\bmphi|_{\Gamma_{w}(t)} \cdot \bfn$.

The above analysis enables us to pass to the limit $\Delta t\to 0$. Therefore, we obtain the weak solution in the fixed domain $\mathsf{B}$. 

\begin{proposition}\label{fixed-domain}
Assume that the extension of the initial data and transport coefficients listed in the Section \ref{Section-extension-parameters} hold. Moreover, the parameter in the artificial pressure $\beta\geq \max\{4, \gamma\}$.  Then there exist $T\in(0,\infty]$ and a weak solution to the extended problem in the sense of Definition $\ref{weak-solutions-on-fixed-domain}$ in the time interval $(0,T)$. The time $T$ is finite only for 
\begin{align*}
\mathrm{either}\hspace{0.3cm}\lim_{s\to T}w(t,y)\searrow \alpha_{\partial\Omega} \hspace{0.2cm} \mathrm{or}\hspace{0.2cm} \lim_{s\to T}w(t,y)\nearrow \beta_{\partial\Omega}, \hspace{0.3cm} \mathrm{for} \hspace{0.1cm}\mathrm{some} \hspace{0.1cm}y\in\Gamma. 
\end{align*}
\end{proposition}

\section{Limits $\omega$, $\zeta$, $\lambda$, $\xi$ $\to$ 0}\label{Section-Limits}

\subsection{Limits of the integrals outside fluid domain}

Before considering the limit process, we list the uniform bounds obtained in Proposition $\ref{fixed-domain}$. 

\begin{lemma}
  The uniform bounds for the solutions $(w, \theta, \varrho, b, \bu, \vartheta)$ hold for any $t\in (0,T)$, as derived in Proposition $\ref{fixed-domain}$.  Indeed, we have 
  \begin{equation}\label{Delta-t-uni-bound-2}
  \begin{aligned}
  &\esssup_{0\leq t\leq T}(\|w_t\|^2_{L^2(\Gamma)}(t)+\|\Delta  w\|^2_{L^2(\Gamma)}(t) + \alpha_2\|\nabla w_t\|^2_{L^2(\Gamma)}(t)+\|\theta\|^2_{L^2(\Gamma)}(t))\leq C,  \\
  &\int_0^T(\alpha_1\|\nabla w_t\|^2_{L^2(\Gamma)}+\|\nabla\theta\|^2_{L^2(\Gamma)})\leq C,  \\
  &\esssup_{0\leq t\leq T}\int_{\mathsf{B}}[\varrho|\bu|^2+|b|^2+{\delta(\varrho+b)^\beta}]\leq C,\\
  &{\esssup_{0\leq t\leq T}\int_{\mathsf{B}} (a^\lambda_w\vartheta^4+\varrho^{\gamma} + \varrho \vartheta)}\leq C\esssup_{0\leq t\leq T}\int_{\mathsf{B}} \varrho e_w(\varrho,\vartheta)\leq C,\\
&\int_0^T\int_{\mathsf{B}}\left(\xi\vartheta^5+\kappa^\zeta_w(|\nabla\log\vartheta|^2+|\nabla\vartheta^{\frac{3}{2}}|^2)\right)\leq C,  \\  
  &  \int_0^T\int_{\mathsf{B}}\frac{1}{\vartheta}\left(\bbS^{\omega}_{w}(\vartheta, \nabla\bu):\nabla\bu+\frac{\kappa_{w}^{\zeta}(\vartheta)|\nabla\vartheta|^2}{\vartheta}\right)\leq C.
  \end{aligned}
  \end{equation} 
  \end{lemma}
  
Note that from $\eqref{Delta-t-uni-bound-2}_6$, one can obtain (using Korn-Poincar\'{e} inequality) 
\begin{align}\label{bound-u-omega}
\|\bu\|^2_{L^2(0,T; W^{1,2}(\mathsf{B}))} \leq \frac{C}{\omega} . 
\end{align}

\vspace*{.2cm}

Also, it follows from $\eqref{Delta-t-uni-bound-2}_6$ that 
\begin{align*}
&\iint_{Q_T^w}\frac{1}{\vartheta}\left(\frac{1}{2}\mu(\vartheta)|\nabla\bu+\nabla^\top \bu -\mathrm{div}\bu\mathbb{I}_2|^2+\eta(\vartheta)|\mathrm{div}\bu|^2\right)\\
    & \leq C  \iint_{Q_T^w}\frac{1}{\vartheta}\bbS(\vartheta,\nabla \bu):\nabla\bu \leq  \int_0^T\int_{\mathsf{B}}\frac{1}{\vartheta}\bbS_w^{\omega}(\vartheta,\nabla \bu):\nabla\bu\leq C. 
\end{align*}
This together with $\eqref{vis-coe}$ yields
\begin{align}\label{gradient-u}
\iint_{Q_T^w}\left(|\nabla\bu+\nabla^\top\bu - \rmdiv \bu \mathbb{I}_2 |^2+|\mathrm{div}\bu|^2\right)\leq C.
\end{align}
Then,
by virtue of the generalized Korn-Poincar\'{e} inequality in Lemma \ref{Korn-inequality}, $\eqref{Delta-t-uni-bound-2}_3$, and $\eqref{gradient-u}$, we derive that 
\begin{align*}
    \|\bu\|^2_{L^2(0,T; W^{1,2}(\Omega_w(t))}\leq C.
\end{align*}
Subsequently, it has
\begin{align*}
&\|\bu\|_{L^2(0,T;L^p(\Omega_w(t)))}\leq C,\hspace{0.3cm}\mathrm{for \hspace{0.2cm}all\hspace{0.2cm}} p\geq 2,\\
&\|\varrho\bu\otimes\bu\|_{L^2(0,T;L^q(\Omega_w(t)))}\leq C, \text{ for } q = \frac{2p\gamma}{(p+2)\gamma + p} .
\end{align*}

Furthermore, it holds that
\begin{align*}
\|\varrho s(\varrho,\vartheta)\|_{L^q(Q_T^w)}+\|\varrho s(\varrho,\vartheta)\bu\|_{L^s(Q_T^w)}+\left\|\frac{\kappa{(\vartheta)}\nabla\vartheta}{\vartheta}\right\|_{L^r(Q_T^w)}\leq C, 
\end{align*}
for some $q, s, r>1$.

Finally, by applying the technique based on ``Bogovski\u{\i}  operator'', one can obtain the improved pressure estimate similar to \eqref{pressure-int}.

\vspace*{.2cm}

Based on the above bounds, we examine the limit as parameters $\xi$, $\omega$, $\zeta$, $\lambda$ $\to$ 0. Our objective is to eliminate all terms originating outside the fluid domain $Q_T^w$. Specifically, we allow these parameters to approach 0 simultaneously, albeit on different scales. Explicitly, one may consider  that 
\begin{align}\label{parameter-ass}
  \xi = \omega^{\frac{1}{2}} = \zeta^{\frac{1}{2}} = \lambda^{\frac{1}{6}} ,
\end{align}
and we prove  the following result.
\begin{lemma}\label{Lemma-vanishing-lower-terms}
Assume that parameters satisfy $\eqref{parameter-ass}$, then as  $\xi\to 0$, we have 
\begin{align*}
&\int_0^T\int_{\mathsf{B}\setminus \Omega_w(t)} a\lambda\vartheta^4\psi_t\to 0,\\
&\int_0^T \int_{\mathsf{B}\setminus \Omega_w(t)}a\lambda\vartheta^3\bu\cdot \nabla\psi\to 0,\\
& \int_0^T \int_{\mathsf{B}\setminus \Omega_w(t)}\bbS^{\omega}_w(\vartheta,\nabla\bu):\nabla\bmphi \to 0, \\
&\int_0^T \int_{\mathsf{B}\setminus \Omega_w(t)}\frac{\kappa^{\zeta}_w(\vartheta)\nabla\vartheta}{\vartheta}\cdot\nabla\psi\to 0, \\
& \int_0^T \int_{\sfB \setminus \Omega_w(t)} \xi \vartheta^4 \psi \to 0 . 
\end{align*}
\end{lemma}

\begin{proof}
According to the uniform bound $\int_0^T\int_{\mathsf{B}}\xi \vartheta^5\leq C$ (see $\eqref{Delta-t-uni-bound-2}_5$), we derive that 
  \begin{align*}
  \int_0^T \int_{\mathsf{B}\setminus \Omega_w(t)}a\lambda\vartheta^4\psi_t&\leq a\lambda\|\vartheta\|^4_{L^5((0,T)\times\mathsf{B})}\|\psi_t\|_{L^{5}((0,T)\times\mathsf{B})} \\
&\leq C a\lambda\xi^{-\frac{4}{5}}
= C a \xi^6 \xi^{-\frac{4}{5}} 
\to 0,\hspace{1cm} \mathrm{as} \hspace{0.2cm} \xi \to 0 .
\end{align*}

Next, by using $\eqref{Delta-t-uni-bound-2}_4$ and $\eqref{bound-u-omega}$, we have  %
\begin{align*}
\int_{0}^T\int_{\mathsf{B}\setminus \Omega_w(t)}a\lambda\vartheta^3\bu\cdot\nabla\psi
&\leq a \lambda \|\nabla \psi\|_{L^\infty((0,T)\times \mathsf{B})} \int_0^T \bigg[\left(\int_{\mathsf{B}\setminus \Omega_w(t)} \vartheta^4\right)^{\frac{3}{4}} \|\bu\|_{L^4(\mathsf{B}\setminus \Omega_w(t))} \bigg]\\
& \leq C a \lambda \|\vartheta^4\|^{\frac{3}{4}}_{L^\infty(0,T; L^1(\mathsf{B}))} \|\bu \|_{L^2(0,T; L^4(\mathsf{B}\setminus \Omega_w(t))}\\
&\leq  \frac{C a \lambda}{\lambda^{\frac{3}{4}} \sqrt{\omega}  } = \frac{Ca\lambda^{\frac{1}{4}}}{\sqrt{\omega}}=C a\xi^{\frac{1}{2}} \to 0,\hspace{1cm} \mathrm{as} \hspace{0.2cm} \xi \to 0.
\end{align*}

After, we find that
\begin{align*} 
&\left|\int_0^T \int_{\mathsf{B} \setminus \Omega_w(t)} \bbS^\omega_w : \nabla \bmphi  \right|     \\
\leq & \|\nabla\bmphi\|_{L^\infty((0,T)\times \sfB)} \left(\int_0^T \int_{\sfB\setminus \Omega_w(t)} \bigg(\frac{1}{\sqrt{\vartheta}} |\bbS^\omega_w|\bigg)^{\frac{10}{9}}\right)^{\frac{9}{10}} \left( \int_0^T \int_{\sfB\setminus \Omega_w(t)} \vartheta^5 \right)^{\frac{1}{10}}     \\
 \leq &\frac{C}{\xi^{\frac{1}{10}}} \left( \int_0^T \int_{\sfB\setminus \Omega_w(t)}\xi \vartheta^5 \right)^{\frac{1}{10}} 
 \left(\int_0^T \int_{\sfB\setminus \Omega_w(t)} \bigg(\frac{1}{\sqrt{\vartheta}} \sqrt{|\bbS^\omega_w: \nabla \bu|}\,\sqrt{|g^\omega_w|(1+\vartheta)}\bigg)^{\frac{10}{9}}\right)^{\frac{9}{10}}  \\
  \leq &\frac{C}{\xi^{\frac{1}{10}}} \left( \int_0^T \int_{\sfB\setminus \Omega_w(t)}\xi \vartheta^5 \right)^{\frac{1}{10}} 
 \left(\int_0^T \int_{\sfB\setminus \Omega_w(t)} \frac{1}{\vartheta} |\bbS^\omega_w : \nabla \bu|  \right)^{\frac{1}{2}} \left( \int_0^T \int_{\sfB \setminus \Omega_w(t)} \big(g^\omega_w (1+\vartheta)\big)^{\frac{5}{4}} \right)^{\frac{2}{5}}   \\
 \leq &\frac{C}{\xi^{\frac{1}{10}}} \left( \int_0^T \int_{\sfB\setminus \Omega_w(t)}\xi \vartheta^5 \right)^{\frac{1}{10}}   \left(\int_0^T \int_{\sfB\setminus \Omega_w(t)} \frac{1}{\vartheta} |\bbS^\omega_w : \nabla \bu|  \right)^{\frac{1}{2}}
  \\
 & \qquad \times \|g^\omega_w\|^{\frac{1}{2}}_{L^{\frac{5}{3}}(((0,T)\times \sfB)\setminus Q^w_T)}\Bigg[ 1+ \frac{1}{\xi^{\frac{1}{10}}} \bigg(\int_0^T \int_{\sfB\setminus \Omega_w(t)} \xi \vartheta^5 \bigg)^{\frac{1}{10}}   \Bigg]  \\
 \leq & \frac{C \sqrt{\omega}}{\xi^{\frac{1}{5}}} 
 = \frac{C \xi}{\xi^{\frac{1}{5}}} \to 0,\hspace{1cm} as \hspace{0.2cm} \xi \to 0,   ,
 \end{align*} 
where we have used the fact \eqref{g-app} and  the uniform bounds 
\begin{align*}
\int_0^T \int_{\sfB} \xi \vartheta^5 \leq C, \ \ \int_0^T \int_{\sfB} \frac{1}{\vartheta} |\bbS^\omega_w: \nabla \bu| \leq C,
\end{align*}
from \eqref{Delta-t-uni-bound-2}.

Thereafter, we calculate the following: 
\begin{align*}
&\int_0^T\int_{\mathsf{B}\setminus \Omega_w(t)}\frac{\kappa^\zeta_w(\vartheta)\nabla\vartheta}{\vartheta}\cdot\nabla\psi\\
&\leq \bigg(\int_0^T\int_{\mathsf{B}\setminus \Omega_w(t)}\frac{\kappa^\zeta_w(\vartheta)|\nabla\vartheta|^2}{\vartheta^2}\bigg)^{\frac{1}{2}} \bigg(\int_0^T\int_{\mathsf{B}\setminus \Omega_w(t)}\kappa^\zeta_w(\vartheta)\bigg)^{\frac{1}{2}}\|\nabla\psi\|_{L^\infty((0,T)\times\mathsf{B})}\\
&\leq C\zeta^{\frac{1}{2}} \left( 1 + \bigg(\int_0^T \int_{\sfB \setminus \Omega_w(t)} \vartheta^3 \bigg)^{\frac{1}{2}} \right)  \\ 
& \leq C\zeta^{\frac{1}{2}}  \left( 1 + \xi^{-\frac{3}{10}}\|\xi \vartheta^5\|^{\frac{3}{10}}_{L^1((0,T)\times \sfB)}  \right) \\
&\leq C \zeta^{\frac{1}{2}}\xi^{-\frac{3}{10}}= C \xi^{1-\frac{3}{10}} \to 0,\hspace{1cm} as \hspace{0.2cm} \xi \to 0  , 
\end{align*}
where we have used \eqref{heat-coe} and 
\eqref{con-app}, and the bound $\eqref{Delta-t-uni-bound-2}_6$.

Finally, we compute that 
\begin{align*}
 \int_0^T \int_{\sfB \setminus \Omega_w(t)} \xi \vartheta^4 \psi \leq \xi^{\frac{1}{5}}\left(\int_0^T \int_{\sfB} \xi \vartheta^5 \right)^{\frac{4}{5}} \|\psi\|_{L^\infty((0,T)\times \sfB)} \leq C \xi^{\frac{1}{5}} \to 0,\hspace{.5cm} as \hspace{0.2cm} \xi \to 0  .
\end{align*}

Thus, the proof is complete. 
\end{proof}

The other terms in $((0,T)\times \sfB)\setminus Q^w_T$ in the weak formulations will vanish in a similar fashion as we have demonstrated in Lemma \ref{Lemma-vanishing-lower-terms}.  Also note that, we do not need to show that the term $\xi\int_0^T\int_{\mathsf{B}}\vartheta^5$ in \eqref{energy-ine} vanishes; instead, we simply utilize the fact that this term has  positive sign. In fact, this term provides control over the temperature in the solid domain $\mathsf{B}\setminus \Omega_w(t)$.

\subsection{Limit system}\label{Section-limit-system-other parameters}
Based on the above analysis, letting parameters $\omega$, $\zeta$, $\lambda$, $\xi$ $\to$ 0, we find that the limiting function $(w, \theta, \varrho, b, \bu, \vartheta)$ satisfies the following items:
\begin{itemize}
\item Function spaces:
\begin{equation}\label{fun-spa-1-delta}
	\left\{\begin{aligned}
&\varrho \geq 0,\hspace{0.2cm} \vartheta >0\hspace{0.3cm} \mathrm{a.e.\hspace{0.1cm} in}\hspace{0.2cm} Q_T^w,\\
&\varrho \in L^\infty(0,T; L^\gamma(\mathbb{R}^2))\cap L^q_{\mathrm{loc}}([0,T]\times \Omega_w(t)) \hspace{0.3cm}{\mathrm{with}} \hspace{0.2cm} \gamma, q>\frac{5}{3}, \hspace{0.2cm}\varrho\in L^\infty(0, T;L^\beta(\mathbb{R}^2)), \\
&b \in L^\infty(0,T; L^2(\mathbb{R}^2)), \hspace{0.1cm} b\in L^\infty(0,T;L^\beta(\mathbb{R}^2)) \\
&\bu \in L^2(0,T;W^{1,2}(\Omega_w(t))), \hspace{0.2cm} \varrho|\bu|^2\in L^\infty(0,T;L^1(\mathbb{R}^2)),\\
&\vartheta \in L^\infty (0,T;L^4(\Omega_w(t))), \hspace{0.3cm}\vartheta, \nabla\vartheta, \log\vartheta, \nabla \log\vartheta \in L^2(Q^w_T),\\
&\varrho s, \varrho s\bu, \frac{\bq}{\vartheta} \in  L^1(Q^w_T);
\end{aligned}
\right.
\end{equation}

\begin{equation}\label{fun-spa-2-delta}
	\left\{\begin{aligned}
& w \in L^\infty (0,T;H^2(\Gamma))\cap W^{1,\infty}(0,T; L^2(\Gamma)), \\
& \sqrt{\alpha_1}w\in H^1(0,T;H^1(\Gamma)),\hspace{0.3cm} \sqrt{\alpha_2} w \in W^{1,\infty}(0,T;H^1(\Gamma)), \\
& \theta >0\hspace{0.1cm} {\mathrm{a.e.}}\hspace{0.1cm}\mathrm{in} \hspace{0.1cm} \Gamma_T, \hspace{0.3cm} \theta \in L^\infty(0,T;L^2(\Gamma))\cap L^2(0,T;H^1(\Gamma)),\\
&\log \theta \in L^2(0,T;H^s(\Gamma)), \hspace{0.1cm} \mathrm{with} \hspace{0.1cm}s\leq \frac{1}{2}.
  \end{aligned}
\right.
\end{equation}

\begin{remark}
The fact $\log \theta \in L^2(0,T; H^s(\Gamma))$ with $0<s\leq \frac12$ can be deduced in the following manner.

Similar to \cite[Lemma 3.10]{MMNRT-22}, one can prove  that 
\begin{align*}
\log \vartheta |_{\Gamma_w(t)} = \log (\vartheta|_{\Gamma_w(t)}) .
\end{align*}
But we have $\log \vartheta \in L^2(0,T; H^1(\Omega_w(t)))$, and so 
$\log \theta = \log \vartheta |_{\Gamma_w(t)} \in L^2(0,T; H^s(\Gamma))$ for any $0<s\leq \frac 12$. In particular, one has $\theta >0$ a.e. in $\Gamma_T$. 

Here, we note that we could reach the regularity $H^{\frac{1}{2}}(\Gamma)$ for $\log \theta$ since the regularity of $w$ (the displacement of the shell) in \eqref{fun-spa-2} allows to ensure that the boundary $\Gamma_w(t)$ of the fluid domain $\Omega_w(t)$ is Lipschitz. 
\end{remark}

\item The continuity on the interface $\Gamma_T$: $w_t \bfn=\bu|_{\Gamma_{w}(t)}$ and $\theta=\vartheta|_{\Gamma_{w}(t)}$.

\item The continuity equation
\begin{align}\label{weak-con-delta}
\int_{Q_T^w}\left(\varrho\partial_t\phi+\varrho\bu\cdot\nabla\phi\right)=-\int_{\Omega_w(0)}\varrho_{0,\delta}\phi(0,\cdot)
\end{align}
holds for all $\phi\in C_c^1([0,T)\times\overline{\Omega_w(t)})$;

\item The no-resistive magnetic equation
\begin{align}\label{mag-equ-delta}
\int_{Q_T^w}\left(b\partial_t\phi+b\bu\cdot\nabla\phi\right)=-\int_{\Omega_w(0)}b_{0,\delta}\phi(0,\cdot)
 \end{align}
 holds for all $\phi\in C_c^1([0,T)\times\overline{\Omega_w(t)})$;

\item The coupled momentum equations
\begin{align}\label{mom-equ-delta}
 &\int_{Q_T^w}\varrho\bu \cdot\partial_t\bmphi+\int_{Q_T^w}\bigg[\varrho\bu\otimes\bu+\bigg(p(\varrho,\vartheta)+{\delta(\varrho+b)^\beta}+\frac{1}{2}b^2\bigg)\mathbb{I}-\bbS(\vartheta,\nabla\bu)\bigg]: \nabla\bmphi\notag \\
 &+\int_{\Gamma_T}(w_t\psi_t-\Delta w\Delta\psi-\alpha_1\nabla w_t\cdot\nabla\psi +\alpha_2\nabla w_t\cdot\nabla\psi_t+\nabla\theta\cdot\nabla\psi)\\
 &=-\int_{\Omega_w(0)}(\varrho\bu)_{0,\delta}\bmphi(0,\cdot)-\int_{\Gamma}v_0\psi(0,\cdot)-\alpha_2\int_{\Gamma}\nabla v_0\cdot\psi(0,\cdot), \notag 
 \end{align}
holds for all $\bmphi\in C^\infty_c([0,T)\times\overline{\Omega_w(t)}) $ and $\psi\in C_c^\infty([0,T)\times \Gamma)$ such that $\bmphi|_{\Gamma_w(t)}=\psi\bfn$ on $\Gamma_T$. 

\item The coupled entropy inequality
\begin{align}\label{entropy-ine-delta}
&\int_{Q_T^w}\left(\varrho s\vphi_t+\varrho s \bu\cdot\nabla\vphi-\frac{\kappa(\vartheta)\nabla\vartheta}{\vartheta}\cdot\nabla\vphi\right)+\langle\mathcal{D}_{\delta};\vphi\rangle_{[\mathcal{M};C]([0,T]\times\Omega_w(t))}\notag \\
&+\int_{\Gamma_T}(\theta\tilde\varphi_t-\nabla\theta\cdot\nabla\tilde\varphi+\nabla w\cdot\nabla\tilde\varphi_t)\\
&\leq -\int_{\Omega_w(0)}\varrho_{0,\delta} s(\varrho_{0,\delta},\vartheta_{0,\delta})\vphi(0,\cdot)-\int_{\Gamma}\theta_0\tilde\varphi(0,\cdot)-\int_{\Gamma}v_0\cdot\nabla\tilde\varphi(0,\cdot), \notag 
\end{align}
holds for all non-negative function $\vphi\in C^\infty_c([0,T)\times\overline{\Omega_w(t)}) $ and $\tilde\varphi \in C_c^\infty([0,T)\times \Gamma)$ such that $\vphi|_{\Gamma_w(t)}=\tilde\varphi$ on $\Gamma_T$, and 
\begin{align}\label{measure-D-delta}
    \mathcal{D}_{\delta} \geq \frac{1}{\vartheta}\left(\bbS(\vartheta,\nabla\bu):\nabla\bu+\frac{\kappa(\vartheta)|\nabla\vartheta|^2}{\vartheta}\right).
\end{align}

\item The coupled energy inequality
\begin{align}\label{energy-ine-delta}
    &\int_{\Omega_w(t)}\left[\frac{1}{2}\varrho|\bu|^2+\frac{1}{2}|b|^2+\varrho e(\varrho, \vartheta)+{\frac{\delta}{\beta-1}(\varrho+b)^{\beta}}\right](t)\notag \\
    &+\int_{\Gamma}\left(\frac{1-\delta}{2}|w_t|^2+\frac{1}{2}|\Delta w|^2+\frac{\alpha_2}{2}|\nabla w_t|^2+\frac{1-\delta}{2}|\theta|^2\right)(t)\notag \\
    &+\int_{[0,t]\times\Gamma}(\alpha_1|\nabla w_t|^2+|\nabla \theta|^2)\\
    &\leq \int_{\Omega_w(0)}\left[\frac{1}{2}\frac{|(\varrho\bu)_{0,\delta}|^2}{\varrho_{0,\delta}}+\frac{1}{2}|b_{0,\delta}|^2+\varrho_{0,\delta}e(\varrho_{0,\delta},\vartheta_{0,\delta})+{\frac{\delta}{\beta-1}(\varrho_{0,\delta}+b_{0,\delta})^{\beta}}\right]\notag \\
    &+\int_{\Gamma}\left(\frac{1}{2}|v_0|^2+\frac{1}{2}|\Delta w_0|^2+\frac{\alpha_2}{2}|\nabla v_0|^2+\frac{1}{2}|\theta_0|^2\right), \notag 
\end{align}
holds for all $t\in [0,T]$. 
\end{itemize}
  
Finally,  we have   the following uniform estimates on $Q_T^w$ (which is concluded from \eqref{energy-ine-delta} and 
\eqref{ini-app}--\eqref{mom-app}--\eqref{ene-entropy-app})
\begin{align}\label{uniform-bound-delta}
&\sup_{0\leq t\leq T}(\|w_t\|^2_{L^2(\Gamma)}(t)+\|\Delta  w\|^2_{L^2(\Gamma)}(t)+\alpha_2\|\nabla w_t\|^2_{L^2(\Gamma)}(t)+\|\theta\|^2_{L^2(\Gamma)}(t))\leq C,\notag \\
&\int_0^t(\alpha_1\|\nabla w_t\|^2_{L^2(\Gamma)}+\|\nabla\theta\|^2_{L^2(\Gamma)})\leq C,\notag \\
&\sup_{0\leq t\leq T}\int_{\Omega_w(t)}[\varrho|\bu|^2+|b|^2+{\delta(\varrho+b)^\beta}]\leq C,\notag \\
&{\sup_{0\leq t\leq T}\int_{\Omega_w(t)} (a\vartheta^4+\varrho^{\gamma}+\varrho \vartheta)\leq C\sup_{0\leq t\leq T}\int_{\Omega_w(t)} \varrho e(\varrho,\vartheta)\leq C,}\\
&\int_0^t\int_{\Omega_w(t)}\left(|\nabla\log\vartheta|^2+|\nabla\vartheta^{\frac{3}{2}}|^2\right)\leq C,\notag \\
&  \int_0^T\int_{\Omega_w(t)}\frac{1}{\vartheta}\left(\bbS(\vartheta, \nabla\bu):\nabla\bu+\frac{\kappa(\vartheta)|\nabla\vartheta|^2}{\vartheta}\right)\leq C, \notag 
\end{align}
where all constants $C$ are independent of parameter $\delta$. 

As consequence of the uniform estimates found in \eqref{uniform-bound-delta}, we can obtain (as obtained in previous sections)
\begin{align}
&\| \bu\|_{L^2(0,T; W^{1,2}(\Omega_w(t))) } \leq C  ,  \label{uni-bound-delta-u}\\
&\|\varrho \bu\|_{L^\infty(0,T; L^{\frac{2\gamma}{\gamma+1}}(\Omega_w(t))) } \leq C ,\label{uni-bound-delta-rho-u} \\
& \|\varrho \bu \otimes \bu \|_{L^2(0,T; L^q(\Omega_w(t))     )  } \leq C, \ \ \text{for } q=\frac{2p\gamma}{(p+2)\gamma + p}  \label{uni-bound-delta-rho-u-u}.
\end{align}

\section{Limit $\delta\to 0$}\label{Section-limit-delta} 
Using the uniform bounds established in $\eqref{uniform-bound-delta}$, we obtain that 
\begin{align}\label{delta-bound-uniform}
&w_{\delta} \rightharpoonup w \hspace{0.5cm}\mathrm{weakly *}\hspace{0.2cm}\mathrm{in}\hspace{0.2cm} L^\infty(0,T; H^2(\Gamma)),\notag\\
&(w_{\delta})_t\rightharpoonup w_t \hspace{0.5cm}\mathrm{weakly *}\hspace{0.2cm}\mathrm{in}\hspace{0.2cm} L^\infty(0,T; L^2(\Gamma)),\notag\\
&\theta_{\delta}\rightharpoonup \theta  \hspace{0.5cm}\mathrm{weakly *}\hspace{0.2cm}\mathrm{in}\hspace{0.2cm} L^\infty(0,T; L^2(\Gamma)),\notag\\
&\varrho_{\delta } \rightharpoonup  \varrho \hspace{0.5cm} \mathrm{weakly *} \hspace{0.2cm}\mathrm{in} \hspace{0.2cm} L^\infty(0,T;L^{\gamma}(\Omega_w(t))),\\
&b_{\delta} \rightharpoonup  b \hspace{0.5cm} \mathrm{weakly *} \hspace{0.2cm}\mathrm{in} \hspace{0.2cm} L^\infty(0,T;L^2(\Omega_w(t)),\notag\\
&\bu_{\delta} \rightharpoonup  \bu \hspace{0.5cm} \mathrm{weakly *} \hspace{0.2cm}\mathrm{in} \hspace{0.2cm} L^2(0,T;W^{1,2}(\Omega_w(t)),\notag\\
&\vartheta_{\delta} \rightharpoonup  \vartheta \hspace{0.5cm} \mathrm{weakly *} \hspace{0.2cm}\mathrm{in} \hspace{0.2cm} L^\infty(0,T;L^{4}(\Omega_w(t))),\notag\\
&\vartheta_{\delta} \rightharpoonup  \vartheta \hspace{0.5cm} \mathrm{weakly} \hspace{0.2cm}\mathrm{in} \hspace{0.2cm} L^2(0,T;H^1(\Omega_w(t))). \notag
\end{align}
Moreover, we mimic the estimate in $\eqref{embedding-Delta-t}$ to get 
\begin{align}\label{delta-w-bound}
w_{\delta}\rightarrow w \hspace{0.5cm}\mathrm{in} \hspace{0.2cm} C^{\frac{1}{5}}([0,T];C^{1,\frac{1}{10}}(\Gamma)).
\end{align}

\subsection{Continuity of velocity and temperature on the interface}\label{Sec-con-interface}


Based on the assumption on the flow function in $\eqref{flow-fun}$, it reveals that $\bu_{\delta}\circ \widetilde{\bmphi_{w_{\delta}}}$ is bounded in $L^2(0,T;W^{1,2}(\mathbb{R}^2))$. Then we have 
\begin{align}\label{weak-u-varphi-delta}
\bu_{\delta}\circ \widetilde{\bmphi_{w_{\delta}}} \rightharpoonup \mathbf{h} \hspace{0.5cm} \mathrm{weakly} \hspace{0.2cm}\mathrm{in} \hspace{0.2cm} L^2(0,T;W^{1,2}(\mathbb{R}^2)).
\end{align}
In light of $\eqref{delta-bound-uniform}_2$, we deduce that $w_t\bfn$ aligns with the trace of $\mathbf{h} $ on $\Gamma$.  Therefore, our remaining task is to demonstrate that $\mathbf{h}=\bu\circ \widetilde{\bmphi_{w}} $. 
For any test function $\varphi\in C_c^\infty((0,T)\times\mathbb{R}^2)$, it has 
\begin{align*}
&\int_{(0,T)\times\mathbb{R}^2}(\bu_{\delta}\circ \widetilde{\bmphi_{w_{\delta}}} -\bu\circ \widetilde{\bmphi_{w}})\cdot \varphi\\
&=\int_{(0,T)\times\mathbb{R}^2}(\bu_{\delta}\circ \widetilde{\bmphi_{w_{\delta}}} -\bu\circ \widetilde{\bmphi_{w_{\delta}}})\cdot \varphi+\int_{(0,T)\times\mathbb{R}^2}(\bu\circ \widetilde{\bmphi_{w_{\delta}}} -\bu \circ \widetilde{\bmphi_{w}})\cdot \varphi .
\end{align*}
Based on $\eqref{delta-w-bound}$, we establish that $(\widetilde{\bmphi_{w_{\delta}}})^{-1}$ converges to $(\widetilde{\bmphi_{w}})^{-1}$ locally uniformly in $[0,T]\times \mathbb{R}^2$, and $\widetilde{\bmphi_{w_{\delta}}}$ converges to $\widetilde{\bmphi_{w}}$ locally uniformly in $[0,T]\times\mathbb{R}^2$. Consequently, $\varphi\circ (\widetilde{\bmphi_{w_{\delta}}})^{-1}$ converges $\varphi\circ (\widetilde{\bmphi_{w}})^{-1}$ uniformly in $[0,T]\times\mathbb{R}^2$, and $\nabla (\widetilde{\bmphi_{w_{\delta}}})^{-1}$ converges $\nabla \widetilde{\bmphi_{w_{\delta}}}$ in $L^p((0,T)\times\mathbb{R}^2)$ for any $p\geq 1$. 
Based on this analysis, it gives 
\begin{align*}
&\lim_{\delta\rightarrow 0}\int_{(0,T)\times\mathbb{R}^2}(\bu_{\delta}\circ \widetilde{\bmphi_{w_{\delta}}} -\bu\circ \widetilde{\bmphi_{w_{\delta}}})\cdot \varphi\\
&=\lim_{\delta\rightarrow 0}\int_{(0,T)\times\mathbb{R}^2}(\bu_{\delta}-\bu)\cdot \varphi\circ (\widetilde{\bmphi_{w_{\delta}}})^{-1}|\mathrm{det}(\nabla \widetilde{\bmphi_{w_{\delta}}})^{-1}|=0.
\end{align*}

Moreover, for any $n\in (1,2)$, it has 
\begin{align}\label{II-u-varphi}
&\int_{(0,T)\times\mathbb{R}^2}(\bu\circ \widetilde{\bmphi_{w_{\delta}}} -\bu \circ \widetilde{\bmphi_{w}})\cdot \varphi\\
&\leq |\{\widetilde{\bmphi_{w_{\delta}}} \neq \widetilde{\bmphi_{w}}\}|^{\frac{1}{n'}}\left(\|\bu\circ \widetilde{\bmphi_{w_{\delta}}}\|_{L^n(0,T;L^n(\mathbb{R}^2))}+\|\bu \circ \widetilde{\bmphi_{w}}\|_{L^n(0,T;L^n(\mathbb{R}^2))}\right)\|\varphi\|_{L^\infty((0,T)\times\mathbb{R}^2)}\notag.
\end{align}
Now, observe that 
\begin{align*}
\|\bu \circ \widetilde{\bmphi_{w_{\delta}}}\|_{L^n(0,T;L^n(\mathbb{R}^2))}&=\bigg(\int_{(0,T)\times\mathbb{R}^2}|\bu|^n|\mathrm{det}\nabla(\widetilde{\bmphi_{w_{\delta}}})^{-1}|\bigg)^{\frac{1}{n}}\\
&\leq \|\bu\|_{L^2(0,T;L^2(\mathbb{R}^2))}\|\mathrm{det}\nabla(\widetilde{\bmphi_{w_{\delta}}})^{-1}\|^{\frac{1}{n}}_{L^{\frac{2}{2-n}}((0,T)\times\mathbb{R}^2)}\\
&\leq C .
\end{align*}
Based on the fact that $|\{\widetilde{\bmphi_{w_{\delta}}} \neq \widetilde{\bmphi_{w}}\}|\rightarrow$ as $\delta \rightarrow 0$, letting $\delta \rightarrow 0$ in $\eqref{II-u-varphi}$, we get 
\begin{align*}
\lim_{\delta \rightarrow 0}\int_{(0,T)\times\mathbb{R}^2}(\bu\circ \widetilde{\bmphi_{w_{\delta}}} -\bu \circ \widetilde{\bmphi_{w}})\cdot \varphi=0.
\end{align*}

This completes the proof of $\mathbf{h}=\bu\circ \widetilde{\bmphi_{w}}$ in $\eqref{weak-u-varphi-delta}$.

\vspace*{.5cm} 

 Next, we prove the continuity of temperature on the interface $\Gamma_w$. 
Utilizing the property of the flow function in $\eqref{flow-fun}$, we deduce that $\vartheta_{\delta}\circ \widetilde{\bmphi_{w_{\delta}}}$ is bounded in $L^2(0,T;W^{1,2}(\mathbb{R}^2))$. Therefore, one has 
\begin{align}\label{weak-vartheta-varphi-delta}
\vartheta_{\delta}\circ \widetilde{\bmphi_{w_{\delta}}} \rightharpoonup  {g} \hspace{0.5cm} \mathrm{weakly} \hspace{0.2cm}\mathrm{in} \hspace{0.2cm} L^2(0,T;W^{1,2}(\mathbb{R}^2)).
\end{align}

By virtue of $\eqref{delta-bound-uniform}_3$, we infer that $\theta$ coincides with the trace of ${g} $ on $\Gamma$. It requires us to show that $g=\vartheta\circ \widetilde{\bmphi_{w}}$. 
Indeed, 
for any test function $\varphi\in C_c^\infty((0,T)\times\mathbb{R}^2)$, we have
\begin{align*}
&\int_{(0,T)\times\mathbb{R}^2}(\vartheta_{\delta}\circ \widetilde{\bmphi_{w_{\delta}}} -\vartheta\circ \widetilde{\bmphi_{w}})\cdot \varphi\\
&=\int_{(0,T)\times\mathbb{R}^2}(\vartheta_{\delta}\circ \widetilde{\bmphi_{w_{\delta}}} -\vartheta\circ \widetilde{\bmphi_{w_{\delta}}})\cdot \varphi+\int_{(0,T)\times\mathbb{R}^2}(\vartheta\circ \widetilde{\bmphi_{w_{\delta}}} -\vartheta \circ \widetilde{\bmphi_{w}})\cdot \varphi .
\end{align*}
Then, 
inspired by the previous analysis, we can also prove that the above integration tends to $0$ as $\delta \rightarrow 0$. Thus  the proof is complete.

\subsection{Higher integrability of density and magnetic field}

 Following a similar argument as in \cite[Section 2.2.5]{FN-09} and \cite[Lemma 4.1]{FNP-01}, we obtain the higher integrability of $\varrho^\gamma+b^2$ uniformly in $\delta$.

\begin{lemma}\label{high-int-press}

Assume that $(\varrho_\delta, b_{\delta}, \bu_{\delta}, \vartheta_\delta, w_\delta)$ is the solution discussed in Section \ref{Section-limit-system-other parameters}. Then it holds that 
\begin{align}\label{high-int-varrho-b-delta}
\int_{0}^T\int_{\Omega_w(t)}(\varrho_\delta^{\gamma+\theta_1}+b_{\delta}^{2+\theta_2}+\delta\varrho_{\delta}^{\beta+\theta_1}+\delta b_{\delta}^{\beta+\theta_2})\leq C, 
\end{align}
for some positive constants $\theta_1$ and $\theta_2$ satisfying 
\begin{align*}
\theta_1 \leq \min\left\{\frac{\gamma-1}{2} - \frac{\gamma}{p}, \frac{\gamma}{4} \right\} \hspace{0.2cm} \text{and} \hspace{0.2cm} \theta_2 \leq 
\min\left\{ \frac{\gamma-1}{\gamma} - \frac{2}{p}, \frac{1}{2} \right\} ,     \hspace{0.2cm} \text{for} \ p>\frac{2\gamma}{\gamma-1} \text{ and }\gamma>1.
\end{align*}

\end{lemma}

\begin{proof}
We apply the technique based on the ``Bogovski\u{\i} operator'' to prove the required estimate. 
\begin{itemize} 
\item[--] To this end, we consider the auxiliary problem: let $D\subset \mathbb R^2$ be any smooth domain, then for given any 
\begin{align}\label{B-1} 
g \in  L^q(D) , \ \ \int_D g \, d x = 0 ,
\end{align}
find a vector field $\mathbf v= \mathcal B [g]$ such that 
\begin{align}\label{B-2}
	\mathbf v \in W^{1,q}_0(D; \mathbb R^2) , \ \  \ \rmdiv \, \mathbf v = g \ \ \text{in } D.
\end{align}

\item[--] The problem \eqref{B-1}--\eqref{B-2} admits many solutions, but we use the construction due to Bogovskii \cite{Bo-80}. In particular, we write the following results from \cite[Galdi, Chapter III.3]{Galdi}.

\item[--] The operator $\mathcal B$ is bounded, linear and enjoys the bound 
\begin{align*} 
\|\mathcal B [g] \|_{W^{1,q}_0(D; \R^2)} \leq c(q) \|g\|_{L^q(D)} \ \text{ for any } 1<q<\infty.
\end{align*} 
If, moreover, $g$ can be written in the form $g=\rmdiv \mathbf G$ for a certain $\mathbf G\in L^r(D; \mathbb R^2)$, $\mathbf G \cdot \mathbf n |_{\partial D}= 0$, then 
\begin{align*} 
\|\mathcal B [g] \|_{L^{r}(D; \mathbb R^2)} \leq c(r) \|\mathbf G\|_{L^{r}(D;\mathbb R^2)} \ \text{ for any } 1<r<\infty   .
\end{align*}
\end{itemize}

$\bullet$ Now, we choose the  test function $(\bmphi, \psi)=\bigg(\tilde \psi \mathcal{B}\left(\varrho^{\theta_1}_\delta-[\varrho^{\theta_1}_{\delta}]_{\Omega_w(t)}\right), 0\bigg)$  
  with \\ $[\varrho_{\delta}^{\theta_1}]_{\Omega_w(t)}=\frac{1}{|\Omega_w(t)|}\int_{\Omega_w(t)}\varrho^{\theta_1}_{\delta}$  for each $t\in [0,T]$, and 
 $\tilde \psi \in C_c^1(0,T)$ 
in $\eqref{mom-equ-delta}$. 
As a consequence, we have 
\begin{align*}
&\int_0^T \tilde \psi \int_{\Omega_w(t)}\bigg(p(\varrho_\delta, \vartheta_\delta)+\delta(\varrho_{\delta}+b_{\delta})^\beta+\frac{1}{2}b^2_\delta\bigg)\varrho_\delta^{\theta_1}\\
=&\frac{1}{|\Omega_w(t)|}\int_0^T\tilde \psi\int_{\Omega_w(t)}[\varrho^{\theta_1}_{\delta}]_{\Omega_w(t)}
\bigg(p(\varrho_\delta, \vartheta_\delta)+\delta(\varrho_{\delta}+b_{\delta})^\beta+\frac{1}{2}b^2_\delta\bigg)
\\
& -\int_0^T \tilde \psi \int_{\Omega_w(t)}\varrho_{\delta}\bu_{\delta}\cdot\partial_t\mathcal{B}\left(\varrho_\delta^{\theta_1}-[\varrho_{\delta}^{\theta_1}]_{\Omega_w(t)}\right) \\ 
& - \int_0^T \partial_t \tilde{\psi} \int_{\Omega_w(t)} \varrho_{\delta}\bu_{\delta}\cdot\mathcal{B}\left(\varrho_\delta^{\theta_1}-[\varrho_{\delta}^{\theta_1}]_{\Omega_w(t)})\right)\\
&  -\int_0^T \tilde{\psi} \int_{\Omega_w(t)}\varrho_{\delta}\bu_{\delta}\otimes\bu_{\delta}:\nabla\mathcal{B}\left(\varrho_\delta^{\theta_1} -[\varrho_{\delta}^{\theta_1}]_{\Omega_w(t)}\right)\\
&+\int_0^T \tilde \psi \int_{\Omega_w(t)}\bbS(\vartheta_{\delta},\nabla\bu_{\delta}):\nabla \mathcal{B}\left(\varrho_\delta^{\theta_1} -[\varrho_{\delta}^{\theta_1}]_{\Omega_w(t)}\right) \\
&:=\sum_{i=1}^5\mathcal{J}_{i}.
\end{align*}

(i) Provided that $\theta_1\leq \gamma$, we use H\"older's inequality, $\eqref{uniform-bound-delta}_3$, and $\eqref{uniform-bound-delta}_4$ to get 
\begin{align*}
|\mathcal{J}_1|&\leq C\|\tilde{\psi}\|_{L^\infty(0,T)}\|\varrho_\delta^{\theta_1}\|_{L^1(0,T; L^1(\Omega_w(t)))}\|p(\varrho_\delta,\vartheta_\delta)+\delta(\varrho_{\delta}+b_{\delta})^\beta+b^2_\delta\|_{L^\infty(0,T;L^1(\Omega_w(t)))}\\
&\leq C \|\tilde{\psi}\|_{L^\infty(0,T)}.
\end{align*}

\vspace*{.2cm}

(ii) Regarding the term $\mathcal{J}_2$, we  employ the renormalized continuity equation
\begin{align*}
a(\varrho_\delta)_t + \rmdiv (a(\varrho_\delta) \bu_\delta) +  \left(a^\prime(\varrho_\delta) \varrho_\delta - a(\varrho) \right) \rmdiv \bu_\delta = 0 , 
\end{align*}
with $a(\rho_\delta)=\rho^{\theta_1}_\delta$.
Consequently, we have 

\begin{align*}
|\mathcal J_2 | &= \left|    \int_0^T \tilde \psi \int_{\Omega_w(t)}  \varrho_\delta \bu_\delta  \cdot \partial_t \mathcal B \left(\varrho_\delta^{\theta_1} - [\varrho_\delta^{\theta_1}]_{\Omega_w(t)}  \right)     \right| \\
& \leq \|\tilde \psi\|_{L^\infty(0,T)} \int_{0}^T \|\varrho_\delta \bu_\delta\|_{L^{\frac{2\gamma}{\gamma+1}}(\Omega_w(t)} \left\|  \partial_t \mathcal B \left(\varrho_\delta^{\theta_1} - [\varrho_\delta^{\theta_1}]_{\Omega_w(t)}  \right) \right\|_{L^{\frac{2\gamma}{\gamma-1}}(\Omega_w(t))} .
\end{align*}
We have (using properties of Bogovski\u{\i}'s operator),
\begin{align*}
&\left\|\partial_t\mathcal{B}\left(\varrho_\delta^{\theta_1}-[\varrho_{\delta}^{\theta_1}]_{\Omega_w(t)}\right)\right\|_{L^{\frac{2\gamma}{\gamma-1}}(\Omega_w(t))} \\
& \leq \left\| \mathcal B \left( \rmdiv (\varrho^{\theta_1}_\delta \bu_\delta) \right)  \right\|_{L^{\frac{2\gamma}{\gamma-1}}(\Omega_w(t))} + \bigg\|\mathcal B \Big[(\theta_1-1)\varrho^{\theta_1}_\delta \rmdiv \bu_\delta -\frac{1}{|\Omega_w(t)|} \int_{\Omega_w(t)} (\theta_1-1)\varrho^{\theta_1}_\delta \rmdiv \bu_\delta  \Big] \bigg\|_{L^{\frac{2\gamma}{\gamma-1}}(\Omega_w(t)) }  \\
& \leq C\|\varrho^{\theta_1}\bu_{\delta}\|_{L^{\frac{2\gamma}{\gamma-1}}(\Omega_w(t))} +
\bigg\|\mathcal B \Big[(\theta_1-1)\varrho^{\theta_1}_\delta \rmdiv \bu_\delta -\frac{1}{|\Omega_w(t)|} \int_{\Omega_w(t)} (\theta_1-1)\varrho^{\theta_1}_\delta \rmdiv \bu_\delta  \Big] \bigg\|_{W^{1,\frac{2\gamma}{2\gamma-1}}(\Omega_w(t))} \\ 
&\leq C\|\varrho^{\theta_1}\bu_{\delta}\|_{L^{\frac{2\gamma}{\gamma-1}}(\Omega_w(t) )} 
+ C \|\varrho^{\theta_1}_\delta \rmdiv \bu_\delta \|_{L^{\frac{2\gamma}{2\gamma-1}}(\Omega_w(t)) }
\\
&\leq C\|\varrho^{\theta_1}_{\delta}\|_{L^{\frac{2\gamma p}{p(\gamma-1)-2\gamma}}(\Omega_w(t) )}\|\bu_{\delta}\|_{L^p(\Omega_w(t))}+C\|\varrho^{\theta_1}_{\delta}\|_{L^{\frac{2\gamma}{\gamma-1}}(\Omega_w(t) )}\|\mathrm{div}\bu_{\delta}\|_{L^2(\Omega_w(t)) } , 
\end{align*}
provided  $p > \frac{2\gamma}{\gamma-1}$, where we have used the embedding $W^{1,\frac{2\gamma}{2\gamma-1}}(\Omega_w(t)) \hookrightarrow L^{\frac{2\gamma}{\gamma-1}}(\Omega_w(t))$.

Then, using H\"older's inequality, \eqref{uniform-bound-delta}$_4$, \eqref{uni-bound-delta-u} and \eqref{uni-bound-delta-rho-u}, we have 
\begin{align*}
|\mathcal{J}_2|\leq & C\|\tilde{\psi}\|_{L^\infty(0,T)}\|\varrho_\delta\bu_\delta\|_{L^\infty(0,T;L^{\frac{2\gamma}{\gamma+1}}(\Omega_w(t))}
\bigg(\|\varrho^{\theta_1}_{\delta}\|_{L^\infty(0,T;L^{\frac{2\gamma p}{p(\gamma-1)-2\gamma}}(\Omega_w(t) )) }\|\bu_{\delta}\|_{L^2(0,T; L^p(\Omega_w(t)))} \\
&+\|\varrho^{\theta_1}_{\delta}\|_{L^\infty(0,T;L^{\frac{2\gamma}{\gamma-1}}(\Omega_w(t) ))}\|\bu_{\delta} \|_{L^2(0,T; W^{1,2}(\Omega_w(t)))}\bigg) \leq C \|\tilde{\psi}\|_{L^\infty(0,T)} , 
\end{align*}
provided that 
\begin{align*}
&\theta_1\frac{2\gamma p}{p(\gamma-1)-2\gamma}\leq \gamma,  \text{ and }  \theta_1\frac{2\gamma}{\gamma-1}\leq \gamma,\\
 \text{i.e., } & \theta_1\leq \frac{\gamma-1}{2}-\frac{\gamma}{p},   \text{ and } \theta_1\leq \frac{\gamma-1}{2},
\end{align*}
thus $\mathcal{J}_2$ is bounded uniformly in $\delta$ if $\theta_1<\frac{\gamma-1}{2}$.

\vspace*{.2cm}

(iii) Let us focus on the term $\mathcal J_3$. Using H\"older's inequality we have   %
\begin{align*}
|\mathcal J_3| & =\left|\int_0^T \partial_t \tilde \psi \int_{\Omega_w(t)} \varrho_\delta \bu_{\delta}     \cdot \mathcal B \left(\varrho^{\theta_1}_{\delta} - [\varrho^{\theta_1}_\delta]_{\Omega_w(t)}  \right)     \right| \\
 & \leq \| \partial_t \tilde \psi\|_{L^\infty(0,T)} \int_0^T \| \varrho_\delta \bu_\delta \|_{L^{\frac{2\gamma}{\gamma+1}}(\Omega_w(t))} \| \mathcal B \left(\varrho^{\theta_1}_{\delta} - [\varrho^{\theta_1}_\delta]_{\Omega_w(t)}  \right) \|_{L^{\frac{2\gamma}{\gamma-1}}(\Omega_w(t)) } \\
 & \leq C \| \varrho_\delta \bu_\delta \|_{L^\infty(0,T; L^{\frac{2\gamma}{\gamma+1}}(\Omega_w(t)))} \int_0^T \| \mathcal B \left(\varrho^{\theta_1}_{\delta} - [\varrho^{\theta_1}_\delta]_{\Omega_w(t)}  \right) \|_{W^{1,\frac{2\gamma}{2\gamma-1}}(\Omega_w(t)) } \\
 & \leq C \| \varrho_\delta \bu_\delta \|_{L^\infty(0,T; L^{\frac{2\gamma}{\gamma+1}}(\Omega_w(t)))} \|\varrho^{\theta_1}_\delta \|_{L^\infty(0,T; L^{\frac{2\gamma}{2\gamma-1}}(\Omega_w(t) )} \leq C,
\end{align*}
where we use the bound \eqref{uni-bound-delta-rho-u} and \eqref{uniform-bound-delta}$_4$ on condition that 
$\theta_1\leq \gamma-\frac{1}{2}$.


\vspace*{.2cm}
(iv) Under condition that $p>\frac{2\gamma}{\gamma-1}$, and $\eqref{uniform-bound-delta}_4$, $\eqref{uni-bound-delta-u}$, and the properties  of  Bogovski\u{\i}'s operator,  the term $\mathcal J_4$ can be estimated as 
\begin{align*}
|\mathcal J_4| &= \left| \int_0^T \tilde \psi \int_{\Omega_w(t)} \varrho_\delta \bu_\delta \otimes \bu_\delta : \nabla \mathcal B \left( \varrho_\delta^{\theta_1} -[\varrho^{\theta_1}_\delta]_{\Omega_w(t)}  \right)   \right| \\
& \leq \int_0^T |\tilde \psi| \|\varrho_\delta\|_{L^\gamma(\Omega_w(t))} \|\bu_\delta\|^2_{L^p(\Omega_w(t))} \left\|\nabla \mathcal B \left( \varrho_\delta^{\theta_1} -[\varrho^{\theta_1}_\delta]_{\Omega_w(t)}  \right) \right\|_{L^{ \frac{p\gamma}{p(\gamma-1)-2\gamma}}(\Omega_w(t)) }\\
& \leq \|\tilde \psi\|_{L^\infty(0,T)} \|\varrho_\delta\|_{L^\infty(0,T; L^{\gamma}(\Omega_w(t)))} \|\bu_\delta\|^2_{L^2(0,T;  L^p(\Omega_w(t)))} \|\varrho^{\theta_1}_{\delta}\|_{L^\infty(0,T;  L^{ \frac{p\gamma}{p(\gamma-1)-2\gamma}}(\Omega_w(t)))    } \\
& \leq C, 
\end{align*}
provided we consider 
\begin{align*}
\theta_1 \frac{p\gamma}{p(\gamma-1)-2\gamma} \leq \gamma 
\ \text{ i.e., } \ \theta_1\leq \gamma-1-\frac{2\gamma}{p}. 
\end{align*}

\vspace*{.2cm}
(v) Finally, assuming  $\theta_1\leq \frac{\gamma}{4}$, using H\"older's inequality and $\eqref{uniform-bound-delta}_4$,  we get 
\begin{align*}
|\mathcal{J}_5|  &= \left|\int_0^T \tilde \psi \int_{\Omega_w(t)} \bbS(\vartheta_\delta, \nabla \bu_\delta): \nabla \mathcal B \left(\varrho^{\theta_1}_\delta - [\varrho^{\theta_1}_\delta]_{\Omega_w(t)}  \right)   \right| \\
&\leq   C\|\tilde{\psi}\|_{L^\infty(0,T)} 
\int_0^T \int_{\Omega_w(t)} \|\bbS(\vartheta_\delta, \nabla \bu_\delta)\|_{L^{\frac{4}{3}}(\Omega_w(t))} \|\varrho_\delta^{\theta_1}\|_{L^4(\Omega_w(t))} \\
& \leq C\|\tilde{\psi}\|_{L^\infty(0,T)} \|\varrho_\delta^{\theta_1}\|_{L^4(\Omega_w(t))} \int_0^T \| \bbS(\vartheta_\delta, \nabla \bu_\delta)\|_{L^{\frac{4}{3}}(\Omega_w(t))} \\ 
&\leq C,
\end{align*}
where we have used the following fact:
\begin{align*}
\int_0^T \| \bbS(\vartheta_\delta, \nabla \bu_\delta)\|_{L^{\frac{4}{3}}(\Omega_w(t))} 
&\leq C \int_0^T \left( \int_{\Omega_w(t)} \bigg| \frac{1}{\sqrt{\vartheta_\delta}} \sqrt{\bbS(\vartheta_\delta, \nabla \bu_\delta) : \nabla \bu_\delta} \bigg|^{\frac{4}{3}} |\vartheta_\delta|^{\frac{4}{3}}   
\right)^{\frac{3}{4}}\\
& \leq C \int_0^T \left(\int_{\Omega_{w}(t)} \frac{1}{\vartheta_\delta} \left|\bbS(\vartheta_\delta, \nabla \bu_\delta): \nabla \bu_\delta \right| \right)^{\frac{1}{2}} \left(\int_{\Omega_w(t)} \vartheta^4_\delta\right)^{\frac{1}{4}} \\
& \leq C \left( \int_0^T \int_{\Omega_w(t)} \frac{1}{\vartheta_\delta} \left|\bbS(\vartheta_\delta, \nabla \bu_\delta): \nabla \bu_\delta\right| \right)^{\frac{1}{2}} \|\vartheta_\delta\|_{L^\infty(0,T; L^4(\Omega_w(t))) } \\
& \leq C ,
\end{align*}
thanks to $\eqref{uniform-bound-delta}_4$ and $\eqref{uniform-bound-delta}_6$.

Combining the above estimates, and choosing $\theta_1$ such that $\theta_1\leq \frac{\gamma-1}{2}-\frac{\gamma}{p}$ and $\theta_1\leq \frac{\gamma}{4}$, we obtain 
\begin{align}\label{extra-density}
\int_{0}^T\int_{\Omega_w(t)}\left( p(\varrho_\delta, \vartheta_\delta) +\delta(\varrho_{\delta}+ b_{\delta})^{\beta} + \frac{1}{2}b^2_{\delta} \right) \varrho^{\theta_1}_{\delta} \leq C   .
\end{align}






\vspace*{.2cm}

$\bullet$ Next, we choose the test function
$(\bmphi, \psi)=\bigg(\tilde \psi \mathcal{B}\left(b^{\theta_2}_\delta-[b^{\theta_2}_{\delta}]_{\Omega_w(t)}\right), 0\bigg]$  
  with \\ $[b_{\delta}^{\theta_2}]_{\Omega_w(t)}=\frac{1}{|\Omega_w(t)|}\int_{\Omega_w(t)}b^{\theta_2}_{\delta}$  for each $t\in [0,T]$, and 
 $\tilde \psi \in C_c^1(0,T)$ 
in $\eqref{mom-equ-delta}$. 
As a consequence, we have 
\begin{align*}
&\int_0^T \tilde \psi \int_{\Omega_w(t)}\bigg(p(\varrho_\delta, \vartheta_\delta)+\delta(\varrho_{\delta}+b_{\delta})^\beta+\frac{1}{2}b^2_\delta\bigg)b_\delta^{\theta_2}\\
=&\frac{1}{|\Omega_w(t)|}\int_0^T\tilde \psi\int_{\Omega_w(t)}[b^{\theta_2}_{\delta}]_{\Omega_w(t)}
\bigg(p(\varrho_\delta, \vartheta_\delta)+\delta(\varrho_{\delta}+b_{\delta})^\beta+\frac{1}{2}b^2_\delta\bigg) 
\\
& -\int_0^T \tilde \psi \int_{\Omega_w(t)}\varrho_{\delta}\bu_{\delta}\cdot\partial_t\mathcal{B}\left(b_\delta^{\theta_2}-[b_{\delta}^{\theta_2}]_{\Omega_w(t)}\right) \\ 
& - \int_0^T \partial_t \tilde{\psi} \int_{\Omega_w(t)} \varrho_{\delta}\bu_{\delta}\cdot\mathcal{B}\left(b_\delta^{\theta_2}-[b_{\delta}^{\theta_2}]_{\Omega_w(t)})\right)\\
&  -\int_0^T \tilde{\psi} \int_{\Omega_w(t)}\varrho_{\delta}\bu_{\delta}\otimes\bu_{\delta}:\nabla\mathcal{B}\left(b_\delta^{\theta_2} -[b_{\delta}^{\theta_2}]_{\Omega_w(t)}\right)\\
&+\int_0^T \tilde \psi \int_{\Omega_w(t)}\bbS(\vartheta_{\delta},\nabla\bu_{\delta}):\nabla \mathcal{B}\left(b_\delta^{\theta_2} -[b_{\delta}^{\theta_2}]_{\Omega_w(t)}\right) \\
&:=\sum_{i=1}^5\mathcal{I}_{i}.
\end{align*}

(i) Provided that $\theta_2\leq 2$, we use H\"older's inequality, $\eqref{uniform-bound-delta}_3$, and $\eqref{uniform-bound-delta}_4$ to get 
\begin{align*}
|\mathcal{I}_1|&\leq C\|\tilde{\psi}\|_{L^\infty(0,T)}\|b_\delta^{\theta_2}\|_{L^1(0,T; L^1(\Omega_w(t)))}\|p(\varrho_\delta,\vartheta_\delta)+\delta(\varrho_{\delta}+b_{\delta})^\beta+b^2_\delta\|_{L^\infty(0,T;L^1(\Omega_w(t)))}\\
&\leq C \|\tilde{\psi}\|_{L^\infty(0,T)}.
\end{align*}

\vspace*{.2cm}

(ii) Regarding the term $\mathcal{I}_2$, we  employ the renormalized continuity equation
\begin{align*}
a(b_\delta)_t + \rmdiv ( a(b_\delta) \bu_\delta) +  \left(a^\prime(b_\delta) b_\delta - a(b_\delta) \right) \rmdiv \bu_\delta = 0 , 
\end{align*}
with $a(b_\delta)=b^{\theta_2}_\delta$.
Consequently, we have 
\begin{align*}
|\mathcal I_2 | &= \left|    \int_0^T \tilde \psi \int_{\Omega_w(t)}  \varrho_\delta \bu_\delta  \cdot \partial_t \mathcal B \left(b_\delta^{\theta_2} - [b_\delta^{\theta_2}]_{\Omega_w(t)}  \right)     \right| \\
& \leq \|\tilde \psi\|_{L^\infty(0,T)} \int_{0}^T \|\varrho_\delta \bu_\delta\|_{L^{\frac{2\gamma}{\gamma+1}}(\Omega_w(t)} \left\|  \partial_t \mathcal B \left(b_\delta^{\theta_2} - [b_\delta^{\theta_2}]_{\Omega_w(t)}  \right) \right\|_{L^{\frac{2\gamma}{\gamma-1}}(\Omega_w(t))} .
\end{align*}

We have (using properties of Bogovski\u{\i}'s operator),
\begin{align*}
&\left\|\partial_t\mathcal{B}\left(b_\delta^{\theta_2}-[b_{\delta}^{\theta_2}]_{\Omega_w(t)}\right)\right\|_{L^{\frac{2\gamma}{\gamma-1}}(\Omega_w(t))} \\
& \leq \left\| \mathcal B \left( \rmdiv (b^{\theta_2}_\delta \bu_\delta) \right)  \right\|_{L^{\frac{2\gamma}{\gamma-1}}(\Omega_w(t))} \bigg\|\mathcal B \Big[(\theta_2-1)b^{\theta_2}_\delta \rmdiv \bu_\delta -\frac{1}{|\Omega_w(t)|} \int_{\Omega_w(t)} (\theta_2-1)b^{\theta_2}_\delta \rmdiv \bu_\delta  \Big] \bigg\|_{L^{\frac{2\gamma}{\gamma-1}}(\Omega_w(t)) }  \\
& \leq C\|b^{\theta_2}\bu_{\delta}\|_{L^{\frac{2\gamma}{\gamma-1}}(\Omega_w(t))} 
\bigg\|\mathcal B \Big[(\theta_2-1)b^{\theta_2}_\delta \rmdiv \bu_\delta -\frac{1}{|\Omega_w(t)|} \int_{\Omega_w(t)} (\theta_2-1)b^{\theta_2}_\delta \rmdiv \bu_\delta  \Big] \bigg\|_{W^{1,\frac{2\gamma}{2\gamma-1}}(\Omega_w(t))} \\ 
&\leq C\|b^{\theta_2}\bu_{\delta}\|_{L^{\frac{2\gamma}{\gamma-1}}(\Omega_w(t) )} 
+ C \|b^{\theta_2}_\delta \rmdiv \bu_\delta \|_{L^{\frac{2\gamma}{2\gamma-1}}(\Omega_w(t)) }
\\
&\leq C\|b^{\theta_2}_{\delta}\|_{L^{\frac{2\gamma p}{p(\gamma-1)-2\gamma}}(\Omega_w(t) )}\|\bu_{\delta}\|_{L^p(\Omega_w(t))}+
C\|b^{\theta_2}_{\delta}\|_{L^{\frac{2\gamma}{\gamma-1}}(\Omega_w(t) )}\|\mathrm{div}\bu_{\delta}\|_{L^2(\Omega_w(t)) } , 
\end{align*}
provided  $p > \frac{2\gamma}{\gamma-1}$, where we have used the embedding $W^{1,\frac{2\gamma}{2\gamma-1}}(\Omega_w(t)) \hookrightarrow L^{\frac{2\gamma}{\gamma-1}}(\Omega_w(t))$.

Then, using H\"older's inequality, \eqref{uniform-bound-delta}$_4$, \eqref{uni-bound-delta-u} and \eqref{uni-bound-delta-rho-u}, we get 
\begin{align*}
|\mathcal{I}_2|\leq & C\|\tilde{\psi}\|_{L^\infty(0,T)}\|\varrho_\delta\bu_\delta\|_{L^\infty(0,T;L^{\frac{2\gamma}{\gamma+1}}(\Omega_w(t))}
\bigg(\|b^{\theta_2}_{\delta}\|_{L^\infty(0,T;L^{\frac{2\gamma p}{p(\gamma-1)-2\gamma}}(\Omega_w(t) )) }\|\bu_{\delta}\|_{L^2(0,T; L^p(\Omega_w(t)))} \\
&+\|b^{\theta_2}_{\delta}\|_{L^\infty(0,T;L^{\frac{2\gamma}{\gamma-1}}(\Omega_w(t) ))}\|\bu_{\delta} \|_{L^2(0,T; W^{1,2}(\Omega_w(t)))}\bigg) \leq C \|\tilde{\psi}\|_{L^\infty(0,T)} , 
\end{align*}
provided that 
\begin{align*}
&\theta_2\frac{2\gamma p}{p(\gamma-1)-2\gamma}\leq 2,  \text{ and }  \theta_2\frac{2\gamma}{\gamma-1}\leq 2,\\
 \text{i.e., } & \theta_2\leq \frac{\gamma-1}{\gamma}-\frac{2}{p},   \text{ and } \theta_2\leq \frac{\gamma-1}{\gamma}.
\end{align*}
Thus $\mathcal{I}_2$ is bounded uniformly in $\delta$ if $\theta_2< 1- \frac{1}{\gamma}-\frac{2}{p}$ on condition that $p>2\gamma/(\gamma-1)$.

\vspace*{.2cm}

(iii) Let us focus on the term $\mathcal I_3$. Using H\"older's inequality we have   %
\begin{align*}
|\mathcal I_3| & =\left|\int_0^T \partial_t \tilde \psi \int_{\Omega_w(t)} \varrho_\delta \bu_{\delta}     \cdot \mathcal B \left(b^{\theta_2}_{\delta} - [b^{\theta_2}_\delta]_{\Omega_w(t)}  \right)     \right| \\
 & \leq \| \partial_t \tilde \psi\|_{L^\infty(0,T)} \int_0^T \| \varrho_\delta \bu_\delta \|_{L^{\frac{2\gamma}{\gamma+1}}(\Omega_w(t))} \| \mathcal B \left(b^{\theta_2}_{\delta} - [b^{\theta_2}_\delta]_{\Omega_w(t)}  \right) \|_{L^{\frac{2\gamma}{\gamma-1}}(\Omega_w(t)) } \\
 & \leq C \| \varrho_\delta \bu_\delta \|_{L^\infty(0,T; L^{\frac{2\gamma}{\gamma+1}}(\Omega_w(t)))} \int_0^T \| \mathcal B \left(b^{\theta_2}_{\delta} - [b^{\theta_2}_\delta]_{\Omega_w(t)}  \right) \|_{W^{1,\frac{2\gamma}{2\gamma-1}}(\Omega_w(t)) } \\
 & \leq C \| \varrho_\delta \bu_\delta \|_{L^\infty(0,T; L^{\frac{2\gamma}{\gamma+1}}(\Omega_w(t)))} \|b^{\theta_2}_\delta \|_{L^\infty(0,T; L^{\frac{2\gamma}{2\gamma-1}}(\Omega_w(t) )} \leq C,
\end{align*}
where we have used the bound \eqref{uni-bound-delta-rho-u} and \eqref{uniform-bound-delta}$_3$ on condition that 
$\theta_2\leq 2-\frac{1}{\gamma}$.


\vspace*{.2cm}
(iv) Under condition that $p>\frac{2\gamma}{\gamma-1}$, and $\eqref{uniform-bound-delta}_4$, $\eqref{uni-bound-delta-u}$, and the properties  of  Bogovski\u{\i}'s operator,  the term $\mathcal J_4$ can be estimated as 
\begin{align*}
|\mathcal I_4| &= \left| \int_0^T \tilde \psi \int_{\Omega_w(t)} \varrho_\delta \bu_\delta \otimes \bu_\delta : \nabla \mathcal B \left( b_\delta^{\theta_2} -[b^{\theta_2}_\delta]_{\Omega_w(t)}  \right)   \right| \\
& \leq C \int_0^T |\tilde \psi| \|\varrho_\delta\|_{L^\gamma(\Omega_w(t))} \|\bu_\delta\|^2_{L^p(\Omega_w(t))} \left\|\nabla \mathcal B \left( b_\delta^{\theta_2} -[b^{\theta_2}_\delta]_{\Omega_w(t)}  \right) \right\|_{L^{ \frac{p\gamma}{p(\gamma-1)-2\gamma}}(\Omega_w(t)) }\\
& \leq C \|\tilde \psi\|_{L^\infty(0,T)} \|\varrho_\delta\|_{L^\infty(0,T; L^{\gamma}(\Omega_w(t)))} \|\bu_\delta\|^2_{L^2(0,T;  L^p(\Omega-w(t)))} \|b^{\theta_2}_{\delta}\|_{L^\infty(0,T;  L^{ \frac{p\gamma}{p(\gamma-1)-2\gamma}}(\Omega_w(t)))    } \\
& \leq C, 
\end{align*} 
provided we consider 
\begin{align*}
\theta_2 \frac{p\gamma}{p(\gamma-1)-2\gamma} \leq  2  
\ \text{ i.e., } \ \theta_2 \leq 2\left(\frac{\gamma-1}{\gamma} -\frac{2}{p}\right). 
\end{align*}

\vspace*{.2cm}
(v) Finally, assuming  $\theta_2\leq \frac{1}{2}$, using H\"older's inequality and $\eqref{uniform-bound-delta}_4$,  we get 
\begin{align*}
|\mathcal{I}_5|  &= \left|\int_0^T \tilde \psi \int_{\Omega_w(t)} \bbS(\vartheta_\delta, \nabla \bu_\delta): \nabla \mathcal B \left(b^{\theta_2}_\delta - [b^{\theta_2}_\delta]_{\Omega_w(t)}  \right)   \right| \\
&\leq   C\|\tilde{\psi}\|_{L^\infty(0,T)} 
\int_0^T \int_{\Omega_w(t)} \|\bbS(\vartheta_\delta, \nabla \bu_\delta)\|_{L^{\frac{4}{3}}(\Omega_w(t))} \|b_\delta^{\theta_2}\|_{L^4(\Omega_w(t))} \\
& \leq C\|\tilde{\psi}\|_{L^\infty(0,T)} \|b_\delta^{\theta_2}\|_{L^4(\Omega_w(t))} \int_0^T \| \bbS(\vartheta_\delta, \nabla \bu_\delta)\|_{L^{\frac{4}{3}}(\Omega_w(t))} \\ 
&\leq C,
\end{align*}
where we have used the following fact:
\begin{align*}
\int_0^T \| \bbS(\vartheta_\delta, \nabla \bu_\delta)\|_{L^{\frac{4}{3}}(\Omega_w(t))} 
&\leq C \int_0^T \left( \int_{\Omega_w(t)} \bigg| \frac{1}{\sqrt{\vartheta_\delta}} \sqrt{\bbS(\vartheta_\delta, \nabla \bu_\delta) : \nabla \bu_\delta} \bigg|^{\frac{4}{3}} |\vartheta_\delta|^{\frac{4}{3}}   
\right)^{\frac{3}{4}}\\
& \leq C \int_0^T \left(\int_{\Omega_{w}(t)} \frac{1}{\vartheta_\delta} \left|\bbS(\vartheta_\delta, \nabla \bu_\delta): \nabla \bu_\delta \right| \right)^{\frac{1}{2}} \left(\int_{\Omega_w(t)} \vartheta^4_\delta\right)^{\frac{1}{4}} \\
& \leq C \left( \int_0^T \int_{\Omega_w(t)} \frac{1}{\vartheta_\delta} \left|\bbS(\vartheta_\delta, \nabla \bu_\delta): \nabla \bu_\delta\right| \right)^{\frac{1}{2}} \|\vartheta_\delta\|_{L^\infty(0,T; L^4(\Omega_w(t))) } \\
& \leq C ,
\end{align*}
thanks to $\eqref{uniform-bound-delta}_4$ and $\eqref{uniform-bound-delta}_6$.


Combining the above estimates, and choosing $\theta_2$ such that $\theta_2\leq \frac{\gamma-1}{\gamma} - \frac{2}{p}$ for some $p> \frac{2\gamma}{\gamma-1}$ and $\theta_2\leq \frac{1}{2}$, we obtain 
\begin{align}\label{extra-magnetic}
\int_{0}^T\int_{\Omega_w(t)}\left( p(\varrho_\delta, \vartheta_\delta) +\delta(\varrho_{\delta}+ b_\delta)^\beta + \frac{1}{2}b^2_\delta \right) b^{\theta_2}_\delta \leq C   .
\end{align}

Thus, by means of \eqref{extra-density} and \eqref{extra-magnetic}, we conclude  the proof.
\end{proof}

\begin{remark}\label{Remark-gamma-condition}

Observe that, to obtain $\varrho \in L^2(0,T; L^2(\Omega_w(t)))$ for the validation of renormalized solutions of the continuity equation, and later on the pointwise convergence of the densities $\{\varrho_\delta\}_{\delta}$ in $(0,T)\times \Omega_w(t)$  (see,   in particular the limit \eqref{auxiliary-T_k-T-rho}),  one needs the condition that
\begin{align*}
\gamma+\theta_1 > 2 , \text{ which implies } \gamma >  2-\theta_1 \geq  2 - \frac{\gamma-1}{2} + \frac{\gamma}{p}  .
\end{align*}
Since, one may consider $p>\frac{2\gamma}{\gamma-1}$ arbitrarily large, so it is sufficient to consider the condition 
\begin{align*}
\gamma > \frac53
\end{align*}
in the whole analysis to fulfill the required criterion.
\end{remark}

\vspace*{.4cm}

In the next two subsections, we shall concentrate on the strong convergence of fluid's temperature $\vartheta_\delta$ and density $\varrho_{\delta}$. We omit other convergence results which are similar to the previous layer's convergence. 

We begin with the following: using $\eqref{Delta-t-uni-bound}$ and structural assumptions $\eqref{state-equ}$, $\eqref{ent-sta}$, we get 
\begin{align}
\|\varrho_{\delta} s(\varrho_{\delta},\vartheta_{\delta})\|_{L^q(Q_T^w)}+\|\varrho_{\delta} s(\varrho_{\delta},\vartheta_{\delta})\bu_{\delta}\|_{L^s(Q_T^w)}+\left\|\frac{\kappa{(\vartheta_{\delta})}\nabla\vartheta_{\delta}}{\vartheta_{\delta}}\right\|_{L^r(Q_T^w)}\leq C, 
\end{align}
for some $q, s, r>1$.

\subsection{Strong convergence of the fluid temperature} 
Similar to Section \ref{Section-strong-conv-temp-Delta-t} (in the layer of convergence w.r.t. $\Delta t$),  we also set 
\begin{align}\label{Z-G-delta}
&{\bf{Z}}_{\delta}:=\bigg(\varrho_{\delta}s(\varrho_{\delta}, \vartheta_{\delta}), \hspace{0.1cm}\varrho_{\delta}s(\varrho_{\delta}, \vartheta_{\delta}) \bu_{\delta}-\frac{\kappa(\vartheta_{\delta})\nabla\vartheta_{\delta}}{\vartheta_{\delta}}\bigg),\notag \\
&{{\bf{G}}_{\delta}}:=(G(\vartheta_{\delta}),0,0), 
\end{align}
where $G(\cdot)$ is a bounded and Lipschitz function 
on $(0,\infty)$.
From the uniform estimates (w.r.t. $\delta$) in the previous subsections, one can deduce that 
$\mathrm {Curl}_{t,x}{\bf{G}}_{\delta}$ is precompact in $W^{-1, q}(Q_T^w)$ for some $q>1$.  

According to the $\eqref{entropy-ine-delta}$ and $\psi|_{\Gamma_T}=\varphi$, it holds that 
\begin{align*} 
\int_{Q_T^w}\mathrm{Div}_{t,x}{\bf{Z}}_{\delta}\psi&=-\langle\mathcal{D}_{\delta};\psi\rangle_{[\mathcal{M};C](Q_T^w)}-\int_{\Gamma_T}(\theta_t\varphi+\Delta\theta\varphi-\Delta w\varphi_t)\notag \\
&\leq C\|\psi\|_{W^{k,l}(Q_T^w)}, 
\end{align*}
for some $k>1$ and $l>3$. 
Then we get the compactness of the quantity $\mathrm{Div}_{t,x}{\bf{Z}}_{\delta}$, namely
\begin{align*}
    \mathrm{Div}_{t,x}{\bf{Z}}_{\delta}\in W^{-k, l^*}(Q_T^w)\Subset W^{-1, l'}(Q_T^w), \hspace{0.3cm}\mathrm{with}\hspace{0.3cm}  l'\in (1,\frac{3}{2}). 
\end{align*}
Furthermore, the Div-Curl lemma  (see \cite{Ta-79} and \cite[Theorem 10.12]{FN-09}) leads to
\begin{align}\label{varrho-s-G-delta}
\overline{\varrho s(\varrho,\vartheta)G(\varrho)}=\overline{\varrho s(\varrho,\vartheta)}\hspace{0.1cm} \overline{G(\varrho)}, \hspace{0.3cm} \mathrm{a.e.}\hspace{0.1cm}\mathrm{in}\hspace{0.1cm} (0,T)\times \Omega_w(t). 
\end{align}

In order to get the strong convergence of $\vartheta$, we again use the parameterized Young measures theory to derive the following fact:
\begin{align}\label{varrho-s-M-G-delta}
\overline{\varrho s_M(\varrho,\vartheta)G(\varrho)}\geq \overline{\varrho s_M(\varrho,\vartheta)}\overline{G(\varrho)}, \hspace{0.2cm} \overline{\vartheta^3G(\vartheta)}\geq \overline{\vartheta^3}\hspace{0.1cm}\overline{G(\vartheta)}, \hspace{0.3cm} \mathrm{a.e.}\hspace{0.1cm}\mathrm{in}\hspace{0.1cm} (0,T)\times\Omega_w(t).
\end{align}

Thus, we use $\eqref{varrho-s-G-delta}$ and $\eqref{varrho-s-M-G-delta}$ to get 
\begin{align*} 
\overline{\vartheta^4}=\overline{\vartheta^3}\vartheta,
\end{align*}
together with the monotonicity of $\vartheta^4$, it holds that 
\begin{align}\label{convergence-vartheta-delta}
\vartheta_{\delta}\to \vartheta,\hspace{0.3cm} \mathrm{a.e.}\hspace{0.1cm} \mathrm{in}\hspace{0.1cm} (0,T)\times\Omega_w(t).
\end{align}

\subsection{Strong convergence of the fluid density and magnetic field}

Setting $L_k(r)$ as the truncation of $r\log r$, where 
\begin{equation*} 
    L_k(r)=\left\{\begin{aligned}
    &r \log r, \hspace{0.5cm} \mathrm{for} \hspace{0.1cm}0\leq r\leq k,\\
   & r \log k+r\int_1^r\frac{T_k(z)}{z^2},\hspace{0.5cm} \mathrm{for} \hspace{0.1cm}r>k, 
    \end{aligned}
    \right.
\end{equation*}
and 
\begin{align}\label{T-k-z}
 T_k(z)=kT\left(\frac{z}{k}\right),
\end{align}
where 

\begin{equation}\label{T-z}
  T(z)=\left\{\begin{split}
& z \hspace{0.5cm}\mathrm{for}\hspace{0.1cm} z\in [0,1)\\
&\textnormal{concave} \hspace{0.5cm}\mathrm{for}\hspace{0.1cm}  z\in [1,3),\\
&2 \hspace{0.5cm}\mathrm{for}\hspace{0.1cm}  z\geq 3.
\end{split}
\right.
\end{equation}

\begin{lemma}
  Let $(\varrho_{\delta}, b_{\delta})$ be solution derived in Section \ref{Section-limit-system-other parameters} and $(\varrho, b)$ is the corresponding limit w.r.t. $\delta\to 0$. Then it has 
\begin{align}
&\overline{T_k(\varrho)}\overline{\bigg(\varrho^{\gamma}+\frac{1}{2}b^2\bigg)}\leq \overline{T_k(\varrho)\bigg(\varrho^{\gamma}+\frac{1}{2}b^2\bigg)},\label{varrho-press}\\
&
\overline{T_k(b)}\overline{\bigg(\varrho^{\gamma}+\frac{1}{2}b^2\bigg)}\leq \overline{T_k(b)\bigg(\varrho^{\gamma}+\frac{1}{2}b^2\bigg)},
\quad \mathrm{a.e.}\hspace{0.1cm}\mathrm{in}\hspace{0.1cm} (0,T)\times\Omega_w(t).\label{b-press}
\end{align}
    
\end{lemma}

\begin{proof}
We define 
\begin{align*}
\mathsf{d}_{\delta}=\varrho_{\delta }+b_{\delta}
, \quad \mathsf{d} = \varrho + b 
\end{align*}
and 
\begin{align*}
(\mathcal{R}_{\varrho_{\delta}},\mathcal{R}_{b_{\delta}})=\left(\frac{\varrho_{\delta}}{\mathsf{d}_{\delta}}, \frac{b_{\delta}}{\mathsf{d}_{\delta}}\right) \ \text{ if } \mathsf{d}_{\delta} \neq 0,  \quad (\mathcal{R}_{\varrho},\mathcal{R}_{b})=\left(\frac{\varrho}{\mathsf{d}}, \frac{b}{\mathsf{d}}\right) \ \text{ if } \mathsf{d} \neq 0 . 
\end{align*} 
Thus, it has 
$0\leq \mathcal{R}_{\varrho_{\delta}},\mathcal{R}_{b_{\delta}}, \mathcal{R}_{\varrho},\mathcal{R}_{b}\leq 1$. 

By virtue of Lemma \ref{strong-con}, we have 
    \begin{align*} 
    \varrho_{\delta}-\mathcal{R}_{\varrho}\mathsf{d}_{\delta}\rightarrow 0, \quad b_{\delta}-\mathcal{R}_{b}\mathsf{d}_{\delta}\rightarrow 0,\quad \mathrm{a s}\hspace{0.1cm} \delta \rightarrow 0, \quad \mathrm{a.e.}\hspace{0.1cm}\mathrm{in}\hspace{0.1cm} (0,T)\times\Omega_w(t). 
    \end{align*}
Moreover, Egorov's theorem implies that for sufficient small positive constant $\sigma$, there exists $\widetilde{Q_T^w}\subset Q_T^w$ such that $|Q_T^w\setminus\widetilde{Q_T^w}|\leq \sigma$ and it has 
    \begin{align}\label{uniform-est}
    \varrho_{\delta}-\mathcal{R}_{\varrho}\mathsf{d}_{\delta}\rightarrow 0, \quad b_{\delta}-\mathcal{R}_{b}\mathsf{d}_{\delta}\rightarrow 0,\quad \mathrm{a s}\hspace{0.1cm} \delta \rightarrow 0, \quad \mathrm{uniformly}\hspace{0.1cm}\mathrm{in}\hspace{0.1cm} \widetilde{Q_T^w}. 
    \end{align}
    
In view of $\eqref{uniform-est}$, it reveals that there exists a positive constant $\delta_0$ such that
    \begin{align}\label{partial-bound}
\mathcal{R}_{\varrho}\mathsf{d}_{\delta}\leq \varrho_{\delta}+1, \quad \mathcal{R}_{b}\mathsf{d}_{\delta}\leq b_{\delta}+1, 
    \end{align}
for all $\delta\leq \delta_0$ and $(t,x)\in \widetilde{Q_T^w}$. Thus, we infer from Lemma $\ref{high-int-press}$ that $\mathcal{R}_{\varrho}\mathsf{d}_{\delta}\in L^{\gamma+\theta_1}(\widetilde{Q_T^w})$ and $\mathcal{R}_{b}\mathsf{d}_{\delta}\in L^{2+\theta_2}(\widetilde{Q_T^w})$.

First, we perform the following decomposition
    \begin{align}\label{varrho-pressure}
&\lim_{\delta\rightarrow 0}\int_{Q_T^w}\Phi T_k(\varrho_\delta)\bigg((\varrho_{\delta})^\gamma+\frac{1}{2}(b_{\delta})^2\bigg)\notag\\
&=\lim_{\delta\rightarrow 0}\int_{\widetilde{Q_T^w}}\Phi T_k(\varrho_\delta)\bigg((\varrho_{\delta})^\gamma+\frac{1}{2}(b_{\delta})^2\bigg)+\lim_{\delta\rightarrow 0}\int_{Q_T^w\setminus\widetilde{Q_T^w}}\Phi T_k(\varrho_\delta)\bigg((\varrho_{\delta})^\gamma+\frac{1}{2}(b_{\delta})^2\bigg), 
   \end{align}
for all non-negative $\Phi\in C(\overline{Q_T^w})$.

Then, we find 
 \begin{align*}
&\lim_{\delta\rightarrow 0}\int_{\widetilde{Q_T^w}}\Phi T_k(\varrho_\delta)\bigg((\varrho_{\delta})^\gamma+\frac{1}{2}(b_{\delta})^2\bigg)\\
&=\lim_{\delta\rightarrow 0}\int_{\widetilde{Q_T^w}}\Phi T_k(\mathcal{R}_{\varrho}\mathsf{d}_\delta)\bigg((\mathcal{R}_{\varrho}\mathsf{d}_\delta)^\gamma+\frac{1}{2}(\mathcal{R}_{b}\mathsf{d}_\delta)^2\bigg)\\
&\hspace{0.5cm}+\lim_{\delta\rightarrow 0}\int_{\widetilde{Q_T^w}}\Phi (T_k(\varrho_{\delta})-T_k(\mathcal{R}_{\varrho}\mathsf{d}_\delta))\bigg((\mathcal{R}_{\varrho}\mathsf{d}_\delta)^\gamma+\frac{1}{2}(\mathcal{R}_{b}\mathsf{d}_\delta)^2\bigg)\\
&\hspace{0.5cm}+\lim_{\delta\rightarrow 0}\int_{\widetilde{Q_T^w}}\Phi T_k(\varrho_{\delta})\bigg[\bigg((\varrho_{\delta})^\gamma+\frac{1}{2}(b_{\delta})^2\bigg)-\bigg((\mathcal{R}_{\varrho}\mathsf{d}_\delta)^\gamma+\frac{1}{2}(\mathcal{R}_{b}\mathsf{d}_\delta)^2\bigg)\bigg]\\
&:=\sum_{i=1}^3\mathcal{K}_{i}.
  \end{align*}
Based on $\eqref{partial-bound}$, it has $(\mathcal{R}_{\varrho}\mathsf{d}_\delta)^\gamma+\frac12(\mathcal{R}_{d}\mathsf{d}_\delta)^2\in L^{\theta_{min}}(\widetilde{Q_T^w})$ for $\theta_{min}=\min\{\frac{\gamma+\theta_1}{\gamma},\frac{2+\theta_2}{2}\}$. 
Moreover, by virtue of $\eqref{uniform-est}$, and the continuity of $z\mapsto T_k(z)$, letting $\delta\rightarrow 0$, we obtain 
 \begin{align*}
\mathcal{K}_2\rightarrow 0.
  \end{align*}
The mean value theorem gives 
\begin{align*}
&\bigg[\bigg(({\varrho_{\delta}})^\gamma+\frac{1}{2}(b_{\delta})^2\bigg)-\bigg((\mathcal{R}_{\varrho}\mathsf{d}_\delta)^\gamma+\frac{1}{2}(\mathcal{R}_{b}\mathsf{d}_\delta)^2\bigg)\bigg]
\\
&\leq C(\varrho_{\delta}+\mathcal{R}_{\varrho}\mathsf{d}_{\delta})^{\gamma-1}
|{\varrho_{\delta}}-\mathcal{R}_{\varrho}\mathsf{d}_{\delta}|
+C({b_{\delta}}+\mathcal{R}_{b}\mathsf{d}_{\delta})|{b_{\delta}} -\mathcal{R}_{b}\mathsf{d}_{\delta}|.
  \end{align*}
Therefore, using $\eqref{uniform-est}$, $\eqref{partial-bound}$, and Lemma $\ref{high-int-press}$, we derive that 
\begin{align*}
|\mathcal{K}_3|
&\leq 
C\lim_{\delta \to 0}\int_{\widetilde{Q^w_T}} \Big[(\varrho_{\delta}+\mathcal{R}_{\varrho}\mathsf{d}_{\delta})^{\gamma-1}
|{\varrho_{\delta}}-\mathcal{R}_{\varrho}\mathsf{d}_{\delta}|
+C({b_{\delta}}+\mathcal{R}_{b}\mathsf{d}_{\delta})|{b_{\delta}} -\mathcal{R}_{b}\mathsf{d}_{\delta}| \Big] =0 .
\end{align*}

Thus, we get
\begin{align}\label{inequality-press}
&\lim_{\delta\rightarrow 0}\int_{\widetilde{Q_T^w}}\Phi T_k(\varrho_\delta)\bigg((\varrho_{\delta})^\gamma+\frac{1}{2}(b_{\delta})^2\bigg)=\int_{\widetilde{Q_T^w}}\Phi \overline{\overline{T_k(\mathcal{R}_{\varrho}\mathsf{d})\bigg((\mathcal{R}_{\varrho}\mathsf{d})^\gamma+\frac{1}{2}(\mathcal{R}_{b}\mathsf{d})^2\bigg)}}\notag\\
&\geq \int_{\widetilde{Q_T^w}}\Phi \overline{\overline{T_k(\mathcal{R}_{\varrho}\mathsf{d})}}\overline{\overline{\bigg((\mathcal{R}_{\varrho}\mathsf{d})^\gamma+\frac{1}{2}(\mathcal{R}_{b}\mathsf{d})^2\bigg)}}, 
  \end{align}
where we have used the facts that $z\mapsto T_k(\mathcal{R}_{\varrho}z)$ and $z\mapsto (\mathcal{R}_{\varrho}z)^\gamma+\frac{1}{2}(\mathcal{R}_{b}z)^2$ are non-decreasing functions, and the notation $\overline{\overline{g(\cdot)}}$ to represent the weak limit of $g(\cdot)$ with respect to  $\delta \rightarrow 0$.

In view of $\eqref{uniform-est}$, the continuous function $z\mapsto T_k(z)$, and $\overline{\overline{(\mathcal{R}_{\varrho}\mathsf{d})^\gamma+\frac{1}{2}(\mathcal{R}_{d}\mathsf{d})^2}}\in L^{\theta_{min}}(\widetilde{Q_T^w})$ for $\theta_{min}=\min\{\frac{\gamma+\theta_1}{\gamma},\frac{2+\theta_2}{2}\}$, we easily derive that 
\begin{align*}
\lim_{\delta\rightarrow 0}\int_{\widetilde{Q_T^w}}\Phi \bigg(T_k(\mathcal{R}_\varrho\mathsf{d}_{\delta})-T_k(\varrho_{\delta})\bigg)\overline{\overline{\bigg((\mathcal{R}_{\varrho}\mathsf{d})^\gamma+\frac{1}{2}(\mathcal{R}_{b}\mathsf{d})^2\bigg)}} = 0.
  \end{align*}

Then it has 
\begin{align*}
&\int_{\widetilde{Q_T^w}}\Phi \overline{\overline{T_k(\mathcal{R}_{\varrho}\mathsf{d})}}\overline{\overline{\bigg((\mathcal{R}_{\varrho}\mathsf{d})^\gamma+\frac{1}{2}(\mathcal{R}_{b}\mathsf{d})^2\bigg)}}\\
&=\lim_{\delta\rightarrow 0}\int_{\widetilde{Q_T^w}}\Phi T_k(\varrho_{\delta})\overline{\overline{\bigg((\mathcal{R}_{\varrho}\mathsf{d})^\gamma+\frac{1}{2}(\mathcal{R}_{b}\mathsf{d})^2\bigg)}}\\
&\hspace{0.5cm}+\lim_{\delta\rightarrow 0}\int_{\widetilde{Q_T^w}}\Phi \bigg(T_k(\mathcal{R}_\varrho\mathsf{d}_{\delta})-T_k(\varrho_{\delta})\bigg)\overline{\overline{\bigg((\mathcal{R}_{\varrho}\mathsf{d})^\gamma+\frac{1}{2}(\mathcal{R}_{b}\mathsf{d})^2\bigg)}}\\
&=\int_{\widetilde{Q_T^w}}\Phi\overline{T_k(\varrho)}\overline{\overline{\bigg((\mathcal{R}_{\varrho}\mathsf{d})^\gamma+\frac{1}{2}(\mathcal{R}_{b}\mathsf{d})^2\bigg)}}.
  \end{align*}

Moreover, it holds 
\begin{align*}
&\int_{\widetilde{Q_T^w}}\Phi\overline{T_k(\varrho)}\overline{\overline{\bigg((\mathcal{R}_{\varrho}\mathsf{d})^\gamma+\frac{1}{2}(\mathcal{R}_{b}\mathsf{d})^2\bigg)}}\\
&=\lim_{\delta\rightarrow 0}\int_{\widetilde{Q_T^w}}\Phi\overline{T_k(\varrho)}\bigg((\varrho_{\delta})^\gamma+\frac{1}{2}(b_{\delta})^2\bigg)\\
&\hspace{0.5cm}+\lim_{\delta\rightarrow 0}\int_{\widetilde{Q_T^w}}\Phi\overline{T_k(\varrho)}\bigg[\bigg((\mathcal{R}_{\varrho}\mathsf{d}_{\delta})^\gamma+\frac{1}{2}(\mathcal{R}_{b}\mathsf{d}_{\delta})^2 \bigg)-\bigg((\varrho_{\delta})^\gamma+\frac{1}{2}(b_{\delta})^2\bigg)\bigg]\\
&=\int_{\widetilde{Q_T^w}}\Phi\overline{T_k(\varrho)}\overline{\bigg(\varrho^\gamma+\frac{1}{2}b^2\bigg)}.
  \end{align*}
  
According to the above estimates, we have 
\begin{align}\label{equality-press}
&\int_{\widetilde{Q_T^w}}\Phi \overline{\overline{T_k(\mathcal{R}_{\varrho}\mathsf{d})}}\overline{\overline{\bigg((\mathcal{R}_{\varrho}\mathsf{d})^\gamma+\frac{1}{2}(\mathcal{R}_{b}\mathsf{d})^2\bigg)}}=\int_{\widetilde{Q_T^w}}\Phi\overline{T_k(\varrho)}\overline{\bigg(\varrho^\gamma+\frac{1}{2}b^2\bigg)}.
  \end{align} 

Then, we put $\eqref{equality-press}$ and $\eqref{inequality-press}$ together to obtain 
\begin{align}\label{inequality-press-1}
&\lim_{\delta\rightarrow 0}\int_{\widetilde{Q_T^w}}\Phi T_k(\varrho_\delta)\bigg((\varrho_{\delta})^\gamma+\frac{1}{2}(b_{\delta})^2\bigg)\geq 
\int_{\widetilde{Q_T^w}}\Phi\overline{T_k(\varrho)}\overline{\bigg(\varrho^\gamma+\frac{1}{2}b^2\bigg)}.
  \end{align}

Due to the boundedness of $T_k(\varrho_{\delta})\bigg((\varrho_{\delta})^\gamma+\frac{1}{2}(b_{\delta})^2\bigg)\in L^{\theta_{min}}(Q_T^w)$ for $\theta_{min}=\min\{\frac{\gamma+\theta_1}{\gamma},\frac{2+\theta_2}{2}\}$, it has 
\begin{align}\label{inequality-press-2}
\lim_{\delta\rightarrow 0}\int_{Q_T^w\setminus\widetilde{Q_T^w}}\Phi T_k(\varrho_\delta)\bigg((\varrho_{\delta})^\gamma+\frac{1}{2}(b_{\delta})^2\bigg)=\int_{Q_T^w\setminus\widetilde{Q_T^w}}\Phi \overline{T_k(\varrho)\bigg(\varrho^\gamma+\frac{1}{2}b^2\bigg)}.
   \end{align}

Inserting $\eqref{inequality-press-1}$ and $\eqref{inequality-press-2}$ into $\eqref{varrho-pressure}$, we obtain
\begin{align*}
\int_{Q_T^w}\Phi \overline{T_k(\varrho)\bigg(\varrho^\gamma+\frac{1}{2}b^2\bigg)}&\geq \int_{\widetilde{Q_T^w}}\Phi \overline{T_k(\varrho)}\overline{\bigg(\varrho^\gamma+\frac{1}{2}b^2\bigg)}+\int_{Q_T^w\setminus\widetilde{Q_T^w}}\Phi \overline{T_k(\varrho)\bigg(\varrho^\gamma+\frac{1}{2}b^2\bigg)}\\
&\geq \int_{{Q_T^w}}\Phi \overline{T_k(\varrho)}\overline{\bigg(\varrho^\gamma+\frac{1}{2}b^2\bigg)}-\int_{Q_T^w\setminus\widetilde{Q_T^w}}\Phi \overline{T_k(\varrho)}\overline{\bigg(\varrho^\gamma+\frac{1}{2}b^2\bigg)}\\
&\hspace{0.5cm}+\int_{Q_T^w\setminus\widetilde{Q_T^w}}\Phi \overline{T_k(\varrho)\bigg(\varrho^\gamma+\frac{1}{2}b^2\bigg)}.
  \end{align*}

Recalling that $|Q_T^w\setminus\widetilde{Q_T^w}|\leq \sigma$, and letting $\sigma\rightarrow 0$, thus the last two terms on the right hand side of the above estimate go to zero. Eventually, we get 
\begin{align*}
\int_{Q_T^w}\Phi \overline{T_k(\varrho)\bigg(\varrho^\gamma+\frac{1}{2}b^2\bigg)}\geq 
\int_{{Q_T^w}}\Phi \overline{T_k(\varrho)}\overline{\bigg(\varrho^\gamma+\frac{1}{2}b^2\bigg)}.
\end{align*}

This completes the proof of $\eqref{varrho-press}$. The proof of $\eqref{b-press}$ can be derived in a similar way, and we omit it for brevity.
\end{proof}

Similarly to the previous discussion, we also construct the weak compactness identity for effective pressure $p_M(\varrho, \vartheta, b)-(\mu(\vartheta)+\eta(\vartheta))\mathrm{div}\bu$.
\begin{lemma}\label{Lemma-weak-compact-effective-id}
  For any $\psi\in C_c^\infty(Q_T^{w_{\delta}})$, it holds that 
  \begin{align}\label{Delta-t-EVF-3}
&\lim_{\delta\to 0}\int_{Q_T^{w_{\delta}}}(p_M(\varrho_{\delta}, \vartheta_{\delta}, b_{\delta})-(\mu(\vartheta_{\delta})+\eta(\vartheta_{\delta}))\mathrm{div}\bu_{\delta})(T_k(\varrho_{\delta})+T_k(b_{\delta}))\psi\notag \\
&=\int_{Q_T^{w_{\delta}}}(\overline{p_M(\varrho_, \vartheta, b)}-\overline{(\mu(\vartheta)+\eta(\vartheta))\mathrm{div}\bu})(\overline{T_k(\varrho)}+\overline{T_k(b)})\psi. 
\end{align}
\end{lemma}

Seeing that $(\varrho_{\delta}, \bu_{\delta})$,  $(\varrho, \bu)$ and $(b_{\delta}, \bu_{\delta})$,  $(b, \bu)$ satisfy the assumptions in Lemma \ref{Renormalized-equ}, so they solve the corresponding renormalized equations. For any $t\in (0,T]$, we have 
\begin{align}\label{L-k-varrho-con}
\int_{\Omega_{w_{\delta}}(t)}L_k(\varrho_{\delta})(t,\cdot)-\int_{\Omega_{w}(0)}L_k(\varrho_{0, \delta})(\cdot)=\int_0^t\int_{\Omega_{w_{\delta}}(s)}T_k(\varrho_{\delta})\mathrm{div}\bu_{\delta}.
\end{align}

As for magnetic field $b_{\delta}$, we also have  
\begin{align}\label{L-k-b-con}
\int_{\Omega_{w_{\delta}}(t)}L_k(b_{\delta})(t,\cdot)-\int_{\Omega_{w}(0)}L_k(b_{0, \delta})(\cdot)=\int_0^t\int_{\Omega_{w_{\delta}}(s)}T_k(b_{\delta})\mathrm{div}\bu_{\delta}.
\end{align}

\begin{align}\label{L-con-k}
\int_{\Omega_{w}(t)}L_k(\varrho)(t,\cdot)-\int_{\Omega_{w}(0)}L_k(\varrho_{0})(\cdot)=\int_0^t\int_{\Omega_{w}(s)}T_k(\varrho)\mathrm{div}\bu. 
\end{align}

Also, it holds 
\begin{align}\label{L-con-k-b}
\int_{\Omega_{w}(t)}L_k(b)(t,\cdot)-\int_{\Omega_{w}(0)} L_k(b_{0})(\cdot)=\int_0^t\int_{\Omega_{w}(s)}T_k(b)\mathrm{div}\bu. 
\end{align}

Subtracting $\eqref{L-con-k}$ from $\eqref{L-k-varrho-con}$, subtracting $\eqref{L-con-k-b}$ from $\eqref{L-k-b-con}$, and adding them together, then letting $\delta \rightarrow 0$, we get 
\begin{align}\label{difference-T-k}
    &\int_{\Omega_{w}(t)}[\overline{L_k(\varrho)}- L_k(\varrho)+\overline{L_k(b)}- L_k(b)](t,\cdot)\notag\\
    &=\int_0^t\int_{\Omega_{w}(s)}[T_k(\varrho)+T_k(b)]\mathrm{div}\bu -\int_0^t\int_{\Omega_{w}(s)}\overline{[T_k(\varrho)+T_k(b)]\mathrm{div}\bu}. 
\end{align}

It follows from $\eqref{Delta-t-EVF}$ and the strong convergence of $\vartheta_\delta$ in $\eqref{convergence-vartheta-delta}$ that 
\begin{align}\label{T-k-varrho-b-div-u}
&-\int_0^t\int_{\Omega_{w}(s)}\overline{[T_k(\varrho)+T_k(b)]\mathrm{div}\bu}\notag\\
&=-\lim_{\delta\rightarrow 0}\int_0^t\int_{\Omega_{w_{\delta}}(s)}[T_k(\varrho_{\delta})+T_k(b_{\delta})]\frac{1}{\mu(\vartheta_{\delta})+\eta(\vartheta_{\delta})}\bigg((\mu(\vartheta_{\delta})+\eta(\vartheta_{\delta}))\mathrm{div}\bu_{\delta}-p_M(\varrho_{\delta}, \vartheta_{\delta}, b_{\delta})\bigg)\notag\\
&\hspace{0.5cm}-\lim_{\delta\rightarrow 0}\int_0^t\int_{\Omega_{w_{\delta}}(s)}[T_k(\varrho_{\delta})+T_k(b_{\delta})]\frac{1}{\mu(\vartheta_{\delta})+\eta(\vartheta_{\delta})}p_M(\varrho_{\delta}, \vartheta_{\delta}, b_{\delta})\notag\\
&=-\int_0^t\int_{\Omega_{w}(s)}[\overline{T_k(\varrho)}+\overline{T_k(b)}]\frac{1}{\mu(\vartheta)+\eta(\vartheta)}\bigg((\mu(\vartheta)+\eta(\vartheta))\mathrm{div}\bu-\overline{p_M(\varrho, \vartheta, b)}\bigg)\notag\\
&\hspace{0.5cm}-\lim_{\delta\rightarrow 0}\int_0^t\int_{\Omega_{w_{\delta}}(s)}[T_k(\varrho_{\delta})+T_k(b_{\delta})]\frac{1}{\mu(\vartheta_{\delta})+\eta(\vartheta_{\delta})}p_M(\varrho_{\delta}, \vartheta_{\delta}, b_{\delta}).
\end{align}

Now, relations $\eqref{difference-T-k}$ and $\eqref{T-k-varrho-b-div-u}$ give rise to 
\begin{align}\label{difference-T-k-L}
    &\int_{\Omega_{w}(t)}[\overline{L_k(\varrho)}- L_k(\varrho)+\overline{L_k(b)}- L_k(b)](t,\cdot)\notag\\
    &=\int_0^t\int_{\Omega_{w}(s)}[\overline{T_k(\varrho)}+\overline{T_k(b)}]\frac{1}{\mu(\vartheta)+\eta(\vartheta)}\overline{p_M(\varrho, \vartheta, b)}\notag\\
    &\hspace{0.5cm}-\lim_{\delta\rightarrow 0}\int_0^t\int_{\Omega_{w_{\delta}}(s)}[T_k(\varrho_{\delta})+T_k(b_{\delta})]\frac{1}{\mu(\vartheta_{\delta})+\eta(\vartheta_{\delta})}p_M(\varrho_{\delta}, \vartheta_{\delta}, b_{\delta})\notag\\
    &\hspace{0.5cm}+\int_0^t\int_{\Omega_{w}(s)}[T_k(\varrho)-\overline{T_k(\varrho)}+T_k(b)-\overline{T_k(b)}]\mathrm{div}\bu, 
\end{align}
 together with $\eqref{varrho-press}$, $\eqref{b-press}$, and the strong convergence of $\vartheta_{\delta}$, which give 
\begin{align}\label{difference-T-k-L-2}
    &\int_{\Omega_{w}(t)}[\overline{L_k(\varrho)}- L_k(\varrho)+\overline{L_k(b)}- L_k(b)](t,\cdot)\notag\\
    &
\leq \int_0^t\int_{\Omega_{w}(s)}[T_k(\varrho)-\overline{T_k(\varrho)}+T_k(b)-\overline{T_k(b)}]\mathrm{div}\bu.
\end{align}

Because of $\eqref{T-k-z}$ and $\eqref{T-z}$, it has 
\begin{align*}
\|T_k(\varrho)-\overline{T_k(\varrho)}\|_{L^2(Q_T^w)}&\leq \liminf_{\delta\rightarrow 0}\|T_k(\varrho)-T_k(\varrho_{\delta})\|_{L^2(Q_T^{w_\delta})}\\
&\leq \liminf_{\delta\rightarrow 0}\|\varrho+\varrho_{\delta}\|_{L^{\gamma+\theta_1}(Q_T^{w_\delta})}\leq C,
\end{align*}
where we have used $\eqref{high-int-varrho-b-delta}$.

Then, direct calculations give that 
\begin{align}\label{T-k-varrho-div-u-estimate}
&\bigg|\int_0^t\int_{\Omega_{w}(s)}[T_k(\varrho)-\overline{T_k(\varrho)}]\mathrm{div}\bu\bigg|\notag\\
&\leq \|T_k(\varrho)-\overline{T_k(\varrho)}\|_{L^2(Q_T^w)}\|\mathrm{div}\bu\|_{L^2(Q_T^w\cap\{\varrho\geq k\})}\notag\\
&\hspace{0.5cm}+\|T_k(\varrho)-\overline{T_k(\varrho)}\|_{L^2(Q_T^w\cap\{\varrho\leq k\})}\|\mathrm{div}\bu\|_{L^2(Q_T^w)}\notag\\
&\leq C\|\mathrm{div}\bu\|_{L^2(Q_T^w\cap\{\varrho\geq k\})}+C\|T_k(\varrho)-\overline{T_k(\varrho)}\|_{L^2(Q_T^w\cap\{\varrho\leq k\})}.
\end{align}
Now, due to $\eqref{high-int-varrho-b-delta}$, the Lebesgue measure of $Q_T^w\cap\{\varrho\geq k\}$ goes to zero as $k \rightarrow \infty$. 
Moreover, we have
\begin{align}
&\|T_k(\varrho)-\overline{T_k(\varrho)}\|_{L^2(Q_T^w\cap\{\varrho\leq k\})}=\|\varrho-\overline{T_k(\varrho)}\|_{L^2(Q_T^w\cap\{\varrho\leq k\})}  \notag \\
&\leq \liminf_{\delta\rightarrow 0}\|\varrho_{\delta}-T_k(\varrho_{\delta})\|_{L^2(Q_T^{w_\delta})}=\liminf_{\delta\rightarrow 0}\|\varrho_{\delta}-T_k(\varrho_{\delta})\|_{L^2(Q_T^{w_\delta}\cap\{\varrho_{\delta}>k\})} \notag \\
&\leq 2\liminf_{\delta\rightarrow 0}\|\varrho_{\delta}\|_{L^2(Q_T^{w_\delta}\cap\{\varrho_{\delta}>k\})}\leq 2k^{1-\frac{\gamma+\theta_1}{2}}\liminf_{\delta\rightarrow 0}\|\varrho_{\delta}\|_{L^{\gamma+\theta_1}(Q_T^{w_\delta})}^{\gamma+\theta_1}
\notag \\
&\rightarrow 0,\quad  \mathrm{as} \hspace{0.2cm}k\rightarrow   \infty, \label{auxiliary-T_k-T-rho}
\end{align}
where we have used the fact that $\gamma+\theta_1>2$, and this is the place where it is crucial to assume $\gamma>\frac{5}{3}$ (see Remark \ref{Remark-gamma-condition}).

Thus, letting $k\rightarrow \infty$ in $\eqref{T-k-varrho-div-u-estimate}$, we obtain 
\begin{align}\label{T-k-varrho-div-u}
\lim_{k\rightarrow \infty}\int_0^t\int_{\Omega_{w}(s)}[T_k(\varrho)-\overline{T_k(\varrho)}]\mathrm{div}\bu=0.
\end{align}

Similarly, one has
\begin{align}\label{T-k-b-div-u}
\lim_{k\rightarrow \infty}\int_0^t\int_{\Omega_{w}(s)}[T_k(b)-\overline{T_k(b)}]\mathrm{div}\bu=0.
\end{align}

To the next, by plugging $\eqref{T-k-varrho-div-u}$ and $\eqref{T-k-b-div-u}$ into $\eqref{difference-T-k-L}$, we get 
\begin{align}\label{negative-L-varrho-b}
\lim_{k\rightarrow \infty}\int_{\Omega_{w}(t)}[\overline{L_k(\varrho)}- L_k(\varrho)+\overline{L_k(b)}- L_k(b)](t,\cdot)\leq 0.
\end{align}

Now, according to the definition, we know that 
\begin{align}
&\lim_{k\rightarrow \infty}\bigg(\|L_k(\varrho)-\varrho\log \varrho\|_{L^1(\Omega_{w}(t))}+\|L_k(b)-b\log b\|_{L^1(\Omega_{w}(t))}\bigg)=0, \ \text{ and }\label{blogb-equal-1}\\
&\lim_{k\rightarrow \infty}\bigg(\|\overline{L_k(\varrho)}-\overline{\varrho\log \varrho}\|_{L^1(\Omega_{w}(t))}+\|\overline{L_k(b)}-\overline{b\log b}\|_{L^1(\Omega_{w}(t))}\bigg)=0\label{blogb-equal-2}.
\end{align}
Then, using 
 the convexity of $z\mapsto z\log z$, it holds that  
\begin{align*}
0\leq \int_{\Omega_{w}(t)}\bigg(\overline{\varrho\log \varrho}-\varrho\log\varrho+\overline{b\log b}-b\log b\bigg)\leq 0,
\end{align*}
where we have used $\eqref{negative-L-varrho-b}$, $\eqref{blogb-equal-1}$, and $\eqref{blogb-equal-2}$. 

So, finally we obtain
\begin{align*} 
\varrho_{\delta}\to\varrho  \ \text{ and } \ b_{\delta}\to b, \hspace{0.2cm} \mathrm{a.e.}\hspace{0.1cm}\mathrm{in}\hspace{0.1cm} (0,T)\times\Omega_w(t).
\end{align*}

\subsection{Maximal interval of existence} 
It follows from $\eqref{T-max}$ that a minimal time could be determined. Then, we prolong the time until degeneracy occurs. For more details, interested readers are referred to \cite{CDEG-05, LR-14}.

\section{Concluding remarks}\label{Sec-Conclusion}

In this paper, we study a fluid-structure interaction system in a two-dimensional framework. A crucial step involves extending the initial time-varying domain $\Omega_w(t)$ to a larger, fixed domain $\mathsf{B}$. 
By establishing the existence of a weak solution within the fixed domain $\mathsf{B}$, and taking advantage of the    facts that $\varrho|_{\mathsf{B}\setminus \Omega_w(t)}=b|_{\mathsf{B}\setminus \Omega_w(t)}=0$ (starting with initial conditions $\varrho_0|_{\mathsf{B}\setminus \Omega_w(0)}=b_0|_{\mathsf{B}\setminus \Omega_w(0)}=0$), we successfully demonstrate the existence of a weak solution to the coupled system in $Q^w_T$. 
Actually in our case, the magnetic field acts solely in the vertical direction on the fluid, and thus the induction equation in  $\mathbb{R}^3$ (see \eqref{mag-eq})
\begin{align}\label{ind-three-dim-mag}
\mathbf B_t-\textbf{curl} (\bu \times \mathbf B) + \textbf{curl} (\nu \textbf{curl}\, \mathbf B) =0, \ \ \rmdiv \mathbf{B}=0 ,  
\end{align}
simplifies to
\begin{align*}
b_t+\mathrm{div}(b\bu)=0,
\end{align*}
in our two-dimensional setting, which basically represents a transport equation.   
But for the general situation  as \eqref{ind-three-dim-mag}, it is  not straightforward to prove $\mathbf{B}|_{\mathsf{B}\setminus \Omega_w(t)}=0$.  So, we decided to   consider this issue   in a future research work, in the context of MHD-structure interaction problem in $3$D-$2$D setting.

\appendix 

\section{Some auxiliary results}\label{Sec-Appendix}

We state generalized Korn-Poincar\'{e}  inequality on a domain with a Lipschitz boundary, which is due to the result in \cite[Theorem 11.20]{FN-09}.

\begin{lemma}[Generalized Korn-Poincar\'{e} inequality]\label{Korn-inequality}
Let $w$ be the displacement of the boundary satisfying $w\in H^2({\Gamma})$ and $\alpha_{\partial\Omega}<w<\beta_{\partial\Omega}$, and the domain $\Omega_w(t)$ is defined through the displacement $w$ as in $\eqref{moving-domain}$. Moreover, suppose that $M$, $L>0$, $\gamma >1$. Then there exists a positive constant $C$ only depending on $q$, $M$, $L$, $\gamma$, $\|w\|_{H^2(\Gamma)}$ such that 
\begin{align*} 
\|\bu\|^2_{ W^{1,q}(\Omega_w(t))}\leq C\left(\|\nabla\bu+\nabla^\top\bu-\rmdiv \bu \mathbb I_2\|^2_{L^q(\Omega_w(t))}+\int_{\Omega_w(t)}\varrho|\bu|\right), \ \ 1< q< \infty, 
\end{align*}
for any $\varrho$ and $\bu$ such that the right-hand side is finite and 
\begin{align*}
\varrho\geq 0\hspace{0.15cm} \mathrm{a.e.}\hspace{0.15cm} \mathrm{on} \hspace{0.15cm} \Omega_w(t), \hspace{0.3cm}\|\varrho\|_{L^\gamma(\Omega_w(t))}\leq L, \hspace{0.3cm}\int_{\Omega_w(t)}\varrho\geq M. 
\end{align*}
\end{lemma}


The renormalized weak solutions (introduced by DiPerna and Lions \cite{Lions-Diperna}) to the continuity equation $\eqref{weak-con}$ and the no-resistive magnetic equation $\eqref{mag-equ}$ in the time varying domain still holds true. To complete the proof, we refer the procedures outlined in \cite[Lemma 2.7]{KMN-24} and \cite[Lemma 2.5]{VWY-19}.

\begin{lemma}\label{Renormalized-equ}
Assume that the density $\varrho\in {L^\infty(0,T;L^\gamma(\Omega_w(t)))\cap L^2(0,T;L^2(\Omega_w(t)))}$, the magnetic field $b\in L^\infty(0,T;L^2(\Omega_w(t)))\cap {L^2(0,T;L^2(\Omega_w(t)))}$, and the velocity $\bu\in L^2(0,T;W^{1,2}(\mathsf{B}))$, where the domain $\mathsf{B}$ is given in Section \ref{Section-fixed-domain}. 
Moreover, they satisfy that 
\begin{align*}
  &\int_0^T\int_{\Omega_w(t)}\varrho(\phi_t+\bu\cdot\nabla\phi)=0,\\
    &\int_0^T\int_{\Omega_w(t)}b(\phi_t+\bu\cdot\nabla\phi)=0,
\end{align*}
for any $\phi\in C_c^\infty([0,T]\times\mathbb{R}^2)$. Then we derive 
\begin{align*}
    &\int_0^t\int_{\Omega_w(s)}  \mathcal{F}(\mathbf{r})(\phi_t+\bu\cdot\nabla\phi)-[\nabla_{\mathbf{r}}\mathcal{F}(\mathbf{r})\cdot\mathbf{r}-\mathcal{F}(\mathbf{r})]\mathrm{div}\bu\phi=\bigg(\int_{\Omega_w(s)}\mathcal{F}(\mathbf{r})\phi\bigg)\bigg|_{s=0}^{s=t}
\end{align*}
where $\mathcal{F}(\cdot):C^1([0,\infty)^2)\to \mathbb{R}$, $\nabla \mathcal{F}(\cdot)\in L^\infty(0,\infty)$, $\mathcal{F}(0)=0$ the function $\mathcal{F}(\mathbf{r})$, where $\mathbf{r}:=(\varrho, b)$, and for all $t\in (0,T]$, for all $\phi\in C^\infty([0,T]\times\mathbb{R}^2)$.

\end{lemma}

The following lemma ensures that the density and the magnetic field vanish outside the domain $\Omega_w(t)$, given that they are zero outside $\Omega_w(0)$ at the initial time. The proof closely follows Lemma 4.4 in \cite{KMN-24}, and for brevity, we omit it here.
\begin{lemma}\label{extend-zero}
Suppose that $\varrho \in L^\infty(0,T;L^{3}(\mathsf{B}))$, $b \in L^\infty(0,T;L^{3}(\mathsf{B}))$, $\bu \in L^2(0,T;W^{1,2}_0(\mathsf{B}))$ satisfy the weak formulations $\eqref{weak-con-delta}$ and $\eqref{mag-equ-delta}$, and assume that the displacement {$w\in W^{1,\infty}(0,T;L^2(\Gamma))\cap W^{1,2}((0,T)\times\Gamma)\cap L^\infty(0,T;H^{2}(\Gamma))$} and $\bu(x+w(t,\bmphi^{-1}(x))\bfn(x))=w_t(t,\bmphi^{-1}(x))\bfn(x)$ in $(0,T)\times\partial\Omega$ in the sense of trace. Then we have 
\begin{align*}
\varrho|_{\mathsf{B}\setminus \Omega_w(t)}=b|_{\mathsf{B}\setminus \Omega_w(t)}=0, \hspace{0.3cm}\mathrm{a.e.}\hspace{0.1cm} t\in (0,T). 
\end{align*}
\end{lemma}

Inspired by \cite[Lemma 2.1]{Wen-21} and \cite[Lemma 2.9]{KMN-24}, we employ similar procedures to establish the almost compactness property $\varrho$ and $b$ on the varying domain.
\begin{lemma}\label{strong-con}
Assume that the sequence $\{(w^i, \varrho^i, b^i, \bu^i)\}$ satisfies the following items:
  \begin{itemize}
  \item the assumption for  structure geometry given in Section \ref{Section-Geometry-ext} holds for each $w^i$;
\item $\varrho^i$ and $b^i$ are the solutions to continuity equation and no-resistive equation (respectively) in $Q^{w^i}_T$ prolonged by zero on $((0,T)\times\mathsf{B})\setminus Q^{w^i}_T$, and $\bu^i$ are the corresponding velocities in $(0,T)\times \mathsf{B}$;
\item the following estimates also hold:
\begin{align*}
&\|w^i\|_{L^\infty(0,T;W^{2,2}(\Gamma))}+\|(w^i)_t\|_{L^\infty(0,T;L^2(\Gamma))}+
\|\varrho^i\|_{L^\infty(0,T;L^\gamma(\Omega_{w^i}(t)))}\\
&+\|b^i\|_{L^\infty(0,T;L^2(\Omega_{w^i}(t)))}+\|\bu^i\|_{L^2(0,T;W^{1,2}(\mathsf{B}))} <  +\infty.
\end{align*}
\end{itemize}
Moreover, it holds that 
\begin{align*}
\lim_{i\rightarrow +\infty}\int_{\Omega_{w^i_0}}\frac{(\varrho^i_{0})^2}{\varrho^i_{0}+b^i_{0}}\leq \int_{\Omega_{w_0}}\frac{\varrho_0^2}{\varrho_0+b_0},
\end{align*}
and 
\begin{align*}
\lim_{i\rightarrow +\infty}\int_{\Omega_{w^i_0}}\frac{(b^i_{0})^2}{\varrho^i_{0}+b^i_{0}} \leq \int_{\Omega_{w_0}}\frac{b_0^2}{\varrho_0+b_0}.
\end{align*}

Then, up to a subsequence, we have  
\begin{align*}
&w^i \rightarrow w \hspace{0.5cm}\mathrm{in}\hspace{0.2cm} C^{\frac{1}{5}}([0,T];C^{1,\frac{1}{10}}(\Gamma)),\\
&\varrho^i\rightharpoonup \varrho \hspace{0.5cm}\mathrm{in}\hspace{0.2cm} C_{\mathrm{weak}}([0,T];L^{\gamma}(\mathsf{B})),\\
&b^i\rightharpoonup b \hspace{0.5cm}\mathrm{in}\hspace{0.2cm} C_{\mathrm{weak}}([0,T];L^2(\mathsf{B})),\\
&\bu^i \rightharpoonup \bu 
\hspace{0.5cm} \mathrm{weakly } \hspace{0.2cm}\mathrm{in} \hspace{0.2cm}L^2(0,T;W_0^{1,2}(\mathsf{B})). 
\end{align*}
In addition, $(\varrho, \bu)$ and $(b, \bu)$ solve the continuity equation and no-resistive magnetic equation in $(0,T)\times\mathsf{B}$, respectively, and 
\begin{align*}
\lim_{i\rightarrow +\infty}\int_0^T\int_{\mathsf{B}}(\varrho^i+b^i)|a^i-a|^p=0,
\end{align*}
for all $p\in[1,\infty)$ and $t\in[0,T]$, where $a^i=\frac{\varrho^i}{\varrho^i+b^i}$ if $\varrho^i+b^i\neq 0$ and $a=\frac{\varrho}{\varrho+b}$ if $\varrho+b\neq 0$, or $a^i=\frac{b^i}{\varrho^i+b^i}$ if $\varrho^i+b^i\neq 0$ and $a=\frac{b}{\varrho+b}$ if $\varrho+b\neq 0$. 
\end{lemma}

\bibliographystyle{siam} 

\bibliography{ref-MHD-structure}

\end{document}